\documentclass[11pt, reqno]{amsart}
\usepackage{paper_style}
\usepackage[paper]{std_math}

\usepackage{local}
\usepackage[mode=buildnew]{standalone}
	
\IfFileExists{./\foldername/abstract.tex}{}{}

\hypersetup{pdfauthor={Gaurav Aggarwal, Konstantin Andritsch}
	,pdftitle=""
	,urlcolor=blue
	,citecolor=red
	,linkcolor=blue
	,colorlinks=true 
}

\newcommandx{\gaadd}[2][2=]{{\color{purple}{\tiny [GA]} #1} \marginpar{\color{purple} [GA] #2}}

\begin{document}

\title{Generalised Gauss--Kuzmin Distribution for Klein sails in $\RR^3$}

\author[G. Aggarwal]{Gaurav Aggarwal}
\address{Institut für Mathematik, University of Zürich, Zürich, Switzerland}
\email{gaurav.aggarwal@math.uzh.ch}

\author[K. Andritsch]{Konstantin Andritsch}
\address{Department of Mathematics, ETH Zürich, Zürich, Switzerland}
\email{konstantin.andritsch@math.ethz.ch}

\subjclass[2020]{Primary: 11J70, Secondary: 37A17, 11H06, 52C05}
\keywords{Klein sails, Gauss–Kuzmin-type distribution, homogeneous dynamics, Diophantine approximation, equidistribution}
\thanks{G.~A. gratefully acknowledges support from the Swiss National Science Foundation (grant 200020-212617).\\
\hspace*{1.3em}K.~A. gratefully acknowledges support from the Swiss National Science Foundation (grant 10003145).}
\date{}

\begin{abstract}	
	The classical Gauss--Kuzmin distribution describes the asymptotic distribution of continued fraction digits. Geometrically, this may be interpreted as the equidistribution of faces of Klein sails in $\RR^2$.

In this paper, we establish a three-dimensional analogue of this phenomenon. We prove the equidistribution of local face structures in generic Klein sails in $\RR^3$, thereby obtaining a higher-dimensional generalization of the Gauss--Kuzmin distribution. In addition, we resolve several open questions posed by Karpenkov~\cite{Ka17}.

Our approach is based on homogeneous dynamics and is motivated by the work of Kontsevich and Suhov~\cite{KS99}. More precisely, we construct a cross-section for the diagonal flow on $\SL_3(\mathbb R)/\SL_3(\mathbb Z)$, such that visits to the cross-section encode the geometry of three-dimensional sails. The principal technical contribution is the proof that the associated cross-sectional measure is finite, which enables us to derive the limiting face statistics.

\end{abstract} 

\maketitle

    \addtocontents{toc}{\protect\setcounter{tocdepth}{1}}

    \tableofcontents


\section{Introduction}

The continued fraction expansion of a real number $\alpha\in\mathbb{R}$ is an expression of the form
\[ \alpha = a_0+\frac{1}{a_1+\displaystyle\frac{1}{a_2+\displaystyle\frac{1}{\ddots}}}, \]
where $a_0\in\mathbb Z$ and $a_i\in\mathbb N$ for $i\ge1$, and is also denoted by $\alpha=[a_0;a_1,a_2,\dots]$. 
This expansion arises naturally from the Euclidean algorithm and plays a central role in the study of Diophantine approximation. In particular, the convergents of the continued fraction expansion provide optimal rational approximations to real numbers and capture important arithmetic properties of the number.

In 1850, C. Hermite posed the problem of generalizing ordinary continued fractions to higher dimensions~\cite{Hermite1839Letter}. In response to this question, \citeauthor{K1895} introduced in 1895 a geometric construction that provides a higher-dimensional analogue of continued fraction expansions~\cite{K1895}. Let us briefly recall the construction of Klein sails. 

Let $L_1,\ldots,L_n$ be $n$ linearly independent rays in $\RR^n$. The choice of directions determines the cone
\[
C_{L_1,\ldots,L_n}
=
\left\{
t_1w_1+\cdots+t_nw_n:
t_1,\ldots,t_n\geq 0
\right\},
\]
where $w_i$ is a vector pointing in the chosen direction of $L_i$.

Consider the convex hull of the non-zero integer lattice points contained in this cone, namely
\[
\operatorname{conv}\big((C_{L_1,\dots,L_n}\cap\mathbb Z^n)\setminus\{0\}\big).
\]
The resulting polyhedron is called the \emph{Klein polyhedron}, and its boundary is called the \emph{Klein sail} associated with $L_1,\dots,L_n$, and denoted by $\Klein_{L_1,\dots,L_n}$. The $(n-1)$-dimensional facets of the Klein sail are referred to as its \emph{faces}.

Throughout the paper, we restrict our attention to cones whose boundary hyperplanes contain no non-zero lattice points. Under this assumption, every face of the associated Klein sail is a bounded integral polytope; see \cite{M03}.

To illustrate the connection with ordinary continued fractions, let us restrict to the case $n=2$ and consider the cone generated by
$L_1=(1,\alpha)$ and $L_2=(1,-1/\alpha)$ for an irrational number $\alpha$. 

\begin{figure}[h!]
    \centering
        \tikzmath{\xmin = -5; \ymin =-10; \xmax = 20; \ymax =12; } 
    \begin{tikzpicture}[scale=0.4]
        \def\alpha{0.7360679774997898} 
        
        \foreach \x in {\xmin, ..., \xmax} {
            \foreach \y in {\ymin, ..., \ymax} {
                \fill[black!20] (\x,\y) circle (1.5pt);
            }
        }

        \draw[thin, ->] (\xmin-0.5,0) -- (\xmax+0.5,0) node[right] {$x$};
        \draw[thin, ->] (0,\ymin-0.5) -- (0,\ymax+0.5) node[above] {$y$};

        \foreach \x in {\xmin, 5, 10, 15, \xmax}
            \draw (\x,2pt) -- (\x,-2pt) node[below, anchor=north, font=\tiny] {{\pgfmathtruncatemacro{\myint}{\x}\myint}};
        \foreach \y in {\ymin, -5, 5, 10}
            \draw (2pt,\y) -- (-2pt,\y) node[left, anchor=east, font=\tiny] {{\pgfmathtruncatemacro{\myint}{\y}\myint}};

        \draw[blue!70!black, domain=0:16, samples=100, ->] 
            plot (\x, {\alpha*\x}); 
        \draw[blue!70!black, domain=0:9, samples=100, ->] 
            plot ({\alpha*\x}, -\x);

        \foreach \point in {
            (2, 1), (7, 5), (11, 8)  
        } {
            \fill[black] \point circle (4pt);
        }

        \draw[thick, black] (3, -4) -- node[above, sloped, font=\footnotesize] {$a_{3}$} (1, -1)
                                    -- node[above, sloped, rotate=180, font=\footnotesize] {$a_1$} (1, 0)
                                    -- node[below, sloped, font=\footnotesize] {$a_2$} (3, 2)
                                    -- node[below, sloped, font=\footnotesize] {$a_4$} (15, 11);

        \draw[thick, dashed, black] (3, -4) --  ++({0.1*11}, {0.1*-15});
        \draw[thick, dashed, black] (15, 11) --  ++({0.06*19}, {0.06*14});
        
        \foreach \point in {
            (1, 0), (3, 2), (15, 11),
            (1, -1), (3, -4)
        } {
            \fill[red] \point circle (4pt);
        }

    \end{tikzpicture}
    \caption{Part of the Klein sail $\Klein_{L_1,L_2}$ associated to the irrational number $\alpha =\frac{-3 + \sqrt{20}}{2} \approx 0.73607$. The continued fraction expansion of $\alpha$ is $[0;\overline{1,2,1,3}]$ and $a_1=1$, $a_2=2$, $a_3=1$, $a_4=3$ are precisely the integer lengths of the edges of the depicted part of the Klein sail. The convergents can be read of the (red) vertices of the sail.}
    \label{fig:gcf-2dim}
\end{figure}
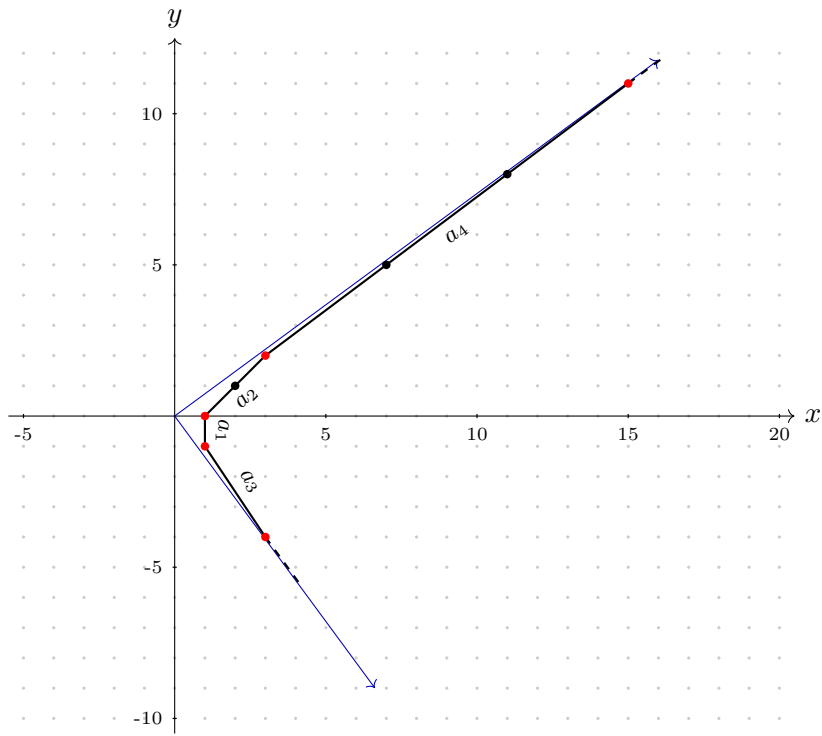

In this case, the Klein sail is an infinite polygonal chain whose vertices are integer lattice points. We associate to each edge of the sail its \emph{integer length}, defined as the number of integer points on the edge minus one. A classical result proved in \cite{EGH89} states that these integer lengths are precisely the entries of the continued fraction expansion of $\alpha$ (see \Cref{fig:gcf-2dim}). Moreover, the vertices of the Klein sail correspond to the convergents of the continued fraction expansion of $\alpha$. 

This correspondence provides a geometric interpretation of the continued fraction expansion of a real number. Consequently, Klein sails in higher dimensions are often referred to as \emph{geometric higher-dimensional continued fractions}; see, for example, \cite{Karpenkovbook}.

The theory of Klein sails has a rich history, closely intertwined with the theory of continued fractions. We briefly recall some of the major developments and refer the reader to~\cite{Karpenkovbook,Lachaudsurvey} and references therein for more detailed accounts.

In the two-dimensional setting, a real number is badly approximable if and only if the partial quotients in its continued fraction expansion are uniformly bounded. Geometrically, this is equivalent to boundedness of the integer lengths of the edges of the associated Klein sail. In \cite{G05,G08}, \citeauthor{G05} established a multidimensional analogue of this characterization. More precisely, \cite[Theorem~2.1]{G08} shows that the local invariants of the vertices and $(n-1)$-dimensional faces of $\Klein$, called \emph{determinants}  (see \emph{loc.\ cit.} for the definition), are uniformly bounded if and only if
\[
\min_{x\in\ZZ^n\setminus\{0\}}
\abs{\langle w_1,x\rangle\cdots\langle w_n,x\rangle}>0,
\]
where $w_i$ denotes the unit normal vector to the ray $L_i$.

Another fundamental result in the classical theory is Lagrange's theorem, which states that a real number has an eventually periodic continued fraction expansion if and only if it is a positive quadratic irrational. This characterization also admits a higher-dimensional analogue. \citeauthor{T83}~\cite{T83} proved that Klein sails arising from algebraic data are topologically periodic, while \citeauthor{K94}~\cite{K94} established the converse by showing that topologically periodic Klein sails are necessarily of algebraic origin (see \Cref{sec:cubic-sails}). For more details, we refer the reader to~\cite{Lachaud,L98}. 

Owing to this periodic structure, Klein sails associated with totally real cubic fields constitute one of the best understood classes of examples. Several such sails have been computed explicitly; see, for example, \citeauthor{BP94}~\cite{BP94}, \citeauthor{P00}~\cite{P00,P05}, and \citeauthor{K04}~\cite{K04}.

The explicit construction and computation of Klein sails has also been a central topic of investigation (see \cite{BP94,Ko96,K06}). In particular, \citeauthor{BP94}~\cite{BP94} proposes an algorithm for the successive computation of the faces of a Klein sail. This algorithm is the natural higher-dimensional analogue of the geometric procedure for computing the continued fraction expansion of a real number. It constructs many vertices lying on faces of level one (see Definition~\ref{def: level}) and, in a lot of cases, determines the entire Klein sail. Since the present paper (Theorem~\ref{thm:face-equi}) provides quantitative control on the number of faces of level greater than one, our results offer theoretical support for this heuristic.

\medskip

Beyond these explicit constructions, \citeauthor{Ar98}~\cite{Ar98,Ar02} considered statistical aspects of geometric multidimensional continued fractions and posed several questions concerning the statistical properties of Klein sails and their local characteristics. Among these questions, he formulated the following conjecture:

\begin{conjecture*}[{\cite[Conjecture~4, first part]{Ar98}}]
For almost every simplicial cone whose vertex is the origin of the space $\RR^n\supset\ZZ^n$, the statistic of any reasonable characteristic of the sail of this cone (that is, of the boundary of the convex hull of the set of integer points lying inside the cone) is universal (independent of the cone).
\end{conjecture*}

The author mentioned that by \emph{reasonable characteristic}, he meant, for instance, the distribution of the integer lengths of the edges, the distribution of $2$-faces according to the number of their edges or by the integral area and by the affine integer equivalence classes, the distribution of the vertices according to the number of edges containing them and the like.

\medskip

In this paper, we study the distribution of affine integer equivalence classes of $(n-1)$-dimensional faces of Klein sails in $\RR^n$. To explain this notion, recall that every face of a Klein sail is an integral polytope contained in an affine rational $(n-1)$-dimensional hyperplane.

Two such faces are said to be \emph{affinely integer equivalent} if one is mapped to the other by an element of the affine unimodular group $\GL_n(\ZZ)\ltimes\ZZ^n$. Since every affine rational $(n-1)$-dimensional hyperplane can be mapped onto $\RR^{n-1}\times\{0\}$ by an element of $\GL_n(\ZZ)\ltimes\ZZ^n$, every face is affinely integer equivalent to an integral polytope in $\ZZ^{n-1}\times\{0\}$. Consequently, the set of affine integer equivalence classes of $(n-1)$-dimensional faces may be identified with
\[
\FC^{n-1}
=
\frac{\{\text{integral polytopes in }\ZZ^{n-1}\}}
{\GL_{n-1}(\ZZ)\ltimes\ZZ^{n-1}}.
\]

For a face $f$ of a Klein sail, we denote its equivalence class in $\FC^{n-1}$ by $[f]$. Explicitly, if $h\in\GL_n(\ZZ)\ltimes\ZZ^n$ maps the affine hyperplane containing $f$ onto $\RR^{n-1}\times\{0\}$, then $[f]$ is the affine integer equivalence class of the projection of $hf$ onto the first $n-1$ coordinates. This definition is independent of the choice of $h$.

Every representative of the class $[f]$ is obtained from $f$ by an affine unimodular transformation. Consequently, the class $[f]$ retains the complete integral-affine structure of the face. In particular, it determines numerical invariants such as the normalized volume (integer area in dimension two), as well as finer combinatorial and arithmetic data, including the number of vertices, the edge structure, and the configuration of lattice points.

For $n=2$, this notion recovers the classical theory of continued fractions. Indeed, every (non-trivial) integral interval in $\ZZ$ is affinely integer equivalent to $\{0,1,\ldots,m\}$ for a unique $m\in\NN$, and hence $\FC^1$ may be naturally identified with $\NN$. Under this identification, the affine integer equivalence class of an edge is completely determined by its integer length.  Consequently, the affine integer equivalence classes of the edges of a Klein sail correspond precisely to the partial quotients in the associated continued fraction expansion.

Their asymptotic distribution is given by the classical Gauss--Kuzmin distribution $\nu_{GK}$, a probability measure on $\NN\simeq\FC^1$. It describes the limiting distribution of the partial quotients in the continued fraction expansion of a Lebesgue-typical real number. More precisely, for Lebesgue almost every $\alpha\in\RR$ and every $k\in\NN$, if $\alpha=[a_0;a_1,a_2,\ldots]$, then
\[
\frac{1}{N}\sum_{n\le N}\mathbf 1_{\{k\}}(a_n)
\longrightarrow
\nu_{GK}(\{k\})
=
\frac{1}{\ln 2}
\ln\!\left(1+\frac{1}{k(k+2)}\right),
\]
as $N\to\infty$.

Interpreted in terms of Klein sails, this means that for a generic pair of lines, the asymptotic proportion of edges having integer length $k$ is $\nu_{GK}(\{k\})$. Thus, the two-dimensional case follows from the classical Gauss--Kuzmin theorem. This classical result provides supporting evidence and motivation for the formulation of Arnold's conjecture in higher dimensions.

For $n\geq 3$, the conjecture was resolved by \citeauthor{KS99}~\cite{KS99}. More precisely, they proved the existence of a limiting measure describing the distribution of affine integer equivalence classes of faces of Klein sails. To state their result, we introduce the following notation.

\begin{definition}
    Let $V\subset\RR^n$ be an integral-affine plane not containing the origin. Then there exists a vector $m= (m_1, \ldots, m_n)\in\ZZ^n$ with $\gcd(m_1, \ldots, m_n)=1$ and an integer $d\ge 1$ such that
    \[ V=\{x\in\RR^n:\langle m,x\rangle=d\}. \]
    The \highlight{integral distance from the origin} of $V$ is defined to be the integer $d$. We denote it by
    \[ \operatorname{dist}_{\mathbb Z}(V,0)=d. \]
\end{definition}
\begin{definition} \label{def: level}
    Let $f$ be a $(n-1)$-dimensional face of a Klein sail $\Klein$. Let $V_f$ be the affine plane containing $f$.
    \begin{enumerate}
        \item The \highlight{level} of $f$ is
        \[ \level(f) \defeq \operatorname{dist}_{\mathbb Z}(V_f,0). \]
    
        \item For $i=1,\ldots, n$, let $P_i=V_f\cap L_i$, where $L_1,\ldots,L_n$ are the rays of the cone defining the Klein sail $\Klein$. The \highlight{visible distance} of $f$ is
        \[ d_v(f) \defeq \max_{i\neq j} \frac{\norm{P_i}[2]}{\norm{P_j}[2]}. \]
    \end{enumerate}
\end{definition}

\begin{theorem}[\cite{KS99}]
\label{thm:KS99}
    There exists a measure $\nu_{\FC}$ on $\FC^{n-1}\times\NN$ such that for every finitely supported function $\varphi\colon\FC^{n-1}\times\NN$ and almost every linearly independent rays $L_1,\ldots,L_n$ in $\mathbb R^n$ the associated Klein sail $\Klein_{L_1,\ldots,L_n}$ satisfies
    \begin{align}\label{eq:KS-msr}
        \frac{1}{T^{n-1}}\sum_{\substack{f\text{ face of }\Klein_{L_1,\ldots,L_n} \\ d_v(f)\leq e^T}} \varphi([f],\level(f)) \to \nu_{\FC}(\varphi).
    \end{align}
\end{theorem}

The proof of this theorem is based on a dual approach to multidimensional continued fractions. Instead of fixing the lattice and varying the simplicial cone, \citeauthor{KS99}~\cite{KS99} fix the cone $\RR_{\geq 0}^n$ and vary the lattice. More precisely, if $g\in\SL_n(\RR)$ is chosen so that $C=g^{-1}\RR_{\geq0}^n$, then the Klein sail associated with the cone $C$ can be equivalently viewed as the Klein sail associated with the fixed cone $\RR_{\geq 0}^n$ and the varying lattice $g\ZZ^n$. Thus, the problem is transferred to the space of unimodular lattices
\[
\{g\ZZ^n:g\in\SL_n(\RR)\}\simeq \SL_n(\RR)/\SL_n(\ZZ).
\]

The diagonal flow preserves the cone $\RR_{\geq0}^n$. \citeauthor{KS99} constructed a cross-section for this flow such that faces of Klein sails correspond to visits of the orbit to the cross-section. The ergodicity of the diagonal flow on $\SL_n(\RR)/\SL_n(\ZZ)$ then yields limiting statistics for affine integer equivalence classes of faces of generic Klein sails. They also studied several related statistics, including the expected number of lattice points in the interiors of faces; see \cite{KS99}.

However, the finiteness of the cross-sectional measure remained as an open question, and subsequently, also raised by \citeauthor{Ka17}~\cite{Ka17} (see Section~\ref{subsec:Question of Karpenkov}) in an equivalent form.

From the perspective of the above theorem, this translates to whether the limiting measure $\nu$ finite? More generally, can the convergence in~\eqref{eq:KS-msr} be extended from finitely supported functions to a larger class of functions, such as bounded functions on $\FC^{n-1}\times\NN$? In particular, it remains unclear whether a direct analogue of the Gauss--Kuzmin distribution holds, where the normalization on the left-hand side of~\eqref{eq:KS-msr} is replaced by
\[
\#\{f:d_v(f)\leq e^T\},
\]
the number of faces of a Klein sail satisfying a bound on its visible distance.

\subsection{Main result}

In this paper, we prove the finiteness of the cross-sectional measure associated with Klein sails in $\RR^3$. As a consequence, we obtain the following generalized Gauss--Kuzmin distribution for Klein sails in $\RR^3$.
\begin{theorem}\label{thm:face-equi}
    The measure $\nu_{\FC}$ in Theorem~\ref{thm:KS99} is finite. Moreover, for every bounded function $\varphi$ on $\FC^2\times\mathbb N$ and for Lebesgue almost every triple of linearly independent rays $L_1,L_2,L_3\subset\mathbb R^3$ the following holds.
     \begin{align}\label{eq:KS-msr 2}
        \frac{1}{T^{n-1}}\sum_{\substack{f\text{ face of }\Klein_{L_1,\ldots,L_n} \\ d_v(f)\leq e^T}} \varphi([f],\level(f)) \to \nu_{\FC}(\varphi).
    \end{align}
    As a result, for Lebesgue almost every triple of linearly independent rays $L_1,L_2,L_3\subset\mathbb R^3$ and for every bounded function $\varphi\colon\FC^2\times\NN\to\RR$, we have
    \[ \frac{1}{N_{\Klein}(T)} \sum_{\substack{f \text{ face of }\Klein\\d_v(f)\le T}} \varphi([f],\level(f)) \longrightarrow \frac{1}{\nu_{\FC}(\FC^2\times\NN)}\int_{\FC^2\times\mathbb N} \varphi\dd\nu_{\FC}, \]
    as $T\to\infty$, where
    \[ N_{\Klein}(T) \coloneqq\#\{f \text{ face of }\Klein : d_v(f)\le T\}. \]
    Equivalently, the affine integer equivalence classes of faces of the Klein sail, ordered according to their visible distance, become equidistributed in $\FC^2\times\mathbb N$ with respect to $\nu_{\FC}$.
    
    Moreover, the measure $\nu_{\FC}$ satisfies
    \begin{align}\label{eq:level asymptote}
        \frac{1}{k^3} \ll \nu_{\FC}(\FC^2\times \{k\}) \ll \frac{\ln k +1}{k^2},
    \end{align}
    for all $k \in \NN$, where the implied constants are independent of $k$.
\end{theorem}

In particular, it follows that for almost every Klein sail $\Klein$ we have $N_{\Klein}(T) \sim c(\log T)^2$ where $c= \nu_{\FC}(\FC^2\times\NN)$. It would be interesting to compute this constant explicitly.

We note that explicit frequencies for certain classes of faces of three-dimensional Klein sail were previously computed by \citeauthor{Ka07}~\cite{Ka07}. The present work establishes a general distributional framework addressing all types of faces simultaneously.

\begin{remark}
    In dimension $2$, every face of a Klein sail has level equal to $1$. This fact is not completely immediate from the definitions. However, it is equivalent to the well-known determinant formula for continued fractions, stating that two consecutive best approximants $\frac{p_n}{q_n}$ and $\frac{p_{n+1}}{q_{n+1}}$ of any irrational number $\alpha\in \RR$ satisfy $p_nq_{n+1} - q_np_{n+1} = \pm 1$. Geometrically, this is equivalent to the fact that the lattice spanned by the vectors $(p_n,q_n)$ and $(p_{n+1},q_{n+1})$ is equal to $\ZZ^2$.
    
    The situation changes dramatically in higher dimensions. Already in dimension $3$, there exist Klein sails with faces of level greater than $1$. Consequently, the geometry and combinatorics of faces become substantially richer than in the classical two-dimensional continued fraction setting.
    
    Indeed, as will become apparent from the proof of Theorem~\ref{thm:face-equi}, the main technical difficulty lies in understanding faces of integral distance greater than $1$.
    It would be interesting to determine the precise asymptotic behaviour of $\nu_{\FC}(\FC^2\times\{k\})$ as $k\to\infty$.
\end{remark}
\begin{remark}
We also refer the reader to Corollary~\ref{cor:no-fund-domain}, where we study the structure of the set
\[
\supp(\nu_{\FC})\cap(\FC^2\times\{k\}),
\]
for $k\ge2$.
\end{remark}

\subsection{Questions of Karpenkov}\label{subsec:Question of Karpenkov}

Another approach to generalizing Gauss--Kuzmin statistics in higher-dimensional
continued fractions is through geometric relative frequencies (see definition below), introduced and
studied extensively by Karpenkov~\cite{Ka07, Ka17} and others. We briefly recall this viewpoint
and explain its relation to our result.

Let $\sK_n$ denote the space of all Klein sails in $\RR^n$. A Klein sail is
determined by an ordered collection of $n$ linearly independent rays
$L_1,\ldots,L_n\subset\RR^n$, considered up to permutation. Thus, the space
$\sK_n$ can be identified with
\[
    \sK_n \simeq \SL_n(\RR)/(\mathrm{D}_n\rtimes S_n),
\]
where $\mathrm{D}_n$ denotes the subgroup of diagonal matrices with positive entries and $S_n$ is the permutation group acting by permuting the rays. Indeed, the correspondence is given by
\[
    g\mapsto \{\RR_{\geq 0}g\mathbf e_1,\ldots,
    \RR_{\geq 0}g\mathbf e_n\},
\]
where $\mathbf e_i$ denote the standard basis vectors, equivalently, by associating to $g$ the simplicial cone~$g\RR_{\geq0}^n$.

Consequently, $\sK_n$ carries a natural smooth manifold structure and admits an
$\SL_n(\RR)$-invariant measure, unique up to scaling. We fix such a measure
$\mu$, which is usually referred to as the Möbius measure.

For two convex subsets $f_1,f_2\subset\ZZ^n$ contained in affine rational
hyperplanes, we say that $f_1$ and $f_2$ are \emph{integer equivalent} if
there exists an element $\gamma\in\SL_n(\ZZ)$ such that $f_2=\gamma f_1$.

For a convex subset $f\subset\ZZ^n$ contained in an affine rational
hyperplane, its relative frequency is given by
\[
    \rel(f)
    =
    \mu\left(
    \{ \Klein \in\sK_n:f\text{ appears as a face of }\Klein \}
    \right).
\]
The $\SL_n(\ZZ)$-invariance of the measure $\mu$ implies that
\[
    \rel(\gamma f)=\rel(f),\qquad \gamma\in\SL_n(\ZZ).
\]
Indeed, the set of Klein sails containing $\gamma f$ as a face is the image under $\gamma$ of the
set of Klein sails containing $f$ as a face, due to invariance of $\ZZ^n$ under $\gamma$, and hence has the same measure by the
$\SL_n(\RR)$-invariance of $\mu$.

Therefore, the relative frequency depends only on the integer equivalence
class of the face. Karpenkov~\cite{Ka17} asked whether these frequencies satisfy suitable
finiteness properties.

\begin{problem}[{\cite[Problem~16]{Ka17}}]
For every positive integer constant C there exist only finitely many pairwise
integer non-equivalent faces with frequencies exceeding C.
\end{problem}

\begin{problem}[{\cite[Problem~17]{Ka17}}]
Is it true that the sum of all relative frequencies for all integer equivalence classes of faces is finite
for higher dimensions $(n \geq 3)$?
\end{problem}

We now answer these questions affirmatively in dimension three. 

\begin{theorem}
\label{thm:Karpenkov question}
    In dimension $3$, the relative frequencies of Klein sail faces satisfy
    \[
        \sum_{[f]}\rel([f])<\infty,
    \]
    where the sum is taken over all integer equivalence classes of faces.
    As a result, for every $C>0$, there exist only finitely many integer
    non-congruent faces satisfying $\rel([f])>C$.
\end{theorem}

The key point
is that integer equivalence classes of faces correspond precisely to subsets of
the cross-section of the diagonal flow constructed in Section~\ref{sec:explicit-crosssection}.
More precisely, the relative frequency of a class of a face is proportional to the measure of the
corresponding subset of the cross-section. Hence, the finiteness of the total
sum of relative frequencies is equivalent to the finiteness of the
cross-sectional measure. For more details, see Section~\ref{subsec:Karpenkov}.

\subsection{Equidistribution of periodic Klein sails}\label{sec:cubic-sails}

Let $\Klein$ be the Klein sail associated to a simplicial cone $C\subset\RR^3$. We say that $\Klein$ is \highlight{periodic}, if and only if there is a rank-$2$ abelian subgroup of $\SL_3(\ZZ)$ which leaves the Klein sail invariant. More precisely, if the simplicial cone is given by $C = g\RR^3_{\geq0}$, then the associated Klein sail $\Klein$ is periodic if and only if 
\[ \Xi(\Klein) \defeq \SL_3(\ZZ) \cap gA_+g^{-1} \]
is a rank-$2$ subgroup of $\SL_3(\ZZ)$. For periodic Klein sails $\Klein$ one can choose a finite set of faces $\mathcal{F}_\Klein$ of $\Klein$ such that every other face of $\Klein$ is the image of a face in $\mathcal{F}_\Klein$ under some element of $\Xi(\Klein)$.

Let $C$ be a simplicial cone generated by the three linearly independent rays $L_1,L_2,L_3\in\RR^3$. Then $C$ is said to be split, if $L_1,L_2,L_3$ are eigendirections of a hyperbolic matrix $h\in\Mat[3](\ZZ)$ whose characteristic polynomial is an irreducible cubic polynomial with real, positive roots. \citeauthor{K94} characterized periodic Klein sails in terms of split cones in \cite{K94}.

\begin{theorem}[{\cite{K94}, \cite[Prop.~3]{L98}}]
    The Klein sail $\Klein$ associated to a simplicial cone $C$ is periodic if and only if $C$ is split.
    In fact, $C$ is split if and only if there is a totally real algebraic number field $F$ of degree $3$ and a $\QQ$-basis $(m_1,m_2,m_3)$ of $F$ such that
    \begin{align}\label{eq:cone-cubic}
        C = \{x\in\RR^3 \st\forall i\in\{1,2,3\}~\langle \sigma(m_i),x\rangle \geq 0\},
    \end{align} 
    where $\sigma\colon F \hookrightarrow \RR^3$ is the Minkowski embedding of $F$. Equivalently,        
    \[
    C = \{c_1L_1 + c_2L_2 + c_3L_3 \in\RR^3\st c_1,c_2,c_3\geq 0\}, 
    \] 
    is generated by the three rays $L_1$, $L_2$ and $L_3$ generated by the column vectors of the matrix
    \[ \begin{pmatrix} - & \sigma(m_1) & - \\- &  \sigma(m_2) & - \\- &  \sigma(m_3) & -\end{pmatrix}^{-1}. \]
\end{theorem}

We explain this equivalence briefly and refer to \cite{ELMV11, K94, L98} for a thorough discussion of this equivalence. 

Let $C\subset\RR^3$ be a split cone. Let $h\in\Mat[3](\ZZ)$ be such such that $C$ is generated by eigendirections of $h$. Then the characteristic polynomial $P$ of $h$ is a cubic polynomial with positive and real roots, and we let $\order_{P}$ be the ring $\ZZ[t]/P$ and $F_P = \order_P\otimes\QQ$ be the associated cubic field. 

The map $\theta\colon\order_P \to \Mat[3](\ZZ)$ obtained by mapping $ t\to h$ and extending linearly makes $\ZZ^3$ into an $\order_P$-module and gives an embedding $F_P\to\Mat[3](\QQ)$. Thus, there is a free rank-$3$ module $M$ in $F_P$, defined up to multiplication by $F_P^\times$, such that $M$ and $\ZZ^3$ are isomorphic as $\order_P$-modules (see \cite{LM33}).

Now, fix a $\ZZ$-basis $(m_1,m_2,m_3)$ of the module $M$, such that the coordinate map $\iota\colon M \to \ZZ^3$ is an $\order_P$-module isomorphism that satisfies
\[ \iota(\lambda m) = \theta(\lambda)\iota(m) \]
for all $\lambda\in\order_P$ and $m\in M$. 
If we identify $\sigma\colon M\otimes\RR\xrightarrow{\simeq}\RR^3$ via the Minkowski embedding $\sigma$ given by the three Galois embeddings $\sigma_1,\sigma_2,\sigma_3\colon F_P\hookrightarrow\RR$, then the cone $C$ corresponds under $\iota$ to the axes of $\RR^3$, that is $\iota\circ\sigma^{-1}(\RR^3_{\geq0}) = C$. Indeed, let $\Psi\colon\ZZ^3\to M$ be the inverse of $\iota$, that is $\Psi(x_1,x_2,x_3) = x_1m_1+x_2m_2+x_3m_3,$
then we claim that 
\[ C = \{x\in\RR^3\st \sigma\circ\Psi(x) \in \RR^3_{\geq0} \}. \]
Since 
\[ \sigma\circ\Psi(x_1,x_2,x_3) = \begin{pmatrix} | & | & | \\ \sigma(m_1) &  \sigma(m_2) & \sigma(m_3) \\| & | & |\end{pmatrix}\begin{pmatrix} x_1 \\ x_2 \\ x_3\end{pmatrix} \]
we have that
\[ \sigma\circ\Psi(x)\in\RR^3_{\geq0} \Leftrightarrow \forall i\in\{1,2,3\}~\langle \sigma(m_i),x\rangle \geq 0, \] 
as claimed. In particular, the Klein sail $\Klein$ associated to $C$ may, under $\iota$, be seen as the convex hull of the set
\[ M^+ \defeq\{m\in M\st \forall i\in\{1,2,3\}~\sigma_i(m)\geq 0 \} \]
of totally positive elements of the module $M$.

In order to see why $\Klein$ is periodic, consider the order
\[ \order_M \defeq \{\lambda\in F_P\st \lambda M \subset M\} \]
associated to $M$ and notice that $\order_P\subseteq\order_M$. By definition of $\theta$ we have
\[ \theta(F_P) \cap \Mat[3](\ZZ) = \theta(\order_M). \]
By Dirichlet's unit theorem, the units $\order_P^\times$ form a rank-$2$ subgroup and the totally positive units
\[ \mathcal{U}_P \defeq \{\lambda\in\order_P^\times\st\sigma_i(\lambda)\geq 0\} \]
are a finite index subgroup of $\order_P^\times$. Since
\[ \theta(\order_M^\times) = \theta(F_P) \cap \GL_3(\ZZ)\subseteq \GL_3(\ZZ), \]
the subgroup $\theta(\mathcal{U}_P)\subseteq \SL_3(\ZZ)$ forms a rank $2$-subgroup leaving invariant the cone $C$ defined by the eigenvectors of the matrix $h$ as well as the associated Klein sail $\Klein$, thus $\Klein$ is periodic. 

Conversely, let $\Klein$ be a periodic Klein sail associated to the cone $C = g\RR^3_{\geq0}$. Then $gA_+g^{-1}\subset\SL_3(\RR)$ is a $2$-dimensional real split torus, and $\Xi(\Klein) = \SL_3(\ZZ)\cap gA_+g^{-1}$ is a rank-$2$ free abelian subgroup, hence a lattice in $gA_+g^{-1}$. In particular, the $A_+$-orbit of $g^{-1}\SL_3(\ZZ)$ is compact (also referred to as periodic).
The characteristic polynomial of any regular element $h\in\Xi(\Klein)$ (that is with pairwise different eigenvalues) is a monic cubic irreducible polynomial with integer coefficients and $C$ is generated by the eigendirections of $h$ and thus is a split cone. 
In particular, by the discussion above we obtain that the Klein sail $\Klein$ may be seen as the convex hull of totally positive elements of a free rank-$3$ module $M$ in a totally real cubic number field $F$. By convenience, we denote the order $\order_M$ associated to this module $M$ also by $\order_\Klein$.

\medskip

There is a natural equivalence relation between periodic Klein sails given by left multiplication by matrices in $\SL_3(\ZZ)$. We denote the class of the Klein sail $\Klein$ by $[\Klein]$.
By examining the discussion above, we note that there is a tight connection between equivalence classes of periodic Klein sails and modules in totally real cubic fields. Indeed, choosing a basis of the module $M$ amounts to picking a representative in the equivalence class of the corresponding Klein sail. Thus 
we have the $1$-to-$1$ correspondence
\begin{gather}\label{eq:packet-equiv}
    \left\{\begin{aligned}
        &\text{totally real cubic fields $F$ considered up to isomorphism}\\[-1.1ex]
        &\text{and free rank-$3$ modules $M$ in $F$, up to multiplication by $F^\times$}
    \end{aligned}
    \right\}\notag\\
    \longleftrightarrow \\
    \{\text{equivalence classes $[\Klein]$ of periodic Klein sails}\}.\notag
\end{gather}
 Given a totally real cubic field $F$ and a free rank-$3$ module $M$ in $F$, defined up to multiplication by $F^\times$, we pick a basis $(m_1,m_2,m_3)$ of $M$ and define the isomorphism $\iota\colon M\to\ZZ^3$ as before. This gives an embedding $\theta\colon F\to\Mat[3](\QQ)$ via the multiplication representation in the chosen basis. 
 The cone $C$ defined as in \eqref{eq:cone-cubic} will give rise to a periodic Klein sail $\Klein$ such that $\Xi(\Klein) = \theta(\mathcal{U}_M)$ is a rank-$2$ subgroup of $\SL_3(\ZZ)$, where $\mathcal{U}_M$ denote the totally positive units of $\order_M= \order_\Klein$. 

 For the other direction, not that two equivalent periodic Klein sails $\Klein$ and $\Klein'$ differ by right multiplication by an element $\gamma\in\SL_3(\ZZ)$. In particular, the associated rank-$2$ subgroups of $\SL_3(\ZZ)$ satisfy $\Xi(\Klein') = \Xi(\gamma\Klein) = \gamma\Xi(\Klein)\gamma^{-1}$. Hence, since they are conjugated by $\gamma\in\SL_3(\ZZ)$, regular elements in $\Xi(\Klein)$ and $\Xi(\Klein')$ define the same characteristic polynomial and thus the same totally real cubic field. Moreover, the difference between $\Klein$ and $\Klein'$ is reflected in the choice of the basis $(m_1,m_2,m_3)$ of the module $M$ described above. Thus, $[\Klein]$ determines $M$ and \eqref{eq:packet-equiv} is established.

\begin{definition}
    Let $\order\subset F$ be an order inside a totally real cubic field. Let $\sigma_1,\sigma_2,\sigma_3\colon F\hookrightarrow\RR$ be the three Galois embeddings.
    \begin{enumerate}
        \item The \highlight{discriminant} of $\order$ is defined as 
                \[ \disc(\order) \defeq (\det(\sigma_i(\alpha_j))_{i,j=1}^3)^2 \]
            where $(\alpha_j)_{j=1}^3$ is any $\ZZ$-basis of $\order$.
        \item The \highlight{regulator} $R_\order$ of $\order$ is the covolume of its unit lattice under the logarithmic embedding. More precisely, let
        \begin{align*}
            \ell\colon\order^\times &\to \left\{(x_1,x_2,x_3)\in \RR^3\st \sum_{i=1}^3x_i = 0\right\}\\
            \alpha &\mapsto (\log|\sigma_1(\alpha)|,\log|\sigma_2(\alpha)|,\log|\sigma_3(\alpha)|)
        \end{align*} 
        then $R_\order = \operatorname{covol}(\pi(\ell(\order^\times)))$, where $\pi\colon\{(x_1,x_2,x_3)\in \RR^3\st \sum_{i=1}^3x_i = 0\}\to\RR^2$ is the projection to the first two coordinates and the covolume is taken with respect to the usual Lebesgue measure on $\RR^2$.
    \end{enumerate} 
\end{definition}

Recall that for a periodic Klein sail $\Klein$, we denote by $\mathcal{F}_\Klein$ a finite set of faces of $\Klein$, such that any other face of $\Klein$ is the image of a face in $\mathcal{F}_\Klein$ under some element of $\Xi(\Klein)$. Note that since $\Xi(\Klein)\subset\SL_3(\ZZ)$, the integer-face types and the levels of the elements of $\mathcal{F}_\Klein$ do not depend on the specific choice of $\mathcal{F}_\Klein$. Further, if $\Klein$ and $\Klein'$ are two equivalent periodic Klein sails they differ by an element in $\SL_3(\ZZ)$, thus the integer-face types and levels of the finite set of faces $\mathcal{F}_\Klein$ and $\mathcal{F}_{\Klein'}$ coincide. In particular, we can associate to every equivalence class of periodic Klein sails $[\Klein]$ a natural measure on $\FC^{2}\times\NN$ given by
\[ \nu_{[\Klein]} \defeq \frac{1}{R_{\order_\Klein}}\sum_{f\in \mathcal{F}_\Klein}\delta_{([f],\level(f))}, \]
where $R_{\order_\Klein}$ is the regulator of the order $\order_\Klein$. Note that the integer-affine type of faces in $\mathcal{F}_\Klein$ can be the same, so $\nu_{[\Klein]}$ is a weighted average of integer-affine types appearing in the periodic Klein sail~$\Klein$.

We define collection of equivalence classes of periodic Klein sails
\[ \mathcal{P}_{\order} = \{[\Klein] \st \Klein \text{ periodic Klein sail and $\order_\Klein = \order$} \} \]
and call it the coarse packet associated to $\order$. It is a finite collection, since up to multiplication by $F^\times$, there are only finitely free rank-$3$ modules $M$ in $F$ such that $\order_M = \order$. Indeed, $M$ can be seen as a rank-$1$ $\order$-module and since $M\subset F$, it is naturally a fractional $\order$-ideal and the condition $\order_M = \order$ requires the ideal to be proper. As there are only finitely many proper fractional $\order$-ideals up to scaling by $F^\times$, finiteness of $\mathcal{P}_{\order}$ follows.

Since $\mathcal{P}_{\order}$ is finite, we can associate to it the probability measure $\nu_{\order}$ on $\FC^{2}\times\NN$ defined by averaging of $\nu_{[\Klein]}$ for every $[\Klein]\in\mathcal{P}_{\order}$, that is
\begin{align}\label{def:nu_O}
     \nu_{\order} \defeq \frac{1}{\abs{\mathcal{P}_{\order}}}\sum_{[\Klein]\in\mathcal{P}_{\order}}\nu_{[\Klein]} = \frac{1}{\abs{\mathcal{P}_{\order}}\cdot R_{\order}}\sum_{[\Klein]\in\mathcal{P}_{\order}}\sum_{f\in \mathcal{F}_\Klein}\delta_{([f],\level(f))}.
\end{align}
As a consequence of the cross-section theory constructed in the paper, together with the influential work of \citeauthor{ELMV11} in \cite{ELMV11}, we deduce the following.

\begin{theorem}\label{thm:periodic-equi}
    Let $\order_n\subset F_n$ be a sequence of orders, in a sequence of totally real fields $K_n$. Let $\nu_{\mathcal{O}_n}$ be the sequence of measures on {$\FC^{2} \times \NN$}, associated to the coarse packet $\mathcal{P}_{\mathcal{O}_n}$ via~\eqref{def:nu_O}.\linebreak
    If $\disc(\mathcal{O}_n)\to\infty$ as $n\to\infty$, then
    \[ \nu_{\order_n}(\varphi) \to \nu_{\FC}(\varphi), \]
    for all finitely supported functions $\varphi$ on $\FC^{2} \times \NN$.
\end{theorem}

In particular, the measure $\nu_{\FC}$ may be naturally obtained by a sequence of finitely supported measures on $\FC^{2}\times\NN$. On the one hand, \Cref{thm:periodic-equi} tell us that the periodic Klein sails behave as complicated as typical (non-periodic) Klein sails. On the other hand, \Cref{thm:periodic-equi} also implies that any given integer-affine type appears as the affine type of a face of some periodic Klein sail.

\begin{corollary}\label{cor:ex-of-simplest-int-type}
    Let $F$ be a totally real cubic field and $\mathcal{O}$ an order in $F$. If $\disc(\mathcal{O})$ is large enough, there
    is a periodic Klein sail $\Klein$ with $\order_\Klein = \order$, such that $\Klein$ contains a face of simplest integer-affine type, \ie a triangle with no interior integral points.
\end{corollary}

\Cref{cor:ex-of-simplest-int-type} gives further evidence for the conjecture \cite[Conjecture 12]{Ka17} that every Klein sail contains a face of simplest integer-affine type.

\subsection{Further discussion and open questions}

Despite substantial progress in the study of Klein sails, their fine face structure remains poorly understood. For instance, in dimension three, it is still unknown whether every Klein sail contains the simplest element of $\FC^{2}$, namely a triangle with no interior integral points. Similarly, it is not known whether every Klein sail necessarily contains a face of level greater than one, or whether there exist Klein sails all of whose faces have level one. In view of \Cref{cor:ex-of-simplest-int-type}, it would be interesting to characterize all integer-affine types that belong to the support of $\nu_{\FC}$.

Moreover, in this paper we mainly focus on cones generated by rays that are totally irrational.  In dimension $n=2$, the asymptotic normality of the continued fraction expansion of most rationals with a high denominator is known from the work of \citeauthor{DS18}~\cite{DS18}, thus, it is natural to expect that the higher dimensional analogue is true as well. More precisely, the asymptotic frequency of faces of \highlight{rational cones} in $\RR^3$, that is cones generated by three rational rays should be the same as for totally irrational ones. We intend to return to these questions in future work.

The present paper focuses exclusively on the two-dimensional faces of three-dimensional Klein sails. It would be interesting to obtain analogous results for edges and vertices. A further direction is to extend our methods to higher dimensions. Although the combinatorial structure of Klein sails becomes substantially more intricate as the dimension increases, we expect that the strategy developed here can be adapted to the higher-dimensional setting.

For further open problems concerning geometric multidimensional continued fractions, we refer the reader to \cite{Ka17}.

\subsection{Organization of the paper}

\Cref{sec:abstract-theory} briefly reviews the abstract theory of cross-sections.
\Cref{sec:explicit-crosssection} introduces the necessary notation for the space of lattices and explicitly constructs the cross-section for the diagonal flow associated with Klein sails.
Assuming the finiteness of the cross-sectional measure, \Cref{sec:proofs-of-main-thms} proves Theorems~\ref{thm:face-equi}, \ref{thm:Karpenkov question}, and~\ref{thm:periodic-equi}.

We now begin the technical core of the paper, namely, proving the finiteness of the cross-sectional measure. In \Cref{sec:Changing the viewpoint}, we reformulate this problem as the finiteness of the measure of a certain subset of $\GL_2(\RR)\ltimes\RR^2/\SL_2(\ZZ)\ltimes\ZZ^2$, an infinite volume homogeneous space naturally identified with the space of two-dimensional affine lattices. We then decompose this subset into a countable family of sets $\{\Cclv\}_{k\in\NN}$ indexed by the level reducing the problem to estimating the measures of these sets.

\Cref{sec:strategy-msr-estimate} studies the geometry of the sets $\Cclv$ and outlines the strategy for obtaining the required measure estimates. It states upper (\Cref{prop:upper-bound-level-measure}) and lower (\Cref{prop:lower-bound-level-measure}) bounds on the set of affine lattices $r^{1/2}A'\ZZ^2+b$ in the cross-section for a fixed unimodular lattice $A'\ZZ^2$ and shows how they imply the desired finiteness result. The remainder of the paper is devoted to proving these two propositions.

\Cref{sec:bounds-on-level} relates the successive minima of the lattice $A\ZZ^2$ to the level of the affine lattice $A\ZZ^2+b$, thereby obtaining the first restrictions on the possible lattices corresponding to a given level.
\Cref{sec:witnesses} establishes congruence restrictions on the lattice points that can occur on faces of Klein sails, while \Cref{sec:reduced-minima} proves the existence of certain configurations on the faces of Klein sails for every level $k\ge2$. Together, these results significantly restrict the possible configurations of lattice points appearing on a face and play a central role in the measure estimates. \Cref{sec:auxiliary-results} collects several auxiliary results, including properties of Farey fractions and elementary geometric facts concerning intersections of triangles, that are used in the proofs.

Finally, \Cref{sec:proof-level-measure} proves \Cref{prop:upper-bound-level-measure,prop:lower-bound-level-measure} by implementing the strategy developed in \Cref{sec:strategy-msr-estimate} together with the results established in \Cref{sec:witnesses,sec:reduced-minima,sec:auxiliary-results}.

\subsection*{Acknowledgement}
The authors thank Menny Aka, Manfred Einsiedler and Alexander Gorodnik for various helpful discussions during the preparation of this paper.


\section{Cross-sections: Abstract Theory}\label{sec:abstract-theory}

This section briefly summarizes the theory of cross-sections, which serves as the main tool of the paper. The exposition is largely based on~\cite{SW24, aggarwalghosh2024joint, AG24Levy}; see also~\cite{CC19,AN93,AthreyaCheung, Marklof2010, MarStro, Nadkarni, Ambrose, Ambrose_Kakutani, Wagh, AGh25, AG26}. Readers already familiar with the theory may skip this section on a first reading.

We begin by recalling the notion of vague convergence of measures.

\begin{definition}[Vague convergence]
\label{def:vague}
Let $X$ be a locally compact second countable topological space. A sequence of Radon measures
$\{\mu_k\}$ on $X$ converges \emph{vaguely} to a Radon measure $\mu$ if
\[
\int_X \varphi\,d\mu_k \longrightarrow \int_X \varphi\,d\mu
\]
for every $\varphi\in C_c(X)$, where $C_c(X)$ is the space of continuous compactly supported functions on $X$.
\end{definition}

\begin{remark}
If $\mu_k$ and $\mu$ are probability measures, then vague convergence is equivalent to weak-* convergence. We recall that the latter is defined by replacing the function space $C_c(X)$ with $C_0(X)$, the space of continuous functions on $X$ vanishing at infinity, in the above definition.
\end{remark}

\subsection{Cross-sections in a manifold}

\begin{definition}[Cross-section]
\label{def:cross-section}

Let $X$ be a smooth manifold of dimension $m$, equipped with a smooth flow $\{a_t : t \in \RR^n\}$. A subset $\cS \subset X$ is called a cross-section for the action if for every $x \in \cS$, there exists an open neighbourhood $U_x \subset X$ of $x$, a constant $\gamma_x > 0$, and a diffeomorphism
\[
\psi_x : U_x \to V_x \subset \RR^{m-n} \times \RR^n,
\]
onto an open neighbourhood $V_x\subset\RR^m$ of the origin such that
\[
\psi_x(\cS \cap U_x)
=
V_x \cap (\RR^{m-n} \times \{0\}),
\]
and for every $y \in \cS \cap U_x$ and every $t \in [-\gamma_x,\gamma_x]^n$ satisfying $a_t y \in U_x$, one has
\[
\psi_x(a_t y)
=
(\pi(\psi_x(y)), t),
\]
where
\[
\pi \colon \RR^{m-n} \times \RR^n \to \RR^{m-n}
\]
denotes the projection onto the first factor. We call $(U_x,\psi_x,\gamma_x)$ as above a flow-box chart.
\end{definition}

\begin{remark}
The definition immediately implies that $\cS$ is an embedded submanifold of $X$ of codimension $n$. Moreover, every open subset of a cross-section is again a cross-section. In particular, the empty set is also a cross-section.

This differs from the classical notion of a global cross-section, where one additionally requires that for every $x \in X$, the set $\{t \in \RR^n : a_t x \in \cS\}$ is non-empty and discrete.  In our setting, discreteness follows directly from the local flow-box structure, whereas the non-empty condition is not required.

\end{remark}

\begin{proposition}
\label{prop:def cross measure}

Let $\cS$ be a cross-section for the flow $\{a_t : t \in \RR^n\}$ on $X$, and let $\mu$ be an $a_t$-invariant Radon probability measure on $X$.
Then there exists a Radon measure $\mu_{\cS}$ on $\cS$ such that for every $x \in \cS$ and every associated flow-box chart $\psi_x : U_x \to V_x \subset \RR^{m-n} \times \RR^n$, one has
\[
(\psi_x)_*(\mu|_{U_x})
=
\left(
(\pi \circ \psi_x)_*(\mu_{\cS}|_{\cS \cap U_x})
\otimes m_{\RR^n}
\right)\Big|_{V_x},
\]
where $m_{\RR^n}$ denotes the Lebesgue measure on $\RR^n$.

Equivalently, in local flow-box coordinates, the measure $\mu$ decomposes as the product of the transverse measure $\mu_{\cS}$ and the Lebesgue measure along the flow directions.

\end{proposition}
\begin{proof}
We construct $\mu_{\cS}$ locally in flow-box charts and then verify compatibility.

Fix $x \in \cS$ and let $(U_x,\psi_x,\gamma_x)$ be a flow-box chart as in Definition~\ref{def:cross-section}. Write
\[
\psi_x : U_x \to V_x \subset \RR^{m-n} \times \RR^n, 
\qquad
\psi_x(y) = (z,t).
\]
By the defining property of the cross-section, every $y \in U_x$ can be uniquely written as
\[
y = a_t y_0,
\qquad y_0 \in \cS \cap U_x,
\]
with $t$ sufficiently small and in these coordinates the flow is a translation in the $\RR^n$-factor.

Define a measure $\nu_x$ on $\pi(V_x) \subset \RR^{m-n}$ by
\[
\nu_x(A)
:=
\frac{1}{(2\gamma_x)^n}\mu\big(\psi_x^{-1}(A \times [-\gamma_x,\gamma_x]^n)\big),
\]
for Borel sets $A \subset \pi(V_x)$. Note that by definition of $\gamma_x$, we have  
\[ A \times [-\gamma_x,\gamma_x]^n \subset U_x. \]
Now we use the $a_t$-invariance of $\mu$. Since the flow acts by translation in the $\RR^n$-coordinate in these charts, it follows from Fubini's theorem that for measurable $B \subset V_x$ of product form $B = A \times I$,
\[
\mu(\psi_x^{-1}(A \times I))
=
\nu_x(A)\, m_{\RR^n}(I),
\]
where $m_{\RR^n}$ is the Lebesgue measure.
This shows that in local coordinates,
\[
(\psi_x)_*(\mu|_{U_x})
=
(\nu_x \otimes m_{\RR^n})|_{V_x}.
\]

It remains to identify $\nu_x$ with the pushforward of a measure on $\cS$. By construction, $\nu_x$ depends only on sets of the form $A \times \{0\}$ under $\psi_x$, and the identification
\[
\cS \cap U_x \xrightarrow{\ \psi_x\ } \pi(V_x) \times \{0\}
\]
allows us to define a measure $\mu_{\cS}$ on $\cS$ by
\[
(\pi \circ \psi_x)_*(\mu_{\cS}|_{\cS \cap U_x}) = \nu_x.
\]

The consistency on overlaps of flow-boxes follows from $a_t$-invariance of $\mu$ and the fact that transition maps between charts preserve the product structure up to translation in the $\RR^n$-direction.

Hence, for every $x \in \cS$,
\[
(\psi_x)_*(\mu|_{U_x})
=
\left(
(\pi \circ \psi_x)_*(\mu_{\cS}|_{\cS \cap U_x})
\otimes m_{\RR^n}
\right)\Big|_{V_x},
\]
as required.
\end{proof}

\begin{lemma}\label{lem:conv-in-crosssection}
Suppose $\cS$ is a cross-section for the flow $\{a_t: t \in \RR^n\}$ on $X$. 
Assume $(\mu_k)_{k \in \NN}$ is a sequence of $a_t$-invariant finite Radon measures converging to $\mu$ in vague topology. 
Then the associated cross-sectional measures $\mu_{k,\cS}$ converge vaguely to $\mu_{\cS}$.
\end{lemma}

\begin{proof}
Fix $x \in \cS$ and choose a flow-box chart $(U_x,\gamma_x,\psi_x)$ as in Definition~\ref{def:cross-section}. Write
\[
\psi_x: U_x \to V_x \subset \RR^{m-n} \times \RR^n,
\qquad \psi_x(y) = (z,t).
\]

By \Cref{prop:def cross measure} we have for each $k$ the local product decomposition
\begin{align*}
    (\psi_x)_*(\mu_k|_{U_x})
=
(\nu_{k,x} \otimes m_{\RR^n})|_{V_x}, \qquad 
&\nu_{k,x} \defeq (\pi \circ \psi_x)_*(\mu_{k, \cS}|_{\cS \cap U_x}), \\
(\psi_x)_*(\mu|_{U_x})
=
(\nu_x \otimes m_{\RR^n})|_{V_x}, \qquad &\nu_{x} \defeq (\pi \circ \psi_x)_*(\mu_{\cS}|_{\cS \cap U_x}),
\end{align*}
Since $\mu_k \to \mu$ in the vague topology, we also have
\[
(\nu_{k,x} \otimes m_{\RR^n})|_{V_x} = (\psi_x)_*(\mu_k|_{U_x}) \;\to\; (\psi_x)_*(\mu|_{U_x})= (\nu_x \otimes m_{\RR^n})|_{V_x}
\quad \text{vague on } V_x.
\]
Since the supports of $\nu_{k,x}$ and $\nu_x$ are contained in $\pi(V_x)$, and since $V_x$ contains
\[
\pi(V_x) \times (-\gamma_x,\gamma_x)^n,
\]
it follows that $\nu_{k,x} \to \nu_x$ in the vague topology.
This further implies vague convergence of $\mu_{k,\cS}$ to $\mu_{\cS}$ on $\cS \cap U_x$.

Since $x \in \cS$ was arbitrary and these local measures agree on overlaps, we conclude that
\[
\mu_{k,\cS} \xrightarrow[k\to\infty]{} \mu_{\cS}
\]
in the vague topology of Radon measures on $\cS$.
\end{proof}

\begin{lemma}
\label{lem:equi-cross-section}
Suppose $\cS$ is a cross-section for the flow $\{a_t: t \in \RR^n\}$ on $X$, and let $J \subset \RR^n$ be a bounded set with smooth boundary.
Let $x \in X$, let $(T_k)_{k \in \NN}$ be an increasing sequence with $T_k \to \infty$, and assume that the probability measures
\[
\mu_k := \frac{1}{|J|} \int_J \delta_{a_{T_k t}x}\dd t
\]
converge in the vague topology to a probability measure $\mu$.
Then
\[
\frac{1}{T_k^n|J|} \sum_{\{r \in J: a_{T_k r}x \in \cS\}} \delta_{a_{T_k r}x}
\;\xrightarrow[k\to\infty]{vague}\;
\mu_{\cS}.
\]
\end{lemma}
\begin{proof}
It suffices to test against nonnegative functions $\varphi \in C_c(\cS)$ supported in a single flow-box.

Fix $x \in \cS$ such that $\supp(\varphi) \subset U_x$, where $(U_x,\gamma_x,\psi_x)$ is the flow-box chart as in Definition~\ref{def:cross-section}. Write
\[
\psi_x: U_x \to V_x \subset \RR^{m-n} \times \RR^n,
\qquad
\psi_x(y) = (z,t).
\]
Let $\eta \in C_c(\RR^n)$ satisfy
\[
\mathrm{supp}(\eta) \subset (-\gamma_x,\gamma_x)^n,
\qquad
\int_{\RR^n}\eta(t)\dd t = 1.
\]
Define $F \in C_c(U_x)$ by
\[
F(a_t y) := \varphi(y)\,\eta(t),
\qquad
y \in \cS \cap U_x,\ \|t\|_\infty < \gamma_x,
\]
and extend $F$ by zero outside of $U_x$.
By Proposition~\ref{prop:def cross measure}, we have
\[
\int_X F\dd\mu = \int_{\cS} \varphi\dd\mu_{\cS}.
\]
Now observe, using the definition of $\mu_k$, that
\[
\int_X F\dd\mu_k
=
\frac{1}{|J|}\int_J F(a_{T_k t}x)\dd t.
\]
Using the flow-box structure, each $t$ with $a_{T_k t}x$ in $\supp(F)$ can be uniquely written as
\[
 t = s+r,
\qquad a_{T_kr} x \in \cS \cap U_x,\ \|T_ks\|_\infty<\gamma_x.
\]
Hence, we may rewrite
\[
F(a_{T_k t}x)
=
\varphi( a_{T_kr} x)\,\eta(T_ks).
\]
Therefore,
\begin{align}\label{eq:int_F}
\begin{split}
    \int_J F(a_{T_k t}x)\dd t
&=
\sum_{\{r :\ a_{T_k r}x \in \cS\}}
\varphi(a_{T_k r}x) \int_{\{s: \   \|T_ks\|_\infty \leq \gamma_x, \  s+r \in J\}} \eta(T_k s)\dd s\\
&= \sum_{\{r \in J :\ a_{T_k r}x \in \cS\}}
\varphi(a_{T_k r}x) \int_{\{s: \   \|T_ks\|_\infty \leq \gamma_x \}} \eta(T_k s)\dd s+
o(1).
\end{split}
\end{align}
To prove the last equality, note that if 
\[
\{s: \ \|T_k s\|_\infty \leq \gamma_x,\ r+s \in J\}=\{s: \   \|T_ks\|_\infty \leq \gamma_x \},
\]
then necessarily $r \in J$ since $s=0$ is contained in the right hand side, and the corresponding terms in the two expressions in \eqref{eq:int_F} coincide. Similarly, if
\[
\{s: \ \|T_k s\|_\infty \leq \gamma_x,\ r+s\in J\}=\emptyset,
\]
then necessarily $r \notin J$, and the contribution vanishes in both expressions.
Thus the discrepancy comes only from those $r$ for which 
\[
\{s: \ \|T_k s\|_\infty \leq \gamma_x,\ r+s\in J\}
\]
is neither empty nor equal to $\{s: \   \|T_ks\|_\infty \leq \gamma_x \}$. Equivalently, these are precisely the points $r$ such that
\[ r+\{s: \   \|T_ks\|_\infty \leq \gamma_x \}
\]
intersects both $J$ and $J^c$. Hence the set of all such points $r+s$ is contained in the $2\gamma_x/T_k$-neighbourhood of $\partial J$. Since $\partial J$ is smooth, the measure of this neighbourhood tends to zero as $k\to\infty$. Because $\varphi$ and $\eta$ are bounded, the total contribution of these boundary terms is $o(1)$, which proves the equality.

Thus, using that $\int_{\RR^n}\eta(t)\dd t = 1$, we obtain
\[
\int_X F\dd\mu_k
=
\frac{1}{T_k^n |J|}
\sum_{\{r \in J : a_{T_k r}x \in \cS\}}
\varphi(a_{T_k r}x)
+ o(1).
\]
By assumption $\mu_k \to \mu$, so
\[
\int_X F\dd\mu_k \longrightarrow \int_X F\dd\mu.
\]
Combining with the identity for $\mu$, we obtain
\[
\frac{1}{T_k^n |J|}
\sum_{\{r \in J : a_{T_k r}x \in \cS\}}
\varphi(a_{T_k r}x)
\longrightarrow
\int_{\cS} \varphi\dd\mu_{\cS}.
\]
This proves vague convergence of the measures and completes the proof. 
\end{proof}

\begin{lemma}
\label{lem:equi-cross-section 2}

Suppose $\cS$ is a cross-section for the flow $\{a_t:t\in\RR^n\}$ on $X$, let $J\subset\RR^n$ be a bounded set with smooth boundary, and let $\mu$ be an ergodic $\{a_t:t\in\RR^n\}$-invariant probability measure on $X$ such that the associated cross-sectional measure $\mu_{\cS}$ is finite.
Then there exists an $\{a_t:t\in\RR^n\}$-invariant subset $Y\subset X$ of full $\mu$-measure such that for every $x\in Y$ and for every bounded continuous function $\varphi$ on $\cS$, we have
\begin{align}
    \label{eq:lem:equi-cross-section}
\frac{1}{T^n|J|}
\sum_{\{t\in J:a_{Tt}x\in\cS\}}
\varphi(a_{Tt}x)
\xrightarrow[T\to\infty]{}
\int_{\cS}\varphi\,d\mu_{\cS}.
\end{align}
\end{lemma}

\begin{proof}
Since $\mu_{\cS}$ is finite, it is enough to prove the claim for functions in $\RR\cdot\mathbf 1+C_c(\cS)$.

By the pointwise ergodic theorem (e.\,g.,~ see \cite[Chap.~8]{EW}) there exists an
$\{a_t:t\in\RR^n\}$-invariant subset $Y_1\subset X$ of full
$\mu$-measure such that for every $x\in Y_1$,
\[
\frac1{|J|}
\int_J\delta_{a_{Tt}x}\,dt
\xrightarrow[T\to\infty]{vague}
\mu.
\]
Hence, Lemma~\ref{lem:equi-cross-section} implies that for every
$f\in C_c(\cS)$ and every $x\in Y_1$, the equation \eqref{eq:lem:equi-cross-section} holds.

It therefore remains to prove the claim for the constant function
$\mathbf1$.
Define $F: X \rightarrow \NN$ as 
$$
F(x)= \#\Phi^{-1}(x),
$$
 where $\Phi:\cS\times[-1,1]^n\to X$ is given by $(y,t)\mapsto a_t y $.
Using the local product structure of \(\mu\) in flow boxes, namely that
locally the measure decomposes as
\[
d\mu = d\mu_{\cS}\,dm_{\mathbb R^n},
\]
we get that
$$
\int_X \#\Phi^{-1}(x)\, d\mu(x)=\int_{\cS\times[-1,1]^n} \mathbf{1}\,  d\mu_{\cS}(y)dm_{\RR^n}(t).
$$
Hence, we obtain
\[
\int_X F\,d\mu
=
m_{\mathbb R^n}([-1,1]^n) \,\mu_{\cS}(\cS)
=
2^n\mu_{\cS}(\cS)
<\infty.
\]
Hence \(F\in L^1(X,\mu)\).

Applying the pointwise ergodic theorem once more, there exists an
$\{a_t:t\in\RR^n\}$-invariant subset $Y_2\subset X$ of full
$\mu$-measure such that
\[
\frac1{|J|}
\int_JF(a_{Tt}x)\,dt
\longrightarrow
\int_XF\,d\mu,
\qquad
x\in Y_2.
\]
Arguing exactly as in the proof of
Lemma~\ref{lem:equi-cross-section}, we obtain
\[
\int_JF(a_{Tt}x)\,dt
=
\frac{2^n}{T^n}
\#\{t\in J:a_{Tt}x\in\cS\}
+o(1).
\]
Combining this with the convergence above yields
\[
\frac1{T^n|J|}
\#\{t\in J:a_{Tt}x\in\cS\}
\longrightarrow
\frac1{2^n}\int_XF\,d\mu
=
\mu_{\cS}(\cS),
\]
which is precisely the required convergence for the constant function.

Finally, set $Y:=Y_1\cap Y_2$, which is the required
$\{a_t:t\in\RR^n\}$-invariant set of full $\mu$-measure.
\end{proof}

\subsection{Cross-section in homogeneous space}

\begin{lemma}
\label{lem:cross hom space}
Let $X = G/\Gamma$, where $G$ is a connected Lie group and $\Gamma < G$ is a lattice. 
Let $\{a_t : t \in \RR^n\}$ be an abelian Lie subgroup of $G$.
Let $H < G$ be a closed Lie subgroup such that
\[
\dim H = \dim G - n,
\qquad
\{a_t : t \in \RR^n\}\cap H = \{e\}.
\]
Assume that for some $x \in X$, the orbit $H x$ is closed and that $\cS \subset H x$ is relatively open in $H x$, i.e. for every $y \in \cS$, there exists a neighbourhood $O_y \subset H$ of the identity such that $O_y y \subset \cS$.
Then $\cS$ is a cross-section for the flow $\{a_t : t \in \RR^n\}$ on $X$.

Moreover, if $\mu$ denotes the $G$-invariant probability measure on $X$, then the associated cross-sectional measure $\mu_{\cS}$ is, up to normalization, the restriction to $\cS$ of the measure on $H x$  induced by the right Haar measure on $H$ under the orbit map $h \mapsto hx$.
\end{lemma}

We will use the following lemma to prove Lemma~\ref{lem:cross hom space}.
\begin{lemma}[{\cite[Lem.~11.31]{EW}}]\label{lem:Haar-product} 
        Let $L$ be a Lie group and let $L_1$, $L_2$ be closed subgroups such that $L_1\cap L_2 = \{e\}$ and $\dim L_1 + \dim L_2 = \dim L$. Then
        \begin{enumerate}
            \item the map $p\colon L_1\times L_2 \to L$ given by $p(\ell_1,\ell_2) = \ell_1\ell_2$ is a diffeomorphism onto an open subset $\mathcal{U} \subseteq L$.
            \item If furthermore $L$ is unimodular, and $\msr_{L_1}^{\mathrm{left}}$, $\msr_{L_2}^{\mathrm{right}}$ denote the left and right Haar measures on $L_1$, $L_2$, respectively, then $p_*(\msr_{L_1}^{\mathrm{left}}\times\msr_{L_2}^{\mathrm{right}})$ is proportional to the restriction to $\mathcal{U}$ of a Haar measure on $L$.
        \end{enumerate}
    \end{lemma}

\begin{proof}[Proof of Lemma~\ref{lem:cross hom space}]
Fix $x\in X$ as in the statement of the lemma. Then the map
\[
H \to H x, \qquad h \mapsto hx
\]
is a local diffeomorphism, and the assumption that $H x$ is closed implies that it is an immersed submanifold of $X$.

We first construct local product coordinates for the $H$-action and the $\{a_t : t \in \RR^n\}$-action. 
By assumption,
\[
\dim \{a_t : t \in \RR^n\} + \dim H = \dim G,
\qquad
\{a_t : t \in \RR^n\} \cap H = \{e\}.
\]
Hence, the hypotheses of Lemma~\ref{lem:Haar-product} apply for $L=G$, $L_1=\{a_t : t \in \RR^n\}$ and $L_2=H$, giving an open neighbourhood $\mathcal{U} \subset G$ of the identity and a diffeomorphism
\[
\RR^n \times H \to \mathcal{U}, \qquad (t,h) \mapsto a_t h.
\]
Composing with the map $G  \to X$, $(g,x) \mapsto gx$, we obtain a smooth map
\[
\Phi\colon \RR^n \times H \to X, \qquad (t,h) \mapsto a_t h x.
\]
This map is a local diffeomorphism near $(0,e)$, and hence, after translation, near every point $(0,h)$ with $h x \in H x$.

Since $\cS$ is relatively open in $H x$, for each $y = h x \in H x$, there exists a neighbourhood of $y$ such that $\Phi$ identifies $X$ locally around $y$ with a product neighbourhood of $(0,h)$ in $\RR^n \times H$ in such a way that:
\begin{itemize}
    \item the $\{a_t : t \in \RR^n\}$-orbit corresponds to the $\RR^n$-factor,
    \item the $H$-orbit in $\RR^n\times H$ corresponds to the local $H$-orbit around $y$ in $\cS$.
\end{itemize}
This verifies the local product structure required in the definition of a cross-section.

We now identify the cross-sectional measure. Since $G$ is unimodular (as it contains a lattice), the Haar measure $m_G$ on $G$ descends to a $G$-invariant probability measure $\mu$ on $X$. 

By Lemma~\ref{lem:Haar-product}, the product of the Lebesgue measure on $\RR^n$ and the right invariant Haar measure $m_H$ on $H$ pushes forward under $(t,h) \mapsto  a_th$ to a measure on $G$ which is proportional to $m_G$. Pushing further to $X = G/\Gamma$, we obtain that locally
\begin{align}
d\mu \;\asymp\;   dm_{\RR^n} \otimes d\mu_{H x},
    \label{eq:lem: cross hom space 1}
\end{align}
where $\mu_{H x}$ is the measure on $H x$ induced from the fixed right Haar measure $m_H$ on $H$, and $m_{\RR^n}$ is the Lebesgue measure along the $\{a_t : t \in \RR^n\}$-orbits. Note that the proportionality constant in \eqref{eq:lem: cross hom space 1} is the same as the one appearing while comparing $m_G$ with $m_{\RR^n} \otimes m_H$, and hence is independent of the point $x$ inside the orbit $Hx$.

By definition of the cross-sectional measure, $\mu_{\cS}$ is obtained by restricting the transverse factor of this product structure to $\cS$. Hence, $\mu_{\cS}$ is proportional to the restriction of $\mu_{H x}$ to $\cS$. 
\end{proof}


\section{Cross-section for the Klein Sails}\label{sec:explicit-crosssection}

Given three linearly independent rays $L_1,L_2,L_3$ with generating vectors $w_1,w_2,w_3\in\RR^3$, that is $L_i = \RR_{>0}w_i$ for $i=1,2,3$, they define a simplicial cone via
\[ C_{L_1,L_2,L_3} = \{c_1w_1+c_2w_2+c_3w_3\in\RR^3~:~ c_1,c_2,c_3\geq0 \}. \]
The associated Klein sail $\Klein$ is defined as the boundary of the convex hull of $C_{L_1,L_2,L_3}\cap\ZZ^3$. Notice that the definition of the Klein sail does not depend on the choice of the vectors $w_i \in L_i$ as any other choice will give the same simplicial cone $C_{L_1,L_2,L_3}$. In particular, for a face $f$ of the Klein sail $\Klein$ with the associated affine plane $V_f$ containing $f$, we may use the vectors $w_i \defeq V_f\cap L_i$ for $i=1,2,3$ to define the cone $C_{L_1,L_2,L_3}$. 

Let
\[
\widetilde{g}_f = \begin{pmatrix} w_1 & w_2 & w_3 \end{pmatrix} \in \GL_3(\RR) 
\]
be the (invertible) $(3\times 3)$-matrix whose columns are given by the vectors $w_i$ and define the normalized matrix 
$$
g_f \defeq \det(\widetilde{g}_f)^{-1/3}\widetilde{g}_f \in \SL_3(\RR).
$$
Then we have $C_{L_1,L_2,L_3} = \widetilde{g}_f\RR^3_{\geq0} = g_f\RR^3_{\geq0}$. Since we are only interested in the integral affine-type of the face $f$, this type will only depend on the $\SL_3(\ZZ)$-coset representative of the matrix $\widetilde{g}_f$.
This suggests to move to the \enquote{dual picture} of Klein sails by considering the following set up.

\medskip
Let
\[
G = \SL_3(\RR),
\qquad
\Gamma = \SL_3(\ZZ),
\qquad
X_3 \defeq \tquot{G}{\Gamma}.
\]
Recall that $X_3$ is identified with the space of unimodular lattices in $\RR^3$ via $g\Gamma \mapsto g\ZZ^3$, and let $\msr_{X_3}$ denote the unique $G$-invariant probability measure on $X_3$.
Let
$$
a_t= \begin{pmatrix}
    e^{t_1} \\ & e^{t_2} \\ && e^{t_3}
\end{pmatrix}, \qquad t= (t_1, t_2, t_3),
$$
and
\begin{align}
    \label{eq:def diagonal group}
    A_+
\defeq
\left\{ a_t
\st t=(t_1,t_2, t_3) \text{ satisfy }
t_1+t_2+t_3=0
\right\}
\leq G.
\end{align}
Thus, \(A_+\) is a two-parameter flow on the space \(X_3\) and can be
identified with the group
\[
\{(t_1,t_2,t_3)\in\RR^3:t_1+t_2+t_3=0\}.
\]
Throughout this paper, for a subset \(K\) of the above group, the notation
\(|K|\) will be used to denote the Haar measure of \(K\), where the unique
choice of measure is given by the equation
\[
\left|
\left\{
(t_1,t_2,t_3)\in\mathbb R^3:
t_1+t_2+t_3=0,\;
|t_i-t_j|\leq1
\text{ for all }i\neq j
\right\}
\right|
=1.
\]
The same normalized measure will also be used to define the cross-sectional
measure on \(X_3\) for the \(A_+\)-flow, defined as in
\Cref{prop:def cross measure}.

Define
\[
\v \defeq 3^{-1/2}(1,1,1)\in\RR^3,
\]
and let $\Hv$ denote the hyperplane orthogonal to $\v$.
For a subset $\Lambda\subset\RR^3$, we define
\[
\Lambda_+ \defeq \Lambda\cap\RR^3_{>0}, \quad \Lambda_{\geq 0}\defeq \Lambda\cap\RR^3_{\geq 0}
\]
For such $\Lambda$, define
\begin{align}
    \label{eq: def s lambda}
    s_\Lambda
\defeq
\inf\{s>0\st (\Lambda_{\geq 0}-s\v)\cap\Hv\neq\emptyset\}.
\end{align}
Equivalently,
\begin{align}\label{eq: def s lambda 2}
s_\Lambda
=
\inf\{\langle \v,w\rangle \st {w\in\Lambda_{\geq 0} \setminus \{0\}} \}.
\end{align}
Note that if $\Lambda$ is a lattice in $\RR^3$, then $0 < s_\Lambda < \infty$. Indeed, finiteness follows from the fact that $s_\Lambda=\infty$ if and only if $\Lambda_+=\emptyset$, while positivity follows from the fact that $\Lambda_+$ has only finitely many points in any bounded subset of $\RR^3$.

\begin{definition}
    We now define the set $\C\subset X_3$ by 
\begin{align}\label{eq:def-Klein-Sail-crosssection}
    \C \defeq \left\{ \Lambda\in X_3\;\middle|\; \text{ the affine $\ZZ$-span of }(\Lambda_+-s_\Lambda\v)\cap\Hv \text{ is an affine lattice in }\Hv,\right\}.
\end{align}
Recall that the affine $\ZZ$-span of a set $S$ is defined as set of all points of the form
\[ c_1 v_1 + \ldots + c_n v_n, \]
where $n \in \NN$, $v_1, \ldots, v_n \in S$ and the integers $c_1, \ldots, c_n\in\ZZ$ satisfy $c_1+ \ldots + c_n= 1$.

\end{definition}
\begin{remark}
    Geometrically, a lattice $\Lambda$ belongs to $\C$ if and only if there exists $s_\Lambda>0$ such that the affine hyperplanes
\[
s\v+\Hv,
\qquad s>0,
\]
satisfy the following properties:
\begin{enumerate}
\item
The hyperplane $s\v+\Hv$ does not intersect
$\Lambda\cap\RR_{\geq0}^3$ for any $0<s<s_\Lambda$.

\item
At the \enquote{first-contact time} $s_\Lambda$, the intersection
\[
(\Lambda\cap\RR_{>0}^3)\cap(s_\Lambda\v+\Hv)
\]
spans an affine lattice in $\Hv$.
Equivalently, the points are not contained in an affine line.
\end{enumerate}
\end{remark}

For $\Lambda \in \C$, define
\[
\Lambda_0 \defeq \Lambda \cap \Hv,
\qquad
\affLambda
\defeq
(\Lambda - s_\Lambda v)\cap \Hv.
\]
Note that $\affLambda$ is equal to the affine $\ZZ$-span of
$(\Lambda_+-s_\Lambda\v)\cap\Hv$. This follows from the fact that
$\Lambda-s_\Lambda\v$ is a two-dimensional lattice and
$(\Lambda_+-s_\Lambda\v)\cap\Hv$ is a convex lattice polygon. More precisely,
suppose that its affine $\ZZ$-span is a proper finite-index sublattice of
$\Lambda-s_\Lambda\v$. Let $a_1,a_2,a_3$ be three non-collinear points of
$(\Lambda_+-s_\Lambda\v)\cap\Hv$ whose affine $\ZZ$-span is this sublattice.
Then the fundamental parallelogram
\[
a_1+[0,1](a_2-a_1)+[0,1](a_3-a_1)
\]
contains a lattice point $x$ of $\Lambda-s_\Lambda\v$ which does not belong
to this sublattice. Write
\[
x=a_1+c_1(a_2-a_1)+c_2(a_3-a_1),
\qquad 0\leq c_1,c_2\leq1.
\]
A direct computation shows that either $x$ or $a_2+a_3-x$ must lies in the triangle with vertices
$a_1,a_2,a_3$. Since both points are lattice points and
are not contained in the assumed affine sublattice, this contradicts the
definition of the affine $\ZZ$-span of
$(\Lambda_+-s_\Lambda\v)\cap\Hv$. Therefore, the affine $\ZZ$-span must be the
whole lattice $\affLambda$.

Also, observe that $\affLambda + s_\Lambda v$ is precisely the set obtained at the \enquote{first-contact slice}
\[
\Lambda\cap (s_\Lambda v+\Hv),
\]
and hence by definition of $\C$, the set $\affLambda$ is an affine lattice in $\Hv$. 
Moreover, the base lattice of $\affLambda$ is $\Lambda_0$. Indeed, if
$t\in\Hv$ satisfies $t+s_\Lambda v\in\Lambda$, then $\affLambda=t+\Lambda_0$. 
We also choose $t_\Lambda\in\Lambda$ such that
\[
\langle v,t_\Lambda\rangle
=
\inf\{
|\langle v,w\rangle| : w\in\Lambda\setminus\{0\}
\}.
\]
Although the choice of $t_\Lambda$ is not unique, two of such choices differ by an element of $\Lambda_0$.
Thus,
\begin{equation*}
    \Lambda = \Lambda_0+\ZZ t_\Lambda.
\end{equation*}
Since $\Lambda$ is unimodular, we have $\langle v,t_\Lambda\rangle=\covol(\Lambda_0)^{-1}$. We define the \highlight{level} of $\Lambda$ by
\begin{align}
    \level(\Lambda) \defeq& s_\Lambda\,\covol(\Lambda_0) \nonumber \\
    =&\frac{s_\Lambda} {\inf\{|\langle v,w\rangle| : w\in\Lambda\setminus\{0\}\}}. \label{eq:def level}
\end{align}
Since $\Lambda=\Lambda_0+\ZZ t_\Lambda$ and $\Lambda_0\subset\Hv$, we have $\{\langle v,w\rangle:w\in\Lambda\}=\{n\langle v,t_\Lambda\rangle:n\in\ZZ\}$.
Consequently,
\[
\level(\Lambda)
=
\min\Bigl\{
\ell\ge1:
(\Lambda_0+\ell t_\Lambda)\cap\RR_{\ge0}^3\neq\emptyset
\Bigr\}.
\]

Another useful interpretation of the level is that it equals the index of the
sublattice of $\Lambda$ generated by the vectors in
\[
\Lambda_+\cap (s_\Lambda v+\Hv),
\]
namely the vectors witnessing that $\Lambda\in\C$.

  The following is a key result needed in the proof of Theorem~\ref{thm:face-equi}, \ref{thm:Karpenkov question}, and~\ref{thm:periodic-equi}.

   \begin{theorem}\label{thm:main}
    The set $\C$ defined in \eqref{eq:def-Klein-Sail-crosssection} is a cross-section for the action of $A_+$ on $X_3$. Moreover, the cross-section measure $\mu_\C$ on $\C$ corresponding to the Siegel--Haar probability measure $\msr_{X_3}$ on $X_3$ (defined in Proposition~\ref{prop:def cross measure}) is finite, and satisfies the following estimate
    \begin{align}
        \label{eq:main:thm}
     \frac{1}{k^3}  \ll \mu_{\C}(\{\Lambda \in \C: \level(\Lambda) =k \}) \ll \frac{\ln k+1}{k^2}, 
    \end{align}
    where the implied constants are independent of $k$.
\end{theorem}

In this section, we prove the first assertion of Theorem~\ref{thm:main}. We postpone the proof of the finiteness of $\mu_\C$ and equation~\eqref{eq:main:thm} to Section~\ref{sec:Changing the viewpoint}. We will need the following notation and lemma for the proof.

    Define
    \begin{align}
        \label{eq:def P}
         P \defeq \left\{\begin{pmatrix} A & b\\  & \det(A)^{-1} \end{pmatrix}\st A\in\GL_2(\RR),~b\in\RR^2\right\} \leq G,
    \end{align}
    and let
    \begin{align}\label{eq:def-of-gv}
         \gv = \begin{pmatrix} -\tfrac{1}{\sqrt{2}}&\tfrac{1}{\sqrt{2}}&0\\ -\tfrac{1}{\sqrt{6}}&-\tfrac{1}{\sqrt{6}}&\tfrac{2}{\sqrt{6}}\\ \tfrac{1}{\sqrt{3}}&\tfrac{1}{\sqrt{3}}&\tfrac{1}{\sqrt{3}}\end{pmatrix}\in \SO_3(\RR),
    \end{align}
    The matrix $g$ satisfies 
    \[ \gv^{-1}\mathbf e_3 = \gv^t \mathbf e_3 = \v, \]
    where $\mathbf e_3= (0,0,1)$ denotes the third standard basis vector of $\RR^3$.

  \begin{proof}[Proof of the first part of Theorem~\ref{thm:main}]
Note that it is enough to show that $\gv\C$ is a relatively open subset of $P\Gamma$ in $X_3$. To see this, note that the claim implies that $\C$ is a relatively open subset of the $H$-orbit of $x$ in $X_3$, where
\[
H=\gv^{-1}P\gv,
\qquad
x=\gv^{-1}\Gamma.
\]
Since $A\cap H=\{e\}$, the result follows from Lemma~\ref{lem:cross hom space}.

To prove the claim, first notice that $\gv\Hv=\RR^2\times\{0\}$. Thus, for $\Lambda\in\C$,
\[
\gv\Lambda_0
=
\gv(\Lambda\cap\Hv)
=
\gv\Lambda\cap(\RR^2\times\{0\})
\]
is a lattice in $\RR^2\times\{0\}$. This immediately implies that
\[
\gv\Lambda
=
\begin{pmatrix}
A & b\\
& \det(A)^{-1}
\end{pmatrix}
\ZZ^3,
\]
for some $A\in\GL_2(\RR)$ and $b\in\RR^2$. Hence $\gv\Lambda\in P\Gamma$, and therefore $\gv\C\subset P\Gamma$.

To prove relative openness, fix $\Lambda\in\C$, and let
\[
\Lambda_+\cap(s_\Lambda v+\Hv)
=
\{w_1,\ldots,w_n\}.
\]
Define
\[
R_\Lambda
=
\{w\in\RR_{\ge0}^3:\langle w,v\rangle\le s_\Lambda\}.
\]
Since $\Lambda \in \C$, there exists a sufficiently small $\e>0$ such that the $2\e$-neighbourhood
\[
R_\Lambda^{(2\e)}
=
\{w\in\RR^3:d(w,R_\Lambda)\le 2\e\}
\]
intersects $\Lambda$ exactly in the set $\{0\}\cup (\Lambda_{\geq 0} \cap(s_\Lambda v+\Hv))$.

Since each $w_i$ lies in the open set $\RR_{>0}^3$, after possibly
decreasing $\e$, we may additionally assume that the $\e$-neighbourhood of
each $w_i$ is contained in $\RR_{>0}^3$ and so in particular does not intersect any coordinate plane.

Now let
\[
h
=
\gv^{-1}
\begin{pmatrix}
A & b\\
& \det(A)^{-1}
\end{pmatrix}
\gv
\in H.
\]
Since $\Lambda$ is discrete and $R_\Lambda^{(2\e)}$ is bounded, by
taking $h$ sufficiently close to the identity, we may ensure that $\|hw- w\|_{\infty} \leq \e$, for all $w \in \RR^3$ with $\|w\|_\infty \leq 2s_\Lambda$.

The latter, along with the choice of $\e$, implies that the $n$ points $\{hw_1,\dots,hw_n\}$ of $h\Lambda$ are contained in 
\[
R_{\Lambda}^{(\e)} \cap \RR_{>0}^3= \{w \in \RR_{> 0}^3: \langle w, \v \rangle \leq s_\Lambda +\e  \}.
\]
and that 
\begin{align*}
    h\Lambda \cap R_\Lambda^{(\e)} &\subset  h\left( \Lambda \cap R_\Lambda^{(2\e)} \right)\\
    &= h \left(  \{0\}\cup (\Lambda_{\geq 0} \cap(s_\Lambda v+\Hv)) \right)\\
    &\subset h \left( \{0\} \cup (\Lambda \cap(s_\Lambda v+\Hv)) \right)\\
    &= \{0\} \cup (h\Lambda \cap (s' v+\Hv)),
\end{align*}
for some $s' \geq s_\Lambda- \e$. The last equality follows from the fact that $h$ preserves $\Hv$ (since $h \in H$). 

Thus, the points
\[
hw_1,\ldots,hw_n
\]
lie in a common affine hyperplane parallel to $\Hv$. Since $h\Lambda\cap R_\Lambda^{(\e)}$ only contains $\{0\}$ and the points on the hyperplane $s'\,v+\Hv$, it follows that $s'$ is the smallest positive value
of $\langle v,w\rangle$ attained by any point in
$h\Lambda\cap\RR_{\ge0}^3\setminus\{0\}$. Hence $s_{h\Lambda}=s'$.

Furthermore, the points $hw_1,\ldots,hw_n$ remain in the interior of the
positive orthant. Therefore the affine $\ZZ$-span of the set
\[
\{hw_1,\ldots,hw_n\} - s_{h\Lambda} \v \subseteq ((h\Lambda)_+ - s_{h\Lambda} \v)\cap \Hv
\]
is still an affine lattice in $\Hv$. Thus, all defining conditions of $\C$ hold for $h\Lambda$, proving that $h\Lambda\in\C$.

Since $\Lambda\in\C$ was arbitrary, $\C$ is relatively open in $Hx$.
Equivalently, $\gv\C$ is relatively open in $P\Gamma$. This completes the
proof. 
\end{proof}

The proof above immediately yields the following consequence, which will be
used later in the proof of Theorem~\ref{thm:face-equi}.

\begin{corollary}
\label{cor:local-stability}
Let $\Lambda\in\C$ and assume that $\Lambda$ does not contain points on the coordinate axes. Then there exists a neighbourhood $U_\Lambda$ of the identity in $\gv^{-1}P\gv$ such that for every $h\in U_\Lambda$,
\[
(h\Lambda)_+\cap(s_{h\Lambda}v+\Hv)
=
h\bigl(\Lambda_+\cap(s_\Lambda v+\Hv)\bigr).
\]
\end{corollary}

\begin{proof}
In the notation of the preceding proof, for $h$ sufficiently close to the
identity, the assumption on $\Lambda$ not having points on the coordinate axes ensures that the only points of $h\Lambda$ contained in the neighbourhood
$R_\Lambda^{(\e)}$ are $0,\,hw_1,\ldots,hw_n$, where
\[
\Lambda_+\cap(s_\Lambda v+\Hv)
=
\{w_1,\ldots,w_n\}.
\]
Since no additional lattice points enter $R_\Lambda^{(\e)}$, the points
$hw_1,\ldots,hw_n$ are precisely the points realized in the first-contact
slice of $h\Lambda$. Therefore,
\[
(h\Lambda)_+\cap(s_{h\Lambda}v+\Hv)
=
\{hw_1,\ldots,hw_n\}
=
h\bigl(\Lambda_+\cap(s_\Lambda v+\Hv)\bigr),
\]
as claimed.
\end{proof}


\section{Proofs of Theorem~\ref{thm:face-equi},~\ref{thm:Karpenkov question} and~\ref{thm:periodic-equi} assuming~Theorem~\ref{thm:main}}\label{sec:proofs-of-main-thms}

\subsection{Proof of Theorem~\ref{thm:face-equi}}
    We will need the following result along with Theorem~\ref{thm:main} to prove Theorem~\ref{thm:face-equi}.

    \begin{theorem}[See {\cite[Thms.~2.1 $\&$ 2.7]{Bil68}} or {\cite[Chap.~4]{Bou04a}}]
\label{muJM}
    Suppose $X$ is a locally compact second countable Hausdorff space. Let $(\eta_l)_{l \in \NN}$ and $\eta$ be Radon measures on $X$ such that $\eta_l(\varphi)\rightarrow \eta(\varphi)$ as $l \rightarrow \infty$, for every bounded continuous function $\varphi$ on $X$. Assume that $\phi: X \rightarrow \RR$ is a bounded measurable function which is $\eta$-almost everywhere continuous. Then
    \[ \eta_l(\phi) \rightarrow \eta(\phi), \quad \text{ as } l \to \infty. \]
\end{theorem}

\begin{proof}[Proof of Theorem~\ref{thm:face-equi}]
Let $\pi_{12}$ and $\pi_3$ denote the natural projections from $ \mathbb R^3=\mathbb R^2\times\mathbb R$ onto $\mathbb R^2$ and $\mathbb R$, respectively.

Define the map
\[
\psi:\mathcal C\longrightarrow \FC^2\times\mathbb N
\]
by
\[
\psi(\Lambda)
=
\left(
\pi_{12}\bigl(p^{-1}\gv(\Lambda_+\cap(s_\Lambda \v+\Hv))\bigr),
\,
\pi_3\bigl(p^{-1}\gv(s_\Lambda\v)\bigr)
\right),
\]
where $\Lambda=\gv^{-1}p\Gamma\in\mathcal C$ with $p\in P$ as defined in \eqref{eq:def P}. Namely, the first component of $\psi(\Lambda)$ is the convex integral polytope obtained from the points in $\Lambda$ witnessing that $\Lambda$ is an element of the cross-section $\C$, whereas the second component of $\psi$ is equal to the level of $\Lambda$.

By Corollary~\ref{cor:local-stability}, the map $\psi$ is locally constant at every point of $O\cap\mathcal C$, where $O$ denotes the set of lattices in $X_3$ containing no point on any coordinate axis. Since $X_3\setminus O$ is an $A_+$-invariant subset of $X_3$ of measure zero, it follows from the definition of $\mu_{\mathcal C}$ that $\mu_{\mathcal C}(\mathcal C\setminus O)=0$.
Hence, $\psi$ is locally constant at $\mu_{\mathcal C}$-almost every point of $\mathcal C$.

Let $F:\FC^2\times\mathbb N\to\mathbb R$ be a bounded function. Since $\psi$ is locally constant almost everywhere, the composition $F\circ\psi$ is continuous at $\mu_{\mathcal C}$-almost every point of $\mathcal C$.

Combining this with Lemma~\ref{lem:equi-cross-section 2} (for $X= X_3$) and Theorem~\ref{muJM} (for $\phi= F \circ \psi$ and $\eta= \mu_\C$), we obtain that there exists an $A_+$-invariant full measure subset $Y$ of $X_3$ such that for all $\Lambda \in Y$, we have
\begin{align} \label{eq:proof-main}
    \frac{1}{T^2|J|}\sum_{\{t\in TJ:\,a_t\Lambda\in\C\}} F\circ\psi(a_t\Lambda) \longrightarrow \mu_{\mathcal C} (F\circ\psi), \qquad (T\to\infty),
\end{align}
where
\[
J = \left\{(t_1,t_2,t_3)\in\mathbb R^3:~t_1+t_2+t_3=0,\; |t_i-t_j|\leq1 \text{ for all }i\neq j \right\}.
\]

Further note that every element $g\in\SL_3(\mathbb R)$ admits a decomposition of the form
\[
g = ra \begin{pmatrix} v_1&v_2&v_3 \end{pmatrix}^{-1}, 
\]
where $a\in A_+$, $r\in\mathbb R$, and $v_1,v_2,v_3$ are unit vectors in $\mathbb S^2 = \{x\in\mathbb R^3:\|x\|_2=1\}$. Indeed, if 
$$
g^{-1} = \begin{pmatrix} w_1&w_2&w_3 \end{pmatrix}
$$
with linearly independent column vectors $w_1,w_2,w_3$ in $\RR^3$, then setting 
\begin{align*}
    v_i=\frac{w_i}{\|w_i\|_2}, \qquad i=1,2,3,
\end{align*}
and 
$$
r = (\|w_1\|_2\|w_2\|_2\|w_3\|_2)^{1/3}\in\RR
$$
one can choose $a\in A_+$ accordingly, to get the desired decomposition. 

Consequently using  $A_+$-invariance and full measurability of $Y$, we get that for Lebesgue-almost every choice of vectors $w_1,w_2,w_3\in\mathbb R^3$, the lattice 
\begin{align*}
\Lambda = r\begin{pmatrix} v_1&v_2&v_3 \end{pmatrix}^{-1} \mathbb Z^3 
\end{align*}
belong to $Y$ and hence, satisfies \eqref{eq:proof-main}, where as before
\begin{align*}
    v_i=\frac{w_i}{\|w_i\|_2}, \qquad i=1,2,3,
\end{align*}
and $r = (\|w_1\|_2\|w_2\|_2\|w_3\|_2)^{1/3}\in\RR$ is chosen so that $\Lambda$ is unimodular.

Now consider three rays $L_1,L_2,L_3$ generated by the unit vectors $v_1,v_2,v_3\in\mathbb{S}^2$, and let $C_{L_1, L_2, L_3}$ denote the cone defined by these rays and let $\Klein$ be the associated Klein sail. Further, define
\[ g=  r \begin{pmatrix} v_1&v_2&v_3 \end{pmatrix}^{-1}, \quad \Lambda = g\ZZ^3 \in X_3\]
where $r\in\RR>0$ is chosen so that $g\in\SL_3(\RR)$.

Consider a face $f$ of the Klein sail $\Klein$ and suppose that $d_v(f)\le e^T$. Let $V_f$ denote the affine plane containing $f$, let $n_f$ be the unit normal vector to $V_f$, such that $\langle n_f, x\rangle >0$ is constant for all $x\in V_f$, and let $E_f$ be the linear plane through the origin orthogonal to $n_f$. For $i=1,2,3$, we let
\[ P_i=V_f\cap L_i. \]
Moreover, we define
\[ t_0= \frac13\log\bigl(\|P_1\|_2 \|P_2\|_2 \|P_3\|_2\bigr), \]
and for $i=1,2,3$, we set
\[ t_i=-\log\|P_i\|_2+t_0. \]
By assumption, we have, $d_v(f)=\max_{i\neq j}\frac{\|P_i\|_2}{\|P_j\|_2}\le e^T$, which implies
\[
\max_{i\neq j}|t_i-t_j|
\le
T,
\]
that is, $t\defeq(t_1, t_2, t_3) \in  TJ$. Notice that for $i=1,2,3$ we have $P_i = \norm{P_i}[2]v_i$ and so by definition of $g$, we get that
\[ a_tgP_i = r\norm{P_i}[2] a_t\begin{pmatrix} v_1 & v_2 & v_3 \end{pmatrix}^{-1} v_i\ = r\exp(t_0)e_i, \]
so the three points $a_tgP_i$ lie in a common translate of $\Hv$. Consequently,
\[ a_tgn_f=\alpha\v \]
for some $\alpha\in\mathbb R_{>0}$, and
\[ a_tgE_f= \Hv. \]

Since $f$ is a face of the Klein sail $\Klein$, obtained by the convex hull of lattice points in the cone $C_{L_1,L_2,L_3}$, the family of affine hyperplanes
\[ sn_f+E_f, \qquad s>0 \]
first intersects the integer lattice points $\ZZ^3$ in the cone $C_{L_1,L_2,L_3}$ precisely at $V_f$, the affine plane containing the face $f$. That is, 
\begin{align}\label{eq:s_f}
    s_f= \min\{s>0:~(\ZZ^3 \cap C_{L_1,L_2,L_3})\cap(sn_f+E_f)\neq\emptyset\}
\end{align}
satisfies
\[ s_fn_f+E_f = V_f. \]
Furthermore, the affine $\mathbb Z$-span of the integer points in $f$ (that is of the points in $V_f \cap C_{L_1,L_2,L_3}^\circ \cap \ZZ^3$) is an affine lattice in $V_f$.

Applying $a_tg$ to the family of affine hyperplanes we get
\[ a_tg(sn_f+ E_f) = s\alpha\v+ \Hv,\qquad s>0 \]
and these first intersect the lattice points $a_t\Lambda$ in the cone $a_tgC_{L_1,L_2,L_3} = \RR^3_{\geq0}$ precisely at $a_tgV_f$. Using \ref{eq:s_f} we thus obtain that
\[ s_{a_t\Lambda} = 
\inf\{s>0:~(a_t\Lambda)_{\geq 0}\cap(sv + \Hv)\neq\emptyset\} = s_f. \]
Further, the affine $\mathbb Z$-span of the lattice points in $(a_t\Lambda)_{+}\cap(s_{a_t\Lambda}v + \Hv)$ is an affine lattice contained in $a_tg V_f$.

Therefore, $a_t\Lambda\in\mathcal C$, and, by construction,
\[ \psi(a_t\Lambda) = ([f],\level(f)). \]
Conversely, every $t\in TJ$ with $a_t\Lambda\in\mathcal C$, corresponds to a face $f$ of the Klein sail $\Klein$ with $d_v(f) \leq e^T$. 

This gives a one-to-one correspondence between the faces $f$ of $\Klein$ with $d_v(f) \leq e^T$ and the parameters $t\in  T  J$ satisfying $a_t\Lambda \in \C$.
Hence, \eqref{eq:proof-main} yields
\begin{align}
\label{eq:proof-final}
 \frac{1}{T^2|J|} \sum_{\substack{
f\text{ face of }\mathcal K\\
d_v(f)\le e^T
}}
F([f],\level(f))
\longrightarrow
\mu_{\mathcal C}(F\circ\psi), \quad \text{as } T \rightarrow \infty.
\end{align}
The theorem now follows using Theorem~\ref{thm:main} by defining $\nu_{\FC} =\psi_*(\mu_{\mathcal C})$ and noting that $|J|=1$ by definition (see Section~\ref{sec:explicit-crosssection}).

\end{proof}


\subsection{Proof of Theorem~\ref{thm:Karpenkov question}} \label{subsec:Karpenkov}

Recall that $\mu$ is the unique (up to scaling) $\SL_3(\RR)$-invariant measure on the space of Klein sails $\sK_3$ obtained by identifying it with $\SL_3(\RR)/(\mathrm{D}_3\rtimes S_3)$.

\begin{lemma}
\label{lem:Kar:1}
    The measure $\mu$ on $\sK_3$ is the pushforward of
    the right-invariant Haar measure $m_P$ on $P$ under the map
    \[
        p\mapsto p^{-1}\gv(D_3\rtimes S_3)
        \simeq p^{-1}\gv\RR_{\geq0}^3,
    \]
    where $P$ and $\gv$ are defined as in \Cref{sec:explicit-crosssection}.
\end{lemma}

\begin{proof}
By a direct calculation, the set $D_3(\gv^{-1}P\gv)$ is a full measure subset of $G$. Hence, by Lemma~\ref{lem:Haar-product}
and the unimodularity of $G$, the Haar measure on $G$ is, up to a
normalizing constant, the pushforward of the product measure
$m_{D_3}\otimes m_P$ under the map
\[
    (a,p)\mapsto a\gv^{-1}p\gv.
\]
Composing this map with taking the inverse and multiplying on right by $\gv^{-1}$, both of which preserve the Haar measure, and then passing to the quotient $G/D_3$, we get that the $G$-invariant measure on the quotient $G/D_3$ is identified, up to a scalar multiple, with the measure induced by
$m_P$. Finally, passing to the quotient
\[
    G/D_3\longrightarrow G/(D_3\rtimes S_3)=\sK_3
\]
gives the measure $\mu$ on the space of Klein sails. Hence, $\mu$
is the pushforward of $m_P$ as claimed.
\end{proof}

We now introduce some notation. Given a Klein sail $\Klein$, we denote by
$F(\Klein)$ the set of all faces of $\Klein$. For a convex subset $f\subset \ZZ^3$ contained in a rational affine plane,
define
\[
    P_f=\{p\in P:f\in F(p^{-1}\gv\RR_{\geq0}^3)\}
\]
and let $\pi:P\longrightarrow X$ be defined as
$$
\pi(p)= \gv^{-1}p\Gamma.
$$ 

\begin{lemma}
\label{lem:Kar:2}
Assume that $f$ is a convex subset of $\ZZ^2\times\{k\}$ for some $k\in\NN$.
Then
\[
    m_P(P_f\setminus\pi^{-1}(\C))=0.
\]
Moreover, for every $ p\in P_f\cap\pi^{-1}(\C)$, we have
\begin{align}
    \label{eq:ba 1}
    \gv^{-1} p f
    =
    \RR_{>0}^3 \cap \pi(p)
    \cap
    (s_{\pi(p)}\v+\Hv).
\end{align}
\end{lemma}

\begin{proof}
If $f$ is contained in an affine line, then $P_f$ is empty and there is
nothing to prove. Hence, we assume that $f$ affinely spans a
two-dimensional lattice.

Consider $p\in P_f$ such that $\gv^{-1}p\ZZ^3$ intersects each coordinate plane only at the origin. By the definition of
$P_f$, the set $f$ is a face of the Klein sail associated to the cone
\[
    C_p=p^{-1}\gv\RR_{\geq0}^3.
\]
Since $f$ is contained in the affine plane $k\mathbf e_3+E_f$ where $E_f=\RR^2\times\{0\}$, the family of affine hyperplanes
\[
    s\mathbf e_3+E_f,\qquad s>0,
\]
meets the lattice points of $C_p\cap\ZZ^3$ for the first time precisely at $s=k$, and the corresponding set of lattice points is exactly $f$.
Applying $\gv^{-1}p$, we obtain that the first intersection of the family of affine hyperplanes
\[
    s\v+\Hv,\qquad s>0,
\]
with
\[
    \RR_{\geq0}^3\cap\gv^{-1}p\ZZ^3
    =
    \RR_{\geq0}^3\cap\pi(p)
\]
is precisely the set $\gv^{-1}pf$. The fact that we can replace $(\RR_{\geq0}^3\cap\pi(p))$ by $(\RR_{>0}^3\cap\pi(p))$ to obtain \eqref{eq:ba 1} follows since by assumption $\gv^{-1}p\ZZ^3$ intersects each coordinate plane only at the origin. Therefore, $\pi(p) \in \C$ and $p$ satisfies~\eqref{eq:ba 1}.
Since the set of elements $p\in P$, such that the lattice $\gv^{-1}p\ZZ^3$ intersects each coordinate plane only at the origin is a full measure subset of $P$, the \namecref{lem:Kar:2} follows.
\end{proof}

\begin{lemma}
\label{lem:Kar:3.1}
Assume that $f$ is a finite convex subset of $\ZZ^2\times\{k\}$ for some $k\in\NN$ and not contained in an affine line.
Then,
\[
    \left|
    \operatorname{Stab}_{\SL_3(\ZZ)}(f)
    \right|
    \leq 6.
\]
\end{lemma}

\begin{proof}
Let $\Gamma_f=\operatorname{Stab}_{\SL_3(\ZZ)}(f)$. Define $\Lambda_f$ to be the sublattice spanned by all elements in $f$. By assumption, $f$ is not contained in an affine line and as $k\geq 1$, $\Lambda_f$ is a finite index sublattice of $\ZZ^3$. Since every $\gamma\in\Gamma_f$ satisfies $\gamma f=f$, we immediately obtain $\gamma(\Lambda_f) = \Lambda_f$.
Moreover, every element of $\Gamma_f$ permutes the finitely many vertices of the face $f$. In particular, every $\gamma\in\Gamma_f$ fixing all vertices of $f$ fixes the finite index sublattice $\Lambda_f$ pointwise and so is the identity. Hence, the action of $\Gamma_f$ on the finite set of vertices is faithful, and therefore $\Gamma_f$ is finite.

Let $\gamma\in\Gamma_f$. Since $\gamma(\Lambda_f) = \Lambda_f$, it follows that $\gamma$ is of the form
\[
    \gamma=
    \begin{pmatrix}
        A&b\\
        &1
    \end{pmatrix},
    \qquad
    A\in\SL_2(\ZZ),\quad b\in\ZZ^2,
\]
and so in particular $\gamma\in P$. Since the $f\subset\ZZ^2\times\{k\}$ is not translation invariant, the natural projection $\Gamma_f \to \SL_2(\ZZ)$, $\gamma = \smat{A}{b}{}{1}\to A$ is injective. However, the finite subgroups of $\SL_2(\ZZ)$ are cyclic of order
$1$, $2$, $4$, or $6$. Therefore, $ |\Gamma_f|\leq 6$, which proves the lemma.
\end{proof}

\begin{lemma}
\label{lem:Kar:3}
Assume that $f$ is a finite convex subset of $\ZZ^2\times\{k\}$ for some $k\in\NN$ and not contained in an affine line. If $P_f\cap P_f\gamma\neq\emptyset$, then $\gamma\in \operatorname{Stab}_{\SL_3(\ZZ)}(f)$. In particular, there are at most $6$ elements $\gamma\in\Gamma$ such that $P_f\cap P_f\gamma\neq\emptyset$.
\end{lemma}

\begin{proof}
Assume that $\gamma\in\Gamma$ satisfies $P_f\cap P_f\gamma\neq\emptyset$. Choose $p_1,p_2\in P_f$ such that $p_1=p_2\gamma$. Then,
\[
    \gamma=p_2^{-1}p_1\in P\cap \Gamma,
\]
and so $\gamma$ preserves the affine plane $\RR^2\times\{k\}$ as well as $\ZZ^3$. Thus we get $\gamma f\subset \ZZ^2\times\{k\}$. 

By definition of $P_f$, $f$ is a face of the Klein sail associated to the simplicial cones
\[
    p_1^{-1}\gv\RR_{\geq0}^3 \quad \text{and}\quad p_2^{-1}\gv\RR_{\geq0}^3 .
\]
Applying $\gamma$ gives that $\gamma f$ is a face of the Klein sail associated to the simplical cone
\[
    \gamma p_1^{-1}\gv\RR_{\geq0}^3
    =
    p_2^{-1}\gv\RR_{\geq0}^3 .
\]
Thus, both $f$ and $\gamma f$ lie in the affine plane
$\RR^2\times\{k\}$ and are faces of the same Klein sail
associated to the simplicial cone $p_2^{-1}\gv\RR_{\geq0}^3$. Since a supporting plane of a Klein sail determines a
unique face, we obtain $ \gamma f=f $.
Therefore, every such $\gamma$ belongs to $\operatorname{Stab}_{\SL_3(\ZZ)}(f)$. By Lemma~\ref{lem:Kar:3.1}, this stabilizer has cardinality at most~$6$.
\end{proof}

\begin{proof}[Proof of Theorem~\ref{thm:Karpenkov question}]
Fix an integral equivalence class $[f] = \SL_3(\ZZ)\cdot f$ of a convex subset $f$ of $\ZZ^3$
contained in a rational affine hyperplane and satisfying $\rel([f])>0$. One may choose a representative $f$ of this class such that $f\subset \ZZ^2\times\{k\}$ for some $k\in\NN$.

Define $C_{[f]}=\pi(P_f)\subset \C$.
Equivalently, $C_{[f]}$ consists of those lattices $\Lambda=g\Gamma\in\C$ 
for which
\[
    g^{-1}
    \bigl(\Lambda_+\cap(s_\Lambda\v+\Hv)\bigr)
    \in \SL_3(\ZZ)\cdot f=[f].
\]
It follows immediately from this description that the sets
$C_{[f]}$ are pairwise disjoint as $[f]$ varies.

By Lemmas~\ref{lem:Kar:1} and~\ref{lem:Kar:2}, we have
\[
    \rel([f])
    =
    m_P(P_f)
    =
    m_P(P_f\cap\pi^{-1}(\C)).
\]
Now consider the restriction $\pi:P_f\cap\pi^{-1}(\C)\longrightarrow C_{[f]}$.
By Lemma~\ref{lem:Kar:3}, every fibre of this map has cardinality at most
$6$. Indeed, if $p_1,p_2\in P_f\cap\pi^{-1}(\C)$ have the same image under
$\pi$, then they differ by an element $\gamma\in\Gamma$
satisfying $P_f\cap P_f\gamma\neq\emptyset$. 

Moreover, by Lemma~\ref{lem:cross hom space}, the measure $\mu_\C$ is,
up to normalization, the measure induced on $\C$ by the restriction of
the right Haar measure on $P$. Hence the bounded multiplicity of $\pi$
gives
\[
    \rel([f]) = m_P(P_f\cap\pi^{-1}(\C))
    \leq
    6\,\mu_\C(C_{[f]}).
\]
Finally, the sets $C_{[f]}$ are pairwise disjoint, so we get by Theorem~\ref{thm:main} that
\[
    \sum_{[f]}\rel([f])
    \leq
    6\sum_{[f]}\mu_\C(C_{[f]})
    \leq
    6\mu_\C(\C)
    <\infty.
\]
\end{proof}

\subsection{Proof of Theorem~\ref{thm:periodic-equi}}\label{sec:proof-cubic-application}

    The proof of \Cref{thm:periodic-equi} boils down to two observations. First, a periodic Klein sail $\Klein$ associated to the split cone $g\RR^3_{\geq0}$ defines a compact $A_+$-orbit $A_+z$ in $X_3$, where $z= g^{-1}\Gamma$ in $X_3$. This compact orbit can be equipped with a unique $A_+$-invariant probability measure $\mu_{A_+z}$. The cross-sectional measure $\mu_{A_+z,\C}$ associated to $\mu_{A_+z}$ projects to $\nu_{[\Klein]}$ under the natural projection $\psi\colon\C\to\FC^2\times\NN$. (defined as in proof of Theorem~\ref{thm:face-equi})
    Second, we use the result of \citeauthor{ELMV11} in \cite{ELMV11} about equidistribution of packets of periodic $A_+$-orbits (see \Cref{thm:equi-packet}) to obtain the convergence of the measures $\nu_{\order_n}$ to the cross-sectional measure $\nu_{\FC}$.
    
    \begin{lemma}\label{lem:per-orbit}
        Let $\Klein$ be a periodic Klein sail associated to the split cone $g\RR^3_{\geq0}$. Let $z\defeq g^{-1}\Gamma$ and let $\mu_{A_+z}$ denote the unique $A_+$-invariant probability measure supported on~$A_+z$. The induced cross-sectional measure $\mu_{A_+z,\C}$ is given by
        \[ \mu_{A_+z,\C} = \frac{1}{R_{\order_\Klein}}\sum_{x\in A_+z\cap\C}\delta_x. \]
        In particular, $$\psi_*(\mu_{A_+z,\C}) = \nu_{[\Klein]},$$ where $\psi$ is defined as in proof of Theorem~\ref{thm:face-equi}
    \end{lemma}
    \begin{proof}
        It is clear that $\supp(\mu_{A_+z,\C}) \subseteq A_+z\cap\C$. Notice, that since $A_+z$ is a compact orbit, the set is finite $A_+z\cap\C$. Indeed, an infinite set would contradict discreteness of the lattice $A_+\cap\Gamma$ in $A_+$.
        
        We show that each of this finitely many points in $A_+z\cap\C$ obtains the same mass, namely $\frac{1}{R_{\order_\Klein}}$. Let $x\in A_+z\cap\C$ be a points in the periodic $A_+$-orbit $A_+z$. Let $(U_x,\psi_x,\gamma_x)$ be a flow-box-chart around $x$. By shrinking $U_x$ if necessary, we may assume that
        \[ U_x\cap(A_+z\cap \C) = \{x\}. \]
        It follows that for small enough $\gamma> 0$ the $\gamma$-box $R_\gamma$ around the origin in the hyperplane\linebreak $\{(t_1,t_2,t_3)\in\RR^3 : t_1+t_2+t_3=0\}$ satisfies
        \[ \mu_{A_+z}(\{a_t\st t\in R_\gamma\}U) = \frac{\Leb(R_\gamma)}{R_{\order_\Klein}}. \]
        
        However, by the local product structure of the measure $\mu_{A_+z}$, we have that for any $0<\gamma\leq\gamma_x$
        \[ \mu_{A_+z}(\{a_t\st t\in R_\gamma\}U) = \mu_{A_+z,\C}(U_x\cap\C)\Leb(R_\gamma) = \Leb(R_\gamma)\mu_{A_+z,\C}(\{x\}). \]
        Hence, for any $x\in A_+\cap\C$ we have 
        \[ \mu_{A_+z,\C}(\{x\}) = \frac{1}{{R_{\order_\Klein}}} \]
        as claimed.
    \end{proof}
    
    Recall that for an order $\order\subset F$ of a totally real cubic number field, we denote by $\mathcal{P}_{\order}$ the finite set equivalent periodic Klein sails $[\Klein]$ for which $\order_\Klein = \order$. Equivalently (see \eqref{eq:packet-equiv}), this packet consists of a finite set of periodic $A_+$-orbits associated to tuples $(F,M)$ of totally real cubic fields $F$ and rank-$3$ modules $M$, defined up to multiplication by $F^\times$, with $\order_M \defeq \{\lambda\in F\st \lambda M \subset M\} = \order$. We denote by $\mu_{\order}$ the $A_+$-invariant probability measure supported on this finite set of periodic $A_+$-orbits. Then in the seminal work \cite{ELMV11} the authors obtain the following.
    
    \begin{theorem}[{\cite[Theorem 1.4]{ELMV11}}]\label{thm:equi-packet}
        Let $\order_n\subset F_n$ be a sequence of orders, in a sequence of totally real fields $F_n$. Let $\mu_{\order_n}$ be the measure supported on the finite set of periodic $A$-orbits associated to the elements in $\mathcal{P}_{\order_n}$. If $\disc(\order_n)\to\infty$, then $\mu_{\order_n} \to \msr_{X_3}$ in the vague topology.
    \end{theorem}
    
    \begin{proof}[Proof of Theorem~\ref{thm:periodic-equi}]
        Let $z=g^{-1}\Gamma\in X_3$ be such that the orbit $A_+z\subset X_3$ is periodic. Then $C_z = g\RR^3_{\geq0}$ is a split cone and we denote by $\Klein_z$ the associated periodic Klein sail. By \Cref{lem:per-orbit}, we have $\psi_*(\mu_{A_+z,\C}) = \nu_{[\Klein_z]}$. In particular, we get that $\psi_*(\mu_{\order_n,\C}) = \nu_{\order_n}$ where $\nu_{\order_n}$ is defined in~\eqref{def:nu_O}.
        
        By \Cref{thm:equi-packet} we have $\mu_{\order_n} \xrightarrow{vague}\msr_{X_3}$ as $n\to\infty$ and since $\mu_{\order_n}$ is an $A$-invariant probability measures we can apply \Cref{lem:conv-in-crosssection} to obtain the convergence $\mu_{\order_n,\C} \to \mu_{\C}$ in vague topology of the associated cross-sectional measures. By continuity of $\psi$ we obtain $\nu_{\order_n} = \psi_*(\mu_{\order_n,\C}) \to \psi_*(\mu_{\C}) = \nu_{\FC}$ in vague topology as well.
    \end{proof}

    \begin{remark}
        Notice that the measure $\mu_{A_+z,\C}$ on the cross-section $\C$ associated to a compact orbit $A_+z\subset X_3$ is not necessarily a probability measure. In particular, it is not clear whether or not the total mass $\mu_{\order,\C}(\FC^2\times\NN)$ of a packet $\mathcal{P}_{\order}$ stays bounded or not when $\disc(\order)\to\infty$. Equivalently, let $\widehat{\mu}_{A_+z,\C}$ denote the uniform probability measure on the finite set $A_+z\cap\C$, and let $\widehat{\mu}_{\order,\C}$ denote the average of these measures for all finitely supported measures in the packet $\mathcal{P}_{\order}$. The total mass $\mu_{\order,\C}(\FC^2\times\NN)$ is bounded if and only if $\widehat{\mu}_{\order,\C}$ converge to the probability measure $\frac{1}{\nu_{\FC}(\FC^2\times\NN)}\nu_{\FC}$. 

        Unfortunately, this problem is out of reach with the current techniques. The reason is a similar phenomenon typically referred to as \highlight{escape of mass}. Although it is possible to control the contribution of $\Clv$, that is faces of a fixed level, to the total mass of a packet, it is possible that faces with very large level receive an unproportional amount of mass. This happens if a periodic orbit $A_+z$ intersects the faces of this large level an unproportional amount of time.

        Comparing with dimension $n=2$, there are well-known bounds for the average length of a periodic continued fraction expansion associated to to reduced quadratic irrationals in quadratic field. Similar bounds for $n\geq3$ regarding the average number of faces of a periodic Klein sail are not known. 
    \end{remark}


\section{Changing the viewpoint}\label{sec:Changing the viewpoint}
Define the semi-direct product
\[
H \defeq \GL_2(\RR)\ltimes \RR^2,
\]
with group multiplication given by $[A,b]\cdot [A',b']=[AA',\, b+Ab']$.
We consider the homogeneous space
\begin{align*}
Y_2
=
\tquot{\GL_2(\RR)\ltimes\RR^2}{\GL_2(\ZZ)\ltimes\ZZ^2}.
\end{align*}
Let $\msr_{Y_2}$ denote the measure on $Y_2$ induced by the right Haar measure $\msr_H^{\mathrm{right}}$ on $H$. Note that $\msr_{Y_2}$ is not finite.
The space $Y_2$ is naturally identified with the space of affine lattices in $\RR^2$, via
\[
[A,b]\bigl(\GL_2(\ZZ)\ltimes\ZZ^2\bigr)
\mapsto
A\ZZ^2 + b.
\]
Since $\smat{1}{0}{0}{-1}\in\GL_2(\ZZ)$ we may always assume that affine lattices written in the form $A\ZZ^2 + b$ satisfy $\det(A)>0$.

Now consider the group homomorphism
\begin{align}\label{def:auto-P-to-H}
\begin{split}
\varphi \colon P &\to H,\\
\begin{pmatrix}
A & b\\
0 & \det(A)^{-1}
\end{pmatrix}
&\mapsto
[\det(A)A,\det(A)b].
\end{split}
\end{align}
Then $\varphi$ is an isomorphism, with inverse given by
\begin{align}\label{def:auto-H-to-P}
\varphi^{-1}([A,b])
=
\frac{1}{\det(A)^{1/3}}
\begin{pmatrix}
A & b\\
0 & 1
\end{pmatrix}.
\end{align}
Since $\varphi(\SL_3(\ZZ)\cap P) =\GL_2(\ZZ)\ltimes\ZZ^2$, this induces an isomorphism
\[
P\Gamma
\simeq
P/(P\cap\Gamma)
\cong
Y_2.
\]
We continue to denote the induced map by $\varphi$.

We now define
\begin{align}
\label{eq:def-cross-H}
\Cc \defeq \varphi(\gv\C).
\end{align}
To study the subset $\Cc \subset Y_2$, we decompose $\C$ according to levels. Recall from \eqref{eq:def level} that the level of an element in $\C$ is a natural number. Hence we have the disjoint decomposition
\begin{align}\label{eq:level-dec}
\C = \bigsqcup_{k\in\NN} \Clv,
\end{align}
where $\Clv = \{\Lambda \in \C \st \level(\Lambda)=k\}$ and for $k \in \NN$, define
\[
\Cclv \defeq \varphi(\gv\Clv).
\]

The following proposition is the main technical result of the paper.

\begin{proposition}
    \label{prop:finite-measure}
There exist constants $c,C>0$ such that for every $k\geq 1$,
\[
\frac{c}{k^3}
\leq
\msr_{Y_2}(\Cclv)
\leq
\frac{C(\ln(k) + 1)}{k^2}.
\]
\end{proposition}

We now prove second part of Theorem~\ref{thm:main} assuming Proposition~\ref{prop:finite-measure}.
\begin{proof}[Proof of second part of Theorem~\ref{thm:main} assuming Proposition~\ref{prop:finite-measure}]
    Using Lemma~\ref{lem:cross hom space}, we note that $\mu_\C$ is, up to normalization, the restriction to $\C$ of the measure on $ (\gv^{-1}P\gv) \gv^{-1}\Gamma$  induced by the right Haar measure on $\gv^{-1}P\gv$. 

    The push-forward of $\mu_\C$ under the map $\varphi\circ\gv$ is the restriction to $\Cc$ of the measure on $Y_2$ induced from the right Haar measure on $H$, i.e, $\msr_{Y_2}$. Thus, we get that
    \[ \mu_{\C}(\{\Lambda \in \C: \level(\Lambda) =k \})= \msr_{Y_2}(\Cclv). \]
    The theorem now follows directly using Proposition~\ref{prop:finite-measure}, and the fact that $\sum_{k \in \NN} \frac{\ln(k) + 1}{k^2} < \infty$.
\end{proof}


\section{Strategy for measure estimate of \texorpdfstring{$\Cclv$}{C(k)}}\label{sec:strategy-msr-estimate}

\subsection{Description of set $\Cclv$}
 
    \begin{lemma}\label{lem:level-equivalence}
Let $k \in \mathbb{N}$ be fixed. The affine lattice $A\mathbb{Z}^2 + b \subset \mathbb{R}^2$ belongs to $\Cclv$ if and only if the following conditions hold:
\begin{enumerate}
    \item $A\mathbb{Z}^2 \cap \bigcup_{1 \leq \ell < k} (-\ell b + \ell \btri) = \emptyset$;\label{itm:level-eq-1}
    \item the set $A\mathbb{Z}^2 \cap (-kb + k\tri)$ generates the lattice $A\mathbb{Z}^2$,\label{itm:level-eq-2}
\end{enumerate}
where 
\begin{align}\label{def:eqi-triangle}
        \btri \defeq \left\{(x,y)\in\RR^2\st 1 \geq \frac{\sqrt{3}x+y}{\sqrt{2}}; 1 \geq \frac{-\sqrt{3}x+y}{\sqrt{2}} ; 1 \geq -\sqrt{2}y\right\},
    \end{align}
    and $\tri$ denotes the interior of $\btri$.
\end{lemma}
     
\begin{proof}
    Suppose $A\ZZ^2 + b \in \Cclv$. Then, by definition, the associated lattice
    \[ \Lambda = \gv^{-1}\det(A)^{-1/3}\mat{A}{b}{}{1}\ZZ^3 \]
    satisfies $\level(\Lambda)=k$ and belongs to the cross-section $\C$ (since $\varphi(\gv\Lambda) = A\ZZ^2 + b\in\Cclv$).

By definition of level,
\[ \level(\Lambda) = \min\Bigl\{n\ge1:(\Lambda_0+n t_\Lambda)\cap\RR_{\ge0}^3\neq\emptyset\Bigr\}, \]
where $\Lambda_0$ and $t_\Lambda$ is defined as in \Cref{sec:explicit-crosssection}. In our case, as $E_v = \gv^{-1}(\RR^2\times\{0\})$, we have
\[ \Lambda_0 = \Lambda \cap \Hv = \gv^{-1} \left(\det(A)^{-1/3}\mat{A}{b}{}{1} \ZZ^3 \cap \RR^2 \times \{0\} \right)= \gv^{-1} \det(A)^{-1/3} \left(A \ZZ^2 \times \{0\} \right), \]
and we can choose 
\[ t_\Lambda = \gv^{-1}\det(A)^{-1/3}\begin{pmatrix}b \\ 1\end{pmatrix}. \] 
Thus, we get that
\begin{align*}
    \level(\Lambda) &= \min\Bigl\{n\ge1: \gv^{-1} \det(A)^{-1/3} \left(A \ZZ^2 \times \{0\} +\begin{pmatrix} nb \\ n\end{pmatrix} \right)\cap\RR_{\ge0}^3\neq\emptyset\Bigr\} \\
    &= \min\Bigl\{n\ge1:\left(A \ZZ^2 \times \{0\} + \begin{pmatrix}nb \\ n\end{pmatrix} \right)\cap \gv\RR_{\ge0}^3\neq\emptyset\Bigr\}.
\end{align*}
An explicit computation yields that an element \((x,y,n)\) is in \(\gv\RR_{\ge0}^3\) if and only if $(x,y) \in n\,\btri$.
Therefore
\[ \left(A \ZZ^2 \times \{0\} + \begin{pmatrix}nb \\ n \end{pmatrix}  \right)\cap  \gv\RR_{\ge0}^3 \neq \emptyset
\quad \Longleftrightarrow \quad A\mathbb{Z}^2 \cap (-nb + n\btri) \neq \emptyset.
\]

Since $\level(\Lambda)=k$, the above computation implies~(\ref{itm:level-eq-1}).

To obtain~(\ref{itm:level-eq-2}), note that $\Lambda \in \C$ implies that the affine $\ZZ$-span of $(\Lambda_+ -k\v)\cap \Hv$ is an affine lattice in $\Hv$. After applying $ \det(A)^{1/3}\gv$ this is equivalent to the fact that the affine $\ZZ$-span of
\[\left( \left( \begin{pmatrix}A & b \\ & 1\end{pmatrix}  \ZZ^3 \cap \gv \RR_{>0}^3 \right)- \begin{pmatrix}0 \\ 0 \\ k\end{pmatrix} \right) \cap (\RR^2 \times \{0\}) \]
is an affine lattice in $\RR^2 \times \{0\}$. Using the fact that \((x,y,k)\) is in \(\gv\RR_{>0}^3\) if and only if $(x,y) \in k\tri$, the above equation implies that the affine $\ZZ$-span of
\[ (A\ZZ^2 + kb) \cap k\tri \]
is an affine lattice in $\RR^2$. The latter implies~(\ref{itm:level-eq-2}).

    This proves one direction. The other direction follows by the reverse line of thoughts.
    \end{proof}
By abuse of notation, for any lattice $A\ZZ^2 + b\in\Cc$ we define (with $\varphi$ as in ~\eqref{def:auto-P-to-H})
\[ \level(A,b) \defeq \level(\Lambda), \quad \Lambda = \gv^{-1}\varphi^{-1}([A,b])\ZZ^3.\]  

\subsection{Heuristic for measure estimates of \texorpdfstring{$\Cclv$}{C(k)}}
The proof of \Cref{prop:finite-measure} proceeds in four main steps.

\medskip

\noindent\textbf{Step 1: Bounds on $\level(A,b)$ in terms of successive minima}
(\Cref{sec:bounds-on-level}).
We first show that the level of the affine lattice $A\ZZ^2+b$ is comparable to the second successive minimum of the lattice $A\ZZ^2$, namely, $\level(A,b)\asymp \succ_2(\tri,A\ZZ^2)$ (see \Cref{lem:covering-successive}). This immediately implies that the set of affine lattices of level one has finite measure.

\medskip

\noindent\textbf{Step 2: Congruence conditions for lattices of level $k\ge2$}
(\Cref{sec:witnesses}).
Suppose that $A\ZZ^2+b$ has level $k\ge2$. Then there exist three points $P_1,P_2,P_3\in A\ZZ^2\cap(-kb+k\tri)$ such that $a_1=P_2-P_1$ and $a_2=P_3-P_1$ form a basis of $A\ZZ^2$. We call such a triple a set of \highlight{witnesses}. In \Cref{lem:desc-of-basis-I}, we prove that the coefficients of $P_1$ with respect to the basis $a_1,a_2$ satisfy certain congruence conditions, which impose strong arithmetic restrictions on the possible witness configurations. 

\medskip

\noindent\textbf{Step 3: Reduced bases}
(\Cref{sec:reduced-minima}).
We introduce the notions of reduced minima and reduced bases, which play a similar role as successive minima, but for the non-symmetric convex set $\tri$. \Cref{lem:red-basis-in-tri} then shows that every affine lattice of level $k\ge2$ admits a set of witnesses whose difference vectors are a reduced basis. This further reduces the number of possible witness configurations.

\medskip

\noindent\textbf{Step 4: Estimating the measure of level-$k$ lattices}
(\Cref{sec:proof-level-measure}).
Combining the congruence conditions from Step~2 with the existence of reduced bases from Step~3, we show that, for each $k$, there are only $O(k)$ possible configurations of witnesses. Fixing one such configuration ${P_1,P_2,P_3}$, the translation vector $b$ must satisfy
\[
b\in\bigcap_{i=1}^3\left(-\frac1kP_i+\tri\right).
\]
On the other hand, since the affine lattice is exactly of level $k$, for every $1\le l<k$ and every lattice point $Q\in A\ZZ^2$,
\[
b\notin -\frac1lQ+\tri.
\]
Using this observation, we construct three lattice points $Q_1,Q_2,Q_3$ together with integers $1\le l_i<k$, and use the fact that the vector $b$ belongs to the set
\[
\bigcap_{i=1}^3\left(-\frac1kP_i+\tri\right)
\setminus
\bigcup_{i=1}^3
\left(-\frac1{l_i}Q_i+\tri\right),
\]
to obtain an explicit upper bound for the measure of the admissible translation vectors $b \in \RR^2/A\ZZ^2$ in terms of $\det(A)$. Integrating this estimate over the space of lattices and using elementary algebraic manipulations yields the required upper bound (\Cref{prop:upper-bound-level-measure}).

For the lower bound, we construct a family of lattices $Z_k$ such that every lattice in $Z_k$ admits a sufficiently large set of admissible translation vectors, yielding the desired lower bound (\Cref{prop:lower-bound-level-measure}).

 Finally, \Cref{prop:finite-measure} follows directly from
Propositions~\ref{prop:upper-bound-level-measure} and
\ref{prop:lower-bound-level-measure}; see
\Cref{subsec:proof of measure estimate}.

\subsection{Proof of \Cref{prop:finite-measure}} \label{subsec:proof of measure estimate}
 \begin{proof}[Proof of \Cref{prop:finite-measure}]
        Fix $k\in\NN$. The isomorphism
        \begin{align*}
            \GL_2^{+}(\RR) &\xlongrightarrow{\sim}~\RR_{>0}\times~\SL_2(\RR),\\
            A &\mapsto (\det(A),\det(A)^{-1/2}A)
        \end{align*}
        with inverse $(r,A')\mapsto r^{1/2}A'$ maps $\GL_2^{+}(\ZZ)$ to $\{1\}\times\SL_2(\ZZ)$. Define $\overline{X}_2\defeq\tquot{\GL_2^{+}(\RR)}{\GL_2^{+}(\ZZ)}$ as well as $X_2\defeq\tquot{\SL_2(\RR)}{\SL_2(\ZZ)}$, then the above gives a well-defined map $\psi\colon \overline{X_2}\to\RR_{>0}\times X_2$ such that $\dd\psi_*\msr_{\overline{X}_2} = \frac{\dd r}{r}\dd\msr_{X_2}$. For any $A\in\GL^{+}_2(\RR)$, let $\mathcal{F}_{A}\subset\RR^2$ denote a fundamental domain of $A\ZZ^2\subset\RR^2$ and let $\mathcal{F}_{A}^{(k)} = \{b\in\mathcal{F}_{A}\st \level(A,b) = k\}$. Using the substitution rule for $\psi$ we obtain
        \begin{align*}
            \msr_{Y_2}(\Cclv) = \int_{\overline{X}_2}\Leb\left(\mathcal{F}_{A}^{(k)}\right)\dd\msr_{\overline{X}_2}(A) = \int_{X_2}\int_{\RR_{>0}}\Leb\left(\mathcal{F}_{r^{1/2}A'}^{(k)}\right)\frac{\dd r}{r}\dd\msr_{X_2}(A').
        \end{align*}

        For $k=1$ we use \Cref{lem:covering-successive} to obtain $0 \leq r^{1/2}\succ_2(\tri,A'\ZZ^2) \leq 2$ which yields, for a given $A'\in\SL_2(\ZZ)$, an upper bound on $r$. By the trivial bound $\Leb(\mathcal{F}_{r^{1/2}A'}^{(k)}) \leq \det(r^{1/2}A') = r$ we get
        \begin{align*}
            \msr_{Y_2}(\Cc^{(1)}) &\leq \int_{X_2}\int_{0}^{\frac{4}{\succ_2(\tri,A'\ZZ^2)^2}}\Leb\left(\mathcal{F}_{r^{1/2}A'}^{(k)}\right)\frac{\dd r}{r}\dd\msr_{X_2}(A')\\
            &\leq \int_{X_2}\int_{0}^{\frac{4}{\succ_2(\tri,A'\ZZ^2)^2}}\dd r\dd\msr_{X_2}(A')\\
            &= 4\int_{X_2}\succ_2(\tri,A'\ZZ^2)^{-2}\dd\msr_{X_2}(A) \leq 4,
        \end{align*} 
        as $\succ_2(\tri,A'\ZZ^2) \gg 1$ for all $A'\ZZ^2\in X_2$ and $\msr_{X_2}$ is a probability measure. A non-trivial lower bound for $\msr_{Y_2}(\Cc^{(1)})$ can also be obtained by the bounds provided in \Cref{lem:covering-successive}. Indeed, if $2r^{1/2}\succ_2(\tri,A'\ZZ^2) < 1$, all affine lattices satisfying $r^{1/2}A'\ZZ^2 + b\in \Cc$ for some $b\in\RR^2$ are of level $1$. It is not hard to see that a large set of $b\in\mathcal{F}_A$ satisfy $r^{1/2}A'\ZZ^2 + b\in \Cc$, yielding a non-trivial lower bound for $\msr_{Y_2}(\Cc^{(1)})$.
        
        For $k\geq 2$ we have by \Cref{prop:upper-bound-level-measure} that
        \[ \int_{\RR_{>0}}\Leb\left(\mathcal{F}_{r^{1/2}A'}^{(k)}\right)\frac{\dd r}{r} \leq \frac{12\cdot9\sqrt{3}\cdot64(\ln(k)+1)}{k(k-1)} \leq \frac{12\cdot9\sqrt{3}\cdot128(\ln(k)+1)}{k^2} \]
        and the bound is independent of $A'\in\SL_2(\RR)$. Thus, we obtain
        \begin{align*}
            \msr_{Y_2}(\Cclv) = \int_{X_2}\int_{\RR_{>0}}\Leb\left(\mathcal{F}_{r^{1/2}A'}^{(k)}\right)\frac{\dd r}{r}\dd\msr_{X_2}(A') \leq \frac{C(\ln(k)+1)}{k^2}
        \end{align*}       
        with $C=12\cdot9\sqrt{3}\cdot128$. By \Cref{prop:lower-bound-level-measure}, there is a subset $Z_k\subset X_2$ of positive measure such that for any $A'\ZZ^2\in Z_k$ 
        \[ \int_{\RR_{>0}}\Leb\left(\mathcal{F}_{r^{1/2}A'}^{(k)}\right)\frac{\dd r}{r} \geq \frac{3\sqrt{3}}{5(2k-1)^2}\left(\frac{\succ_1(\tri,A'\ZZ^2)}{\succ_2(\tri,A'\ZZ^2)}\right)^2 \geq \frac{3\sqrt{3}}{20 \cdot k^2}\left(\frac{\succ_1(\tri,A'\ZZ^2)}{\succ_2(\tri,A'\ZZ^2)}\right)^2. \]
        Moreover, the second part of \Cref{prop:lower-bound-level-measure} implies that         
        \begin{align*}
            \msr_{Y_2}(\Cclv) &\geq \int_{Z_k}\int_{\RR_{>0}}\Leb\left(\mathcal{F}_{r^{1/2}A'}^{(k)}\right)\frac{\dd r}{r}\dd\msr_{X_2}(A') \geq \frac{3\sqrt{3}}{20\cdot k^2}\int_{Z_k}\left(\frac{\succ_1(\tri,A'\ZZ^2)}{\succ_2(\tri,A'\ZZ^2)}\right)^2\dd\msr_{X_2}(A') \geq \frac{c}{k^3}
        \end{align*}       
        where $c$ is an absolute positive constant not depending on $k$.
    \end{proof}

    \begin{proposition}\label{prop:upper-bound-level-measure}
        Let $k\geq 2$ and $A'\in\SL_2(\RR)$. Then, 
        \[ \int_{\RR_{>0}}\Leb(\mathcal{F}_{r^{1/2}A'}^{(k)})\frac{\dd r}{r} \leq \frac{12\cdot9\sqrt{3}\cdot64(\ln(k)+1)}{k(k-1)}. \]
    \end{proposition}

    \begin{proposition}\label{prop:lower-bound-level-measure}
        Let $k\geq 2$. There is a subset $Z_k\subset X_2$ depending on $k$, such that for any $A'\ZZ^2\in Z_k$ we have
        \[ \int_{\RR_{>0}}\Leb(\mathcal{F}_{r^{1/2}A'}^{(k)})\frac{\dd r}{r} \geq \frac{3\sqrt{3}}{5(2k-1)^2}\left(\frac{\succ_1(\tri,A'\ZZ^2)}{\succ_2(\tri,A'\ZZ^2)}\right)^2. \]
        Moreover, $Z_k$ satisfies
        \[ \int_{Z_k}\left(\frac{\succ_1(\tri,A'\ZZ^2)}{\succ_2(\tri,A'\ZZ^2)}\right)^2\dd\msr_{X_2}(A') \asymp \frac{1}{k}. \]
    \end{proposition}


\section{Bounds on level}\label{sec:bounds-on-level}
    Recall that for any $A\ZZ^2 + b\in\Cc$, we have (with $\varphi$ as in ~\eqref{def:auto-P-to-H})
    \[ \level(A,b) = \level(\Lambda), \quad \Lambda = \gv^{-1}\varphi^{-1}([A,b])\ZZ^3. \]
    
    The aim of this subsection is to bound $\level(A,b)$ in terms of the successive minima of the associated lattice, which are defined as follows.

      \begin{definition}
        Let $L\subset\RR^2$ be a lattice and $S\subseteq\RR^2$ a convex set. For $i\in\{1,2\}$, the $i$-the successive minimum with respect to $S$ is defined as
        \[ \succ_i(S, L) \defeq \inf\{r>0\st r(S-S)\cap L \text{ contains $i$ linearly independent vectors}\}, \]
        where $r(S-S) = \{r(x-y): x,y \in S\}$ for $r \in \RR$.
    \end{definition}
    
    The following lemma is the main result of this section.
    \begin{lemma} \label{lem:covering-successive} 
        Let $A\ZZ^2 + b\subset\RR^2$ be an affine lattice in $\Cc$, then
        \[ \tfrac{1}{2}\succ_2(\tri, A\ZZ^2) \leq  \level(A,b) < 2\succ_2(\tri,A\ZZ^2) + 1. \]
    \end{lemma}

    To prove \Cref{lem:covering-successive}, we need the following.
    \begin{lemma}\label{lem:level-bounds}
        Let $A\ZZ^2 + b\subset\RR^2$ be an affine lattice in $\Cc$, then
        \[
        \tfrac{1}{2}\mu_2(\tri,A\ZZ^2) \leq  \level(A,b) < \mu_2(\tri,A\ZZ^2) + 1, 
        \]
        where the \highlight{covering radius} $\mu_2(S,L)$ of a convex set $S\subseteq\RR^2$ and a lattice $L\subset\RR^2$ is defined as
        \[ \mu_2(S,L) = \inf\{r>0\st rS + L = \RR^2\}.  \]
    \end{lemma}
    \begin{proof}
        First, we prove the upper bound. The upper bound is trivial in the case of $\level(A,b)=1$. Thus, assume that $ \level(A,b) = k \geq 2$. By \Cref{lem:level-equivalence} and the definition of $\level(A,b)$, we know that \[ A\ZZ^2\cap \left(-(k - 1)b + (k-1)\tri\right) = \emptyset. \]
        This is equivalent to $(k-1)b \not\in (k-1)\tri + A\ZZ^2$ which implies $\mu_2((k-1)\tri,A\ZZ^2) > 1$. Thus, we get
        \[ 1 < \mu_2((k-1)\tri,A\ZZ^2) = (k-1)^{-1}\mu_2(\tri,A\ZZ^2), \]
        and the upper bound follows.

        Second, we prove the lower bound. Again, by \Cref{lem:level-equivalence} the elements in $A\ZZ^2 \cap (-kb + k\tri)$ generate the full lattice $A\ZZ^2$. We claim that $2k\tri + A\ZZ^2 = \RR^2$.
        Pick three points $P_1,P_2,P_3\in A\ZZ^2 \cap (-kb + k\tri)$ which span the lattice $A\ZZ^2$, that is
        \[ A\ZZ^2 = \{c_1(P_2 - P_1) + c_2(P_3 - P_1)\st c_1,c_2\in\ZZ\}. \]
        Observe that $2k\tri + A\ZZ^2 = \RR^2$ is equivalent to $2(-kb + k\tri) + A\ZZ^2 = \RR^2$. We know that $2P_1, 2P_2, 2P_3\in 2(-kb + k\tri)$ and as $2(-kb + k\tri)$ is convex, the convex combinations
        \[ P_1 + P_2 = \frac{1}{2}2P_1 + \frac{1}{2}2P_2,\quad  P_1 + P_3 = \frac{1}{2}2P_1 + \frac{1}{2}2P_3, \quad P_2 + P_3 = \frac{1}{2}2P_2 + \frac{1}{2}2P_3. \]
        are all elements of $2(-kb + k\tri)$. This shows that $2(-kb + k\tri)$ contains the fundamental domain of $A\ZZ^2$ spanned by the four vectors
        \[ 2P_1, \quad 2P_1 + (P_2-P_1) = P_1 + P_2,\quad  2P_1+(P_3-P_1) = P_3 + P_1, \quad 2P_1 + (P_2-P_1) +(P_3-P_1) = P_2 + P_3, \]
        so $2(-kb + k\tri) + A\ZZ^2 = \RR^2$ follows.
    \end{proof}

    \begin{proof}[Proof of \Cref{lem:covering-successive}]
        The lemma follows directly using Lemma~\ref{lem:level-bounds} and the following relation between the covering radius $\mu_2$ and the successive minima $\succ_2$ (cf. \cite[Lemma (2.4)]{KL88})
        \[ \succ_2(S,L) \leq \mu_2(S,L) \leq 2\succ_2(S,L), \]
        which holds for any lattice $L\subset\RR^2$ and any convex subset $S\subseteq\RR^2$.
    \end{proof}

    \begin{remark}
        Both bounds in the \Cref{lem:level-bounds} are sharp. To see this consider the hexagonal lattice $A\ZZ^2$ with first and second successive minimum of size $a$. For this specific lattice, one can prove that $\mu_2(\tri,A) = \tfrac{\sqrt{2}}{\sqrt{3}}a$. If one chooses $a=\sqrt{2}$ so that $A\ZZ^2$ is of covolume $\sqrt{3}$, one has $\mu_2(\tri,A) = \tfrac{2}{\sqrt{3}}$. The lower bound is then equal to $1$ and the upper bound is $\tfrac{2}{\sqrt{3}}+1<3$ and there exist $b_1,b_2\in\RR^2$ such that $A\ZZ^2 + b_1\in\Cc$ and $A\ZZ^2 + b_2\in\Cc$ are of level $1$ and $2$, respectively.
    \end{remark}


\section{Congruence condition on points in the interior of \texorpdfstring{$\tri$}{triangle}}\label{sec:witnesses}

    The aim of this section is to prove the following lemma.
    \begin{lemma}\label{lem:desc-of-basis-I}
        Let $k\geq2$ and $A\ZZ^2 +b\in\Cclv$ be a lattice of level $k$. Assume there are $P_1,P_2,P_3\in A\ZZ^2\cap (-kb + k\tri)$ such that $a_1 \defeq P_2 - P_1$, and $a_2 \defeq P_3 - P_1$ form a basis of $A\mathbb{Z}^2$. If
        \[ P_1 = c_1a_1 + c_2a_2 \in A\ZZ^2, \]
        then 
        \[ \gcd(c_1,k) = \gcd(c_2,k) = 1, \]
        and at least one of the following holds:
        \[
            c_1 + 1 \equiv 0 \pmod{k}, \qquad
            c_2 + 1 \equiv 0 \pmod{k}, \qquad
            c_1 + c_2 \equiv 0 \pmod{k}.
        \]
    \end{lemma}
    \begin{proof}
        First, we note that
        \[ P_1 = c_1a_1 + c_2a_2,\quad P_2 = (c_1+1)a_1 + c_2a_2,\quad P_3 = c_1a_2 + (c_2+1)a_3. \]
        Assume there is a linear combination $P \defeq d_1 P_1 + d_2 P_2 + d_3 P_3$, with $d_1,d_2,d_3\in\ZZ$ satisfying
        \begin{enumerate}
            \item $d_1,d_2,d_3\geq 0$; \label{itm:equations-1}
            \item $\ell\defeq d_1+d_2+d_3 < k$;\label{itm:equations-2}
            \item $\ell c_1 + d_2 \equiv 0\mod k$;\label{itm:equations-3}
            \item $\ell c_2 + d_3 \equiv 0\mod k$.\label{itm:equations-4}
        \end{enumerate} 
        Then, we have $\tfrac{1}{k}P \in A\ZZ^2\cap (-\ell b + \ell\tri)$, a contradiction to $A\ZZ^2 + b\in\Cclv$. Indeed, substituting $P_1,P_2$ and $P_3$ yields
        \begin{align*}
            P &= d_1(c_1a_1 + c_2a_2) + d_2((c_1+1)a_1 + c_2a_2) + d_3(c_1a_1 + (c_2+1)a_2)\\
            & = (\ell c_1 + d_2) a_1 + (\ell c_2 + d_3) a_2 \in kA\ZZ^2,
        \end{align*}
        by \eqref{itm:equations-3} and \eqref{itm:equations-4}, thus $\tfrac{1}{k}P \in A\ZZ^2$. The fact that $\tfrac{1}{k}P \in (-\ell b + \ell\tri)$ follows from the convex combination
        \[ \frac{d_1}{\ell} (P_1 + kb) + \frac{d_2}{\ell}(P_2 +kb) + \frac{d_3}{\ell}(P_3 + kb) = \frac{1}{\ell}P + k b  \in k\tri, \]
        of points in the convex set $k\tri$. Hence, the existence of $d_1,d_2,d_3\in\ZZ$ satisfying (\labelcref{itm:equations-1})--(\labelcref{itm:equations-4}) contradicts the level assumption of $A\ZZ^2 +b\in\Cc$.

        Assume that $\gcd(c_1,k)\eqdef m\geq 1$ and write $k = k'm$. Then,
        \[ d_1 = k'(m-1) - (k'c_2 \mod k),\quad d_2=0,\quad d_3 = (k'c_2 \mod k) \]
        gives a solution to (\labelcref{itm:equations-1})--(\labelcref{itm:equations-4}). The only non trivial statement is $d_1 \geq 0$, which follows from $ k'(m-1) \geq (k'c_2 \mod k)$, as $k'(m-1)$ is the largest multiple of $k'$ in $\{0,\dots,k-1\}$.
        The case $\gcd(c_2,k)>1$ can be dealt with analogously.

        Hence, we assume that $c_1$ and $c_2$ are relatively prime to $k$. If
        \[ c_1+1,c_2+1,c_1+c_2\not\equiv 0\mod k, \]
        then by \cite[Theorem 2]{W64} there exists $\ell\in\NN$ such that $1\leq\ell<k$ and
        \[ \ell + (\ell c_1 \mod k) + (\ell c_2 \mod k) \leq k. \]
        Using this $\ell$, we can set 
        \[ d_1 = k - \ell - (\ell c_1 \mod k) - (\ell c_2 \mod k)\geq 0,\quad d_2 = (\ell c_1 \mod k), \quad d_3 = (\ell c_2\mod k), \]
        and get a solution to (\labelcref{itm:equations-1})--(\labelcref{itm:equations-4}). The \namecref{lem:desc-of-basis-I} follows.
    \end{proof}
    An immediate \namecref{cor:no-fund-domain} of the \namecref{lem:desc-of-basis-I} above is the fact that the witnesses of $A\ZZ^2 +b\in\Cclv$ being of level $k\geq 2$ cannot give an entire fundamental domain of $A\ZZ^2$. 
    Although \Cref{cor:no-fund-domain} is not needed for the proof of \Cref{thm:main}, we state it because it drastically reduces the possible shapes of integer-affine types of level $k\geq 2$.  
    
    \begin{corollary}\label{cor:no-fund-domain}
        Let $k\geq2$ and $A\ZZ^2 +b\in\Cclv$ be a lattice of level $k$. Assume there are $P_1,P_2,P_3\in A\ZZ^2\cap (-kb + k\tri)$ such that $a_1 \defeq P_2-P_1$ and $a_2 \defeq P_3-P_1$ form a basis of~$A\ZZ^2$. Then
        \[ P_1 + a_1 + a_2\not\in A\ZZ^2\cap(-kb + k\tri). \]
    \end{corollary}
    \begin{proof}
        Assume that $P_4=P_1 + a_1 + a_2\in A\ZZ^2\cap (-kb + k\tri)$. Then we can apply \Cref{lem:desc-of-basis-I} to the points $P_1' = P_4$, $P_2'=P_3$ and $P_3'=P_2$. We get $(P_2'-P_1') = -a_1$, $(P_3'-P_1')=-a_2$ and
        \[ P_1' = (c_1+1)a_1 + (c_2+1)a_2 = -(c_1+1)(-a_1) - (c_2+1)(-a_2). \]
        We conclude that $\gcd(-(c_1+1),k)=\gcd(-(c_2+1),k)=1$ and either $-c_1$, $-c_2$, or $-c_1-c_2-2$ are congruent to $0$ modulo $k$. These conditions contradict the ones obtained for $P_1$, $P_2$, and $P_3$. Indeed, since $k\geq 2$ the only possible modularity conditions are $c_1+c_2\equiv 0\mod k$ and $-c_1-c_2-2\equiv 0\mod k$, in case that $k=2$. However, either $c_1$ or $c_1+1$ is even and so one of the $\gcd$ conditions $\gcd(c_1,k)=1$ and $\gcd(-(c_1+1),k)=1$ yields a contradiction.
    \end{proof}

    
\section{Reduced Minima and Reduced Bases}\label{sec:reduced-minima}
    
    Consider the three unit vectors
    \[ v_1 = (0,1),\quad  v_2 = (\tfrac{\sqrt{3}}{2}, -\tfrac{1}{2}),\quad v_3=(-\tfrac{\sqrt{3}}{2}, -\tfrac{1}{2}), \]
    and define for all $i\in\{1,2,3\}$, the closed and open halfspaces 
    \[ H_i \defeq \{P\in\RR^2\st \langle v_i,P\rangle \geq 0\},  \qquad  H_i^{\circ} \defeq \{P\in\RR^2\st \langle v_i,P\rangle > 0\}. \]
    Also, define the three half-open cones
    \begin{align*}
        \cone_1 \defeq H_2\cap H_3^{\circ},\quad \cone_2 \defeq H_3\cap H_1^{\circ}\quad \cone_3 \defeq H_1\cap H_2^{\circ}.
    \end{align*}
    It is easy to see that $\cone_i \subset -H_i$ and that we have the disjoint decomposition
    \[ \RR^2\setminus\{0\} = \cone_1\sqcup -\cone_1 \sqcup \cone_2 \sqcup -\cone_2 \sqcup \cone_3 \sqcup -\cone_3. \]

    
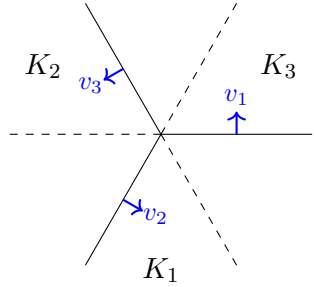
\begin{figure}[H]
        \begin{center}
            \begin{tikzpicture}[scale=1]
                \coordinate (O) at (0,0);
                
                \draw[dashed] (-2,0) -- (O);
                \draw (O) coordinate (L1_start) -- (2,0);
                \coordinate (V1_start) at ($ (L1_start)!0.5!(2,0) $);
                \draw[->, thick, blue] (V1_start) -- ++(0,0.3) coordinate (V1_end);
                \node[blue, font=\small, above] at ($ (V1_start)!0.5!(V1_end) + (0,0.1) $) {$v_1$};
                
                \draw ({-2*cos(60)}, {-2*sin(60)}) coordinate (L2_start)-- (O);
                \draw[dashed] (O) -- ({2*cos(60)}, {2*sin(60)});
                \coordinate (V2_start) at ($ (L2_start)!0.5!(O) $);
                \draw[->, thick, blue] (V2_start) -- ++(-30:0.3) coordinate (V2_end);
                \node[blue, font=\small, below right] at ($ (V2_start)!0.5!(V2_end) + (0,0.1) $) {$v_2$};
                
                \draw[dashed] ({-2*cos(120)}, {-2*sin(120)})  -- (O);
                \draw (O) coordinate (L3_start) -- ({2*cos(120)}, {2*sin(120)});
                \coordinate (V3_start) at ($ (L3_start)!0.5!({2*cos(120)}, {2*sin(120)}) $);
                \draw[->, thick, blue] (V3_start) -- ++(210:0.3) coordinate (V3_end);
                \node[blue, font=\small, below left] at ($ (V3_start)!0.5!(V3_end) + (0,0.1) $) {$v_3$};
                
                \node at (30:1.8) {$\cone_3$};
                \node at (90:1.8) {};
                \node at (150:1.8) {$\cone_2$};
                \node at (210:1.8) {};
                \node at (270:1.8) {$\cone_1$};
                \node at (330:1.8) {};
            \end{tikzpicture} 
        \end{center}

    \end{figure}

    Observe that 
    \begin{align}\label{eq:dir-identity}
        v_1 + v_2 + v_3 = 0,
    \end{align}
    and that
    \begin{align}\label{eq:dir-pair-corr}
        \langle v_i, v_j \rangle = \begin{cases} 1, &\text{ if } i = j,\\ -\tfrac{1}{2}, &\text{ if } i \neq j.\end{cases}
    \end{align} 
    We note that $\btri$ as defined in \eqref{def:eqi-triangle} satisfies
    \begin{align}\label{eq:desc-btri}
        \btri = \{P\in\RR^2\st \forall i\in\{1,2,3\}~\langle v_i,P\rangle \geq -\tfrac{1}{\sqrt{2}} \},
    \end{align}
    and so
    \begin{align}\label{eq:desc-tri}
        \tri = \{P\in\RR^2\st \forall i\in\{1,2,3\}~\langle v_i,P\rangle > -\tfrac{1}{\sqrt{2}} \}.
    \end{align}
    
    \begin{lemma}\label{eq:basics-on-vectors}
        For any $P\in\RR^2$, we have
        \begin{align}\label{eq:point-desc}
            P = \frac{2}{3}\big(\langle v_1,P\rangle v_1 +  \langle v_2,P\rangle v_2 +  \langle v_3,P\rangle v_3\big),
        \end{align}
        and
        \begin{align}
        \label{eq:eq:basics-on-vectors}
            \max_{i\in\{1,2,3\}}\{\abs{\langle v_i, P \rangle}\} \leq 2\max_{i\in\{1,2,3\}}\{\langle v_i, P \rangle\}.
        \end{align}
    \end{lemma}
    \begin{proof}
       The equality in \eqref{eq:point-desc} follows directly from~\eqref{eq:dir-identity} and~\eqref{eq:dir-pair-corr}. 
       
       To obtain \eqref{eq:eq:basics-on-vectors}, fix $j\in\{1,2,3\}$ such that $\abs{\langle v_j,P\rangle} = \max_{i\in\{1,2,3\}}\left\{\abs{\langle v_i, P \rangle}\right\}$.
       If $\langle v_j,P\rangle \geq 0$ the statement is trivially true, so assume $\langle v_j,P\rangle < 0$. By \eqref{eq:dir-identity} we have $\langle v_1 + v_2 + v_3, P\rangle = 0$, and so it follows that 
       \[ \max_{i\in\{1,2,3\}}\{\langle v_i, P \rangle\} \geq \tfrac{1}{2} \abs{\langle v_j,P\rangle}. \]
    \end{proof}

\subsection{Reduced minima: Definition and Existence}
   
    \begin{lemma}\label{lem:red-minima}
        Let $L\subset\RR^2$ be a lattice. There is a unique triple of elements $(\rho_1(L),\rho_2(L),\rho_3(L))$ in $L$ such that for all $i\in\{1,2,3\}$
        \begin{align}\label{def:red-minima}
            \rho_i(L)\in K_i,\quad\langle v_i, \rho_i(L)\rangle = \max\{\langle v_i,P\rangle\st P\in \cone_i\cap L\} < 0
        \end{align}
        and
        \begin{align}\label{eq:sum-of-adm}
            \rho_1(L) + \rho_2(L) + \rho_3(L) = 0.
        \end{align}
        We call $\rho_i(L)$ the $i$-th reduced minimum of $L$ (with respect to $\tri$).
    \end{lemma}
    \begin{remark}
        For the proof of \Cref{thm:main}, it would suffice to establish the existence of reduced minima for almost every lattice $L\subset\RR^2$. Indeed, for almost every lattice, $L\cap\RR v_i=\{0\}$ for each $i=1,2,3$, and hence \eqref{def:red-minima} uniquely determines $\rho_i(L)$. The identity \eqref{eq:sum-of-adm} then follows from Case~2 of the proof below. For completeness, however, we establish \Cref{lem:red-minima} for all lattices.
    \end{remark}
    \begin{proof}   
        The existence of the reduced minima of $L$ is established as follows. Since $L$ is a lattice, $\cone_i\cap L$ is not empty for all $i\in\{1,2,3\}$.

        \medskip
        
        \textbf{Case 1:} There is an index $i\in\{1,2,3\}$ for which the set
        \[ R_i \defeq \left\{P\in \cone_i\cap L\st \langle v_i, P\rangle = \max\{\langle v_i,P'\rangle\st P'\in \cone_i\cap L\}\right\} \]
        has more than one element.

        It will become apparent from the proof below that in this case this index is unique. For ease of exposition, assume that $R_1$ contains two elements $P\neq Q$. Then, $\langle v_1, P-Q\rangle =0$ and by \eqref{eq:dir-identity} we have $\langle v_2, P-Q \rangle = -\langle v_3, P-Q \rangle \neq 0$. By exchanging $P$ and $Q$ if necessary, we may assume $\langle v_2,P-Q \rangle > 0$ and $\langle v_3,P-Q \rangle < 0$. Since $\cone_1$ is convex we have $\{Q + t(P-Q) \st t\in [0,1]\}\cap L\subset \cone_1$ and so we may further assume that $P-Q\in L$ is primitive. Pick the unique $m\in\NN$ with
        \[ -(m-1)\langle v_3,P-Q \rangle < \langle v_3, Q \rangle \leq -m\langle v_3,P-Q \rangle \]
        and define $Q'\defeq Q + m(P-Q)$. We claim that $R_2 = \{-Q'\}$ and $R_3 = \{P-Q\}$.
        
        First, we show that $R_3 = \{P-Q\}$. From
        \[ \langle v_1, P-Q \rangle=0, \quad \langle v_2, P-Q \rangle > 0\quad \text{and}\quad \langle v_3, P-Q \rangle < 0 \]
        it follows that $P-Q\in \cone_3\cap L$. Assume there is an element $P'\neq P-Q \in \cone_3\cap L$ such that $0 > \langle v_3, P'\rangle \geq \langle v_3,P-Q\rangle$. This implies that
        \[ \langle v_3, P' + Q \rangle \geq \langle v_3, P \rangle  > 0, \]
        where the last inequality holds as $P\in \cone_1 \subset H_3^{\circ}$, as well as
        \[ \langle v_2, P' + Q \rangle = \langle v_2, P' \rangle  + \langle v_2, Q \rangle > 0, \]
        as $P'\in \cone_3\subset H_2^\circ$ and $Q\in \cone_1\subset H_2$, so that $P' + Q\in \cone_1\cap L$ follows. Thus, we have
        \begin{align*}
            \langle v_1, P' + Q \rangle = \langle v_1, P' \rangle + \langle v_1, Q \rangle <0
        \end{align*}
        and either $\langle v_1, P' \rangle = 0$, which contradicts primitivity of $P-Q$ (as $P'$ is a lattice element contained in $\{t(P-Q)\st t\in[0,1)\}$), or $\langle v_1, P' \rangle > 0$, which is a contradiction to $Q\in R_1$. Thus, we get $R_3 = \{P-Q\}$.
        
        Next, we show that $R_2 = \{-Q'\}$. Notice that $-Q'\in \cone_2\cap L$ follows from
        \begin{align*}
            &\langle v_1, -Q'\rangle = -\langle v_1, Q\rangle - m\langle v_1, P-Q\rangle = -\langle v_1, Q\rangle > 0,\\
            &\langle v_2, -Q'\rangle = -\langle v_2, Q\rangle - m\langle v_2,P-Q\rangle < 0,\\
            &\langle v_3, -Q'\rangle = -\langle v_3, Q\rangle + m\langle v_3,Q-P\rangle \geq 0.
        \end{align*}
        The first inequality follows as $Q\in \cone_1$, the second as $Q\in \cone_1\subset H_2$ and $P-Q\in \cone_3\subset~H_2^{\circ}$ and the third by definition of $m$. Assume there is an element $P'\neq -Q' \in \cone_2\cap L$ such that $0 > \langle v_2,P'\rangle \geq \langle v_2, -Q'\rangle$. This implies that
        \[ \langle v_2, Q + P' + Q'\rangle = \langle v_2, Q \rangle + \langle v_2, P' + Q'\rangle \geq 0 \]
        as $Q\in \cone_1\subset H_2$. Further, notice that $Q' = (m-1)(P-Q) + P$ so we get
        \[ \langle v_3, Q + P' + Q'\rangle = \langle v_3, Q + (m-1)(P-Q) \rangle + \langle v_3, P' + P\rangle > 0 \]
        as $Q + (m-1)(P-Q)\in \cone_1\subset H_3^\circ$ by the choice of $m$, $P'\in \cone_2\subset H_3$ and $P\in \cone_1\subset H_3^\circ$. It follows that $Q + P' + Q' \in \cone_1\cap L$ and so
        \[ \langle v_1, Q \rangle > \langle v_1, Q + P' + Q' \rangle \Leftrightarrow \langle v_1, -P' \rangle > \langle v_1, Q' \rangle = \langle v_1, Q \rangle \]
        by the definition of $Q$ (the strict inequality holds by primitivity of $P-Q$ and the fact that $-Q' + (P-Q)\not\in H_3$). Now pick the unique $n\in\NN$ with
        \[ (n-1)\langle v_3, Q-P \rangle \leq \langle v_3, P' \rangle < n\langle v_3,Q-P \rangle \]
        so that $n(Q-P) - P' \in H_3^\circ$. Since
        \[ \langle v_1, n(Q-P) - P' \rangle = \langle v_1,- P' \rangle > \langle v_1, Q \rangle \]
        we cannot have $n(Q-P) - P' \in \cone_1 = H_2\cap H_3^\circ$ and thus $n(Q-P) - P' \not\in H_2$. It follows that $P' + n(P-Q) \in H_2^\circ$.
        Finally, observe that
        \[ \langle v_1, P' + n(P-Q) \rangle = \langle v_1, P' \rangle > 0, \]
        so we obtain $P' + n(P-Q)\in \cone_3=H_1\cap H_2^\circ$ and in particular $P' + n(P-Q)\neq P-Q$ as $\langle v_1, P' + n(P-Q)\rangle > 0$, but $\langle v_1, P-Q\rangle = 0$. However, then
        \[ \langle v_3, P' + n(P-Q) \rangle = \langle v_3, P' + (n-1)(P-Q) \rangle + \langle v_3, P-Q \rangle \geq \langle v_3, P-Q\rangle \]
        is a contradiction to $R_3 = \{P-Q\}$. Thus, $P'$ cannot exist so $R_2 = \{-Q'\}$ and the claim follows.  
        In this case, the reduced minima are
        \[ (\rho_1(L),\rho_2(L),\rho_3(L)) = (Q' - (P-Q), -Q', P-Q) \]
        and they satisfy $\rho_1(L) + \rho_2(L) + \rho_3(L) = 0$. This finishes the proof of existence in the case that one of the sets $R_1,R_2,R_3$ has more than one element. 

        \medskip
        \textbf{Case 2:} For all $i\in\{1,2,3\}$, $R_i$ contains a single element, which we denote by $\rho_i$.
        
        We claim that $\rho_1 + \rho_2 + \rho_3 = 0$, which shows that $(\rho_1(L), \rho_2(L), \rho_3(L))$ is well-defined. Notice that $\rho_1 +\rho_2\in H_3^{\circ}$ and since $\rho_1\in H_2$ and $\rho_2\in H_1^{\circ}$ we have
        \[ \langle v_2, \rho_1+\rho_2\rangle \geq \langle v_2,\rho_2\rangle,\quad\text{and}\quad \langle v_1, \rho_1+\rho_2\rangle > \langle v_1,\rho_2\rangle. \]
        It follows that $\rho_1+\rho_2\not\in \cone_1\sqcup \cone_2$, otherwise we get a contradiction to the definitions of $\rho_1$ or $\rho_2$. Hence, $\rho_1+\rho_2\in H_3^{\circ}\setminus(\cone_1\sqcup \cone_2) = -\cone_3$ and so $-(\rho_1+\rho_2)\in \cone_3$. Thus, either $\rho_3 = -(\rho_1+\rho_2)$ and we are done, or 
        \[ \langle v_3, \rho_3\rangle > \langle v_3, -(\rho_1 +\rho_2)\rangle. \]
        If $\rho_1+\rho_2+\rho_3\neq 0$, one can analogously obtain
        \[ \langle v_2, \rho_2\rangle > \langle v_2, -(\rho_1 +\rho_3)\rangle,\quad\text{and}\quad \langle v_1, \rho_1\rangle > \langle v_1, -(\rho_2 +\rho_3)\rangle. \]
        Summing up all these inequalities, we get
        \begin{align*}
             \langle v_1, \rho_1\rangle + \langle v_2, \rho_2\rangle + \langle v_3, \rho_3\rangle &> \langle v_1, -(\rho_2 +\rho_3)\rangle + \langle v_2, -(\rho_1 +\rho_3)\rangle + \langle v_3, -(\rho_1 +\rho_2)\rangle\\
             &= \langle v_2 + v_3, -\rho_1\rangle + \langle v_1 + v_3, -\rho_2\rangle + \langle v_1 + v_2, -\rho_3\rangle\\
             &= \langle v_1, \rho_1\rangle + \langle v_2, \rho_2\rangle + \langle v_3, \rho_3\rangle,
        \end{align*}
        where we used \eqref{eq:dir-identity}. This is a contradiction and so $\rho_1+\rho_2+\rho_3~=0$ follows.
       
    \end{proof}

\subsection{Reduced bases in $\tri$}

    \begin{definition}\label{def:red-basis}
        We call a basis $\{a_1,a_2\}$ of a lattice $L\subset\RR^2$ reduced, if
        \[ \{a_1,a_2\}\subset\{\pm\rho_1(L),\pm\rho_2(L),\pm\rho_3(L)\}. \]
    \end{definition}
    Observe that there are at most $12$ reduced bases for any lattice $L\subset\RR^2$.
    
    \begin{lemma}\label{lem:red-basis-in-tri}
        Let $L\subset\RR^2$ be a lattice. If for some $r>0$ and $s\in\RR^2$, the elements in $(-s + r\tri)\cap L$ generate the lattice $L$, then $(-s + r\tri)\cap L$ contains three points $P_1,P_2$ and $P_3$ such that their differences $\{P_2-P_1,P_3-P_1\}$ form a reduced basis.
    \end{lemma}
    \begin{proof}
        First, notice that since the elements in $(-s + r\tri)\cap L$ generate the lattice $L$, there exist $P_1,P_2,P_3\in (-s + r\tri)\cap L$ such that $\{P_2-P_1,P_3-P_1\}$ is a basis of $L$.
        
        Up to relabeling $P_1,P_2$ and $P_3$, we may assume that $P_2-P_1\in\cone = \cone_1\sqcup\cone_2\sqcup\cone_3$, say $P_2-P_1\in\cone_i$. 
        By definition of $\rho_i(L)$, we have that $0 > \langle v_i,\rho_i(L)\rangle \geq \langle v_i, P_2-P_1\rangle$. We claim that $P_1+\rho_i(L)\in (-s + r\tri)$. Indeed, we have
        \[ \langle v_i,s + P_1 + \rho_i(L)\rangle \geq \langle v_i,s + P_1 \rangle + \langle v_i, P_2-P_1\rangle = \langle v_i,s + P_2 \rangle > -\frac{r}{\sqrt{2}} \]
        since $P_2\in -s + r\tri$. Further, for $i\neq j\in\{1,2,3\}$ we have $\langle v_j,\rho_i(L)\rangle \geq 0$ and so
        \[ \langle v_j,s + P_1 + \rho_i(L)\rangle = \langle v_j,s + P_1 \rangle + \langle v_j, \rho_i(L)\rangle \geq \langle v_j,s + P_1 \rangle > -\frac{r}{\sqrt{2}} \]
        since $P_1\in -s + r\tri$. It follows that $P_1+\rho_i(L)\in (-s + r\tri)\cap L$ as claimed.

        Now one of $P_2-P_1$ and $P_3-P_1$ is not collinear with $\rho_i(L)$, thus at least one of $\{P_3-P_1,\rho_i(L)\}$ and $\{P_2-P_1,\rho_i(L)\}$ spans a finite index sublattice of $L$. By convexity of $(-s + r\tri)$ we can find a point $P\in(-s + r\tri)\cap L$ in the triangle spanned by $P_1$, $P_1+\rho_i(L)$ and $P_2$, or $P_3$, respectively, such that $P-P_1$ and $\rho_i(L)$ form a basis of $L$. To sum up, we may assume that we have $P_1,P_2,P_3\in(-s + r\tri)\cap L$, such that $\{P_2-P_1,P_3-P_1\}$ forms a basis of $L$ and that $P_2-P_1 = \rho_i(L)$ for some $i\in\{1,2,3\}$. We will now distinguish four cases
        \[ P_3-P_1 \in (K_j\sqcup K_{j'})\sqcup (-K_j\sqcup -K_{j'})\sqcup -K_i\sqcup K_i = \RR^2\setminus\{0\} \]
        for $j\neq j'\in\{1,2,3\}\setminus\{i\}$.
        
        Recall that for $e\in\{j,j'\}$ the set $\{\rho_i(L),\rho_e(L)\}$ is a basis of $L$ and that $\rho_{j'}(L) = -\rho_i(L) - \rho_j(L)$. Since $\{\rho_i(L),P_3-P_1\}$ is a basis  as well, we have
        \[ P_3 - P_1 = c_1\rho_e(L) + c_2\rho_i(L) \]
        for $c_1\in\{\pm1\}$ and some $c_2\in\ZZ$. Up to replacing $e\in\{j,j'\}$ with the unique element in $\{j,j'\}\setminus\{e\}$ and using $\rho_{j'}(L) = -\rho_i(L) - \rho_j(L)$ we may assume that $c_1=-1$.
        
        \textbf{Case 1:} $P_3-P_1\in K = K_j\sqcup K_{j'}$, say $P_3-P_1\in K_j$. 
        
        By the analogous argument as for $\rho_i(L)$, we deduce that $P_1+\rho_j(L)\in(-s+r\tri)\cap L$. However, then $P_1', P_2'\defeq P_1+\rho_i(L),P_3'\defeq P_1+\rho_j(L)\in(-s+r\tri)\cap L$ is a triple of points for which $\{P_2'-P_1',P_3'-P_1'\}\subseteq\{\pm\rho_1(L),\pm\rho_2(L),\pm\rho_3(L)\}$.

        \textbf{Case 2:} $P_3-P_1\in -K_j\sqcup -K_{j'}$, say $P_3-P_1\in -K_j$.

        We claim that $e=j$. Indeed, if $c_2 \geq 0$, then $P_3-P_1\in -K_j \subset -H_{j'}$ and $\rho_i(L)\in H_{j'}$ implies
        \[ 0 \geq \langle v_{j'}, P_3-P_1 \rangle = \langle v_{j'},-\rho_e(L)\rangle + c_2\langle v_{j'},\rho_i(L)\rangle \geq \langle v_{j'},-\rho_e(L)\rangle, \]
        and so $e=j$ as claimed. On the other hand, if $c_2 < 0$, then $P_3-P_1\in -K_j \subset H_j$ and $\rho_i(L)\in H_j$ implies
        \[ 0 \leq \langle v_j, P_3-P_1 \rangle = \langle v_j,-\rho_e(L)\rangle + c_2\langle v_j,\rho_i(L)\rangle \leq \langle v_j,-\rho_e(L)\rangle, \]
        and so $e\neq j'$, that is $e=j$ as claimed.
        
        We have $P_3-P_1\in -K_j$, and so $P_1-P_3\in K_j$. By the analogous argument as for $\rho_i(L)$, we deduce that $P\defeq P_3+\rho_j(L)= P_1 + c_2\rho_i(L)\in(-s+r\tri)\cap L$. By convexity of $(-s+r\tri)$ and since $P_1,P\in(-s+r\tri)$ we have that $P -\sgn(c_2)\rho_i(L)\in(-s+r\tri)\cap L$. Thus, $P_1'\defeq P, P_2'\defeq P-\sgn(c_2)\rho_i(L)$ and $P_3'\defeq P_3=P+\rho_e(L)$ forms a basis of $L$ for which $\{P_2'-P_1',P_3'-P_1'\}\subseteq\{\pm\rho_1(L),\pm\rho_2(L),\pm\rho_3(L)\}$.        
        
        \textbf{Case 3:} $P_3-P_1\in -K_i$.

        We claim that $c_2<0$. Indeed, if $c_2\geq 0$ then $P_3-P_1\in -K_i \subset H_i^{o}$ implies
        \[ 0 < \langle v_i, P_3 - P_1\rangle = \langle v_i, -\rho_e(L) + c_2\rho_i(L)\rangle \leq \abs{\langle v_i, \rho_e(L)\rangle} + \langle v_i,\rho_i(L)\rangle \leq 0, \]
        which is a contradiction, so $c_2 < 0$. However, then $P_3 - c_2\rho_i(L) = P_1  - \rho_e(L)$ is in $(-s+r\tri)$. Indeed, we have
        \[ \langle v_i, s + P_3 - c_2\rho_i(L) \rangle = \langle v_i, s + P_1\rangle - \langle v_i,\rho_e(L)\rangle \geq \langle v_i, s + P_1\rangle +\langle v_i,\rho_i(L)\rangle = \langle v_i,P_2\rangle \geq  -\frac{r}{\sqrt{2}}, \]
        since $P_2 \in (-s+r\tri)$ and as $\rho_i(L)\in K_i\subset H_j\cap H_{j'}$ we have
        \[ \langle v_j, s+ P_3 - c_2\rho_i(L) \rangle = \langle v_j, s+ P_3\rangle - c_2\langle v_j,\rho_i(L)\rangle \geq \langle v_j,s+P_3\rangle -\frac{r}{\sqrt{2}}, \]
        as well as
        \[ \langle v_{j'}, s+ P_3 - c_2\rho_i(L) \rangle = \langle v_{j'}, s+ P_3\rangle - c_2\langle v_{j'},\rho_i(L)\rangle \geq \langle v_{j'},s+P_3\rangle -\frac{r}{\sqrt{2}}, \]
        since $P_3 \in (-s+r\tri)$. The differences of $P_1'\defeq P_1, P_2'\defeq P_1+\rho_i(L)$ and $P_3'\defeq P_1-\rho_e(L)$ form a basis of $L$ such that $\{P_2'-P_1',P_3'-P_1'\}\subseteq\{\pm\rho_1(L),\pm\rho_2(L),\pm\rho_3(L)\}$.

        \textbf{Case 4:} $P_3-P_1\in K_i$.
        
        We claim that $c_2>0$. Indeed, if $c_2<0$ then $P_3-P_1\in K_i \subset -H_i^{o}$ implies
        \[ 0 > \langle v_i, P_3 - P_1\rangle = \langle v_i, -\rho_e(L) + c_2\rho_i(L)\rangle \geq \langle v_i, -\rho_e(L)-\rho_i(L)\rangle \geq 0, \]
        since $-\rho_e(L)-\rho_i(L)\in H_i$, which is a contradiction, so $c_2 \geq 0$.
        If $c_2 = 0$ then $P_3-P_1 = -\rho_e(L) \in -K_e$ which is a contradiction to $P_3-P_1\in K_i$.
        
        By definition of the reduced minima of $L$ we can pick $1\leq c_2'\leq c_2$ such that 
        \[ P_1' \defeq P_1 + c_2'\rho_i(L)\in(-s+r\tri)\cap L \text{ and } P_3-P_1' \in K_j\sqcup K_{j'}. \]
        Say $P_3-P_1'\in K_j$, then by the analogous argument as for $\rho_i(L)$, we get $P_1' + \rho_j(L)\in(-s+r\tri)\cap L$. Thus,  the differences of $P_1', P_2'\defeq P_1'-\rho_i(L)$ and $P_3'\defeq P_1'+\rho_j(L)$ form a basis of $L$ for which $\{P_2'-P_1',P_3'-P_1'\}\subseteq\{\pm\rho_1(L),\pm\rho_2(L),\pm\rho_3(L)\}$. Notice that $P_2'=P_1+(c_2'-1)\rho_i(L)\in (-s+r\tri)\cap L$, by convexity of $-s+r\tri$ and since $P_1,P_1+c_2'\rho_i(L)\in-s+r\tri$.

        This finishes the proof of the \namecref{lem:red-basis-in-tri}.
    \end{proof}

    
\section{Auxiliary results}\label{sec:auxiliary-results}

\subsection{Geometry of triangles}
    For a finite set of points $\mathcal{P}\subseteq\RR^2$ we define 
    \begin{align}
        \label{eq:def delta P}
        \tri(\mathcal{P}) \defeq \bigcap_{P\in\mathcal{P}}(-P+\tri),
    \end{align}
    as well as
    \[ s(\mathcal{P}) \defeq \max\left\{0, 1 -  \frac{\sqrt{2}}{3}\left(\sum_{i=1}^3 \nu_i(\mathcal{P})\right)\right\},\quad\text{and}\quad C(\mathcal{P}) \defeq \frac{2}{3}\left(\sum_{i=1}^3 \nu_i(\mathcal{P})v_i \right) \]
    where 
    \begin{align}
        \label{eq:def nu i}
        \nu_i(\mathcal{P}) \defeq \max_{P\in\mathcal{P}}\{\langle v_{i},-P\rangle\}
    \end{align}
    for all $i\in\{1,2,3\}$. Note that $s(\mathcal{P})$ is translation invariant and $C(\mathcal{P})$ is translation equivariant, \ie 
    \[ s(\mathcal{P} - P) = s(\mathcal{P}), \qquad C(\mathcal{P}-P) = C(\mathcal{P}) - P, \] 
    for any $P\in\RR^2$.
    
    \begin{lemma}\label{lem:intersec-tri}
        Let $\mathcal{P}\subseteq\RR^2$ be a finite set of points. Then $\tri(\mathcal{P})\neq\emptyset$ if and only if $s\defeq s(\mathcal{P})>0$. In this case, $\tri(\mathcal{P}) = C+s\tri$ is an equilateral triangle with center $C\defeq C(\mathcal{P})\in\RR^2$.
        Moreover, the Lebesgue measure of $\tri(\mathcal{P})$ is given by 
        \[ \Leb(\tri(\mathcal{P})) = \frac{3\sqrt{3}}{2}s(\mathcal{P})^2. \]
    \end{lemma}
    \begin{proof}
        Notice that for any $P\in\RR^2$, $p\in -P +\tri$ is equivalent to $p+P\in\tri$, thus, by \eqref{eq:desc-tri},
        \[ -P + \tri = \left\{p\in\RR^2\st\forall i \in\{1,2,3\}~ \langle v_i, p+P \rangle > -\frac{1}{\sqrt{2}} \right\}. \] 
        The intersection of two sets of the form $-P+r\tri$ and $-P'+r'\tri$ is easily seen to be either empty or an equilateral triangle, thus $\tri(\mathcal{P})$ is an equilateral triangle as well.
        
        If $\tri(\mathcal{P})\neq\emptyset$, we may pick $p\in\tri(\mathcal{P})$ and conclude that $\langle v_i, p \rangle > \langle v_i, -P\rangle - \frac{1}{\sqrt{2}}$ for all $P\in\mathcal{P}$, hence $\langle v_i, p \rangle > \nu_i(\mathcal{P}) -\frac{1}{\sqrt{2}}$. We obtain that $s>0$, since by \eqref{eq:dir-identity}
        \[ 0 = \langle v_1 + v_2 + v_3, p \rangle > \nu_1(\mathcal{P}) + \nu_2(\mathcal{P}) + \nu_3(\mathcal{P}) - \frac{3}{\sqrt{2}} = \frac{3}{\sqrt{2}}\left(\frac{\sqrt{2}}{3}\left(\sum_{i=1}^3 \nu_i(\mathcal{P})\right) -1\right) = -\frac{3}{\sqrt{2}}s. \] 
        
        Assume now that $s>0$, then for all $i\in\{1,2,3\}$ and $P\in\mathcal{P}$ we get, using \eqref{eq:dir-pair-corr}, that
        \begin{align*}
            \langle v_i, C + P \rangle &= \frac{2}{3}\left(\sum_{j=1}^3 \nu_i(\mathcal{P})\langle v_i,v_j\rangle \right) +\langle v_i, P \rangle\\
            &=  \frac{2}{3}\nu_i(\mathcal{P}) - \frac{1}{3}\left(\sum_{j\neq i}\nu_j(\mathcal{P})\right) + \langle v_i, P\rangle\\
            &= (\nu_i(\mathcal{P}) - \langle v_i, -P\rangle) - \frac{1}{3}\left(\sum_{i=1}^3 \nu_i(\mathcal{P})\right)\\
            &\geq 0 - \frac{1}{3}\cdot\frac{3}{\sqrt{2}}(1-s) = - \frac{1}{\sqrt{2}} + \frac{s}{\sqrt{2}} > -\frac{1}{\sqrt{2}},
        \end{align*}
        where the last inequality holds as $s>0$. It follows that $C\in\tri(\mathcal{P})$ and so $\tri(\mathcal{P}) \neq \emptyset$.
        
        Let $i\in\{1,2,3\}$ be given. Then one vertex $P_i$ of $\tri(\mathcal{P})$ is uniquely determined by satisfying
        \[ \langle v_j,P_i \rangle = \nu_j(\mathcal{P})\quad\text{and}\quad \langle v_e,P_i \rangle = \nu_e(\mathcal{P}), \]
        for $j\neq e\in\{1,2,3\}\setminus\{i\}$, thus, by \eqref{eq:point-desc}, 
        \[ P_i = \frac{2}{3}\big((2\nu_j(\mathcal{P})+\nu_e(\mathcal{P}))v_j + (\nu_j(\mathcal{P})+2\nu_e(\mathcal{P}))v_e\big). \]
        The center of $\tri(\mathcal{P})$ is then given by $\frac{1}{3}(P_1 + P_2 + P_3)$, which can be seen to be equal to $C$ by a straightforward calculation.

        Lastly, the formula for $\Leb(\tri(\mathcal{P}))$ follows as the side length of $\tri(\mathcal{P})$ is equal to $\sqrt{6}s(\mathcal{P})$ and so
        \[ \Leb(\tri(\mathcal{P})) = \frac{\sqrt{3}}{4}(\sqrt{6}s(\mathcal{P}))^2 = \frac{3\sqrt{3}}{2}s(\mathcal{P})^2. \]
    \end{proof}

\subsection{Farey fractions}

   We recall standard properties of Farey fractions (cf. \cite[Chapter III]{HW08}).
    \begin{lemma}\label{lem:Farey-fractions}
        Let $F_n = \{\frac{h}{k} \st 0\leq h\leq k\leq n,~\gcd(h,k) = 1\}$ be the set of Farey fractions of order~$n$. If $\frac{h'}{k'}<\frac{h}{k}<\frac{h''}{k''}$ are three successive terms of $F_n$, then
        \[ \frac{h}{k} = \frac{h' + h''}{k' + k''}. \]
        Moreover, we have $k'h - h' k = 1$ and $kh'' - h k''= 1$, hence
        \[ \frac{h}{k} - \frac{h'}{k'} = \frac{1}{kk'}\qquad \text{ and } \qquad \frac{h''}{k''} - \frac{h}{k} = \frac{1}{kk''}. \]
    \end{lemma}

    \begin{remark}\label{rem:all-coprime-classes}
        Notice that $k'h - h' k = 1$ implies that $k'$ is the multiplicative inverse of $h$ modulo $k$. Similarly, $k''$ is the multiplicative inverse of $-h$ modulo $k$. In particular, if $k$ is fixed and $h\in\{1,\dots,k-1\}$ runs through all co-prime residue classes of $k$, so does $k'$.
    \end{remark}


\section{Proof of Proposition~\ref{prop:upper-bound-level-measure} and~\ref{prop:lower-bound-level-measure}}\label{sec:proof-level-measure}
\subsection{Notation and Setup}
\label{subsec:Notation and Setup}
    
    Fix $A'\in\GL_2(\RR)$. Let $\fund_{A'}\subset\RR^2$ denote a fundamental domain of $A'\ZZ^2\subset\RR^2$ and let 
    \begin{align}
        \label{eq:def fundamental a k}
        \fund_{A'}^{(k)} = \{b\in\fund_{A'}\st \level(A',b) = k\}.
    \end{align}
    By \Cref{lem:red-basis-in-tri} we know that if $b\in\RR^2$ satisfies $\level(A',b) = k$, then there are three points $P'_1,P'_2,P'_3\in A'\ZZ^2\cap(-kb + k\tri)$ such that $\{P'_2-P'_1, P'_3-P'_1\}$ is a reduced bases of $A'\ZZ^2$. Thus, for every reduced basis $\{a_1,a_2\}$ we define        
    \[ \fund_{\{a_1,a_2\}}^{(k)} \defeq \{b\in\fund_{A'}^{(k)}\st \exists P'_1,P'_2,P'_3\in A'\ZZ^2\cap(-kb + k\tri); ~\{P'_2-P'_1,P'_3-P'_1\}=\{a_1,a_2\}\} \]
    and we note that
    \[ \fund_{A'}^{(k)} \subseteq \bigcup_{\{a_1,a_2\}} \fund_{\{a_1,a_2\}}^{(k)}, \]
    where the union is taken over all reduced bases, of which there are at most $12$. We will proceed by analyzing $\fund_{\{a_1,a_2\}}^{(k)}$ for every reduced basis $\{a_1,a_2\}$.
    
    We decompose the set $\fund_{\{a_1,a_2\}}^{(k)}$ according to the representation of $P'_1$ in terms of the reduced basis $a_1\defeq P'_2-P'_1,a_2\defeq P'_3-P'_1$. Assume that $P'_1 = c_1a_1 + c_2a_2$, then by \Cref{lem:desc-of-basis-I} we have $\gcd(c_1,k) = \gcd(c_2,k) = 1$, and at least one of
    \[ c_1 + 1 \equiv 0 \pmod{k}, \qquad c_2 + 1 \equiv 0 \pmod{k}, \qquad c_1 + c_2 \equiv 0 \pmod{k} \]
    holds. It follows that $\fund_{\{a_1,a_2\}}^{(k)}$ can be divided into at most $3\phi(k)$ many subsets of $\RR^2$ under the natural projection from $\RR^2$ to $\fund_{A'}$, where $\phi$ denotes Euler's totient function. Indeed, there are at most $3\phi(k)$ possibilities for elements $c_1'\in\{0,\dots,k-1\}$ and $c_2'\in\{-k,\dots,-1\}$ satisfying $\gcd(c_1',k) = \gcd(c_2',k) =1$, and at least one of
    \[ c_1' + 1 \equiv 0 \pmod{k}, \qquad c_2' + 1 \equiv 0 \pmod{k}, \qquad c_1' + c_2' \equiv 0 \pmod{k}. \]
    For each of these choices we define
    \[ B_{\{a_1,a_2\}}(c_1',c_2') \defeq \left\{\begin{aligned}b\in \RR^2  \st &P'_1,P'_2,P'_3\in A'\ZZ^2\cap(-kb + k\tri), \text{ where}\\~& P'_1 = c_1'a_1 + c_2'a_2;~ a_1=P'_2-P'_1,~a_2=P'_3-P'_1~\end{aligned}\right\}. \]
    Then
    \begin{align}\label{eq:first-dec}
        \fund_{\{a_1,a_2\}}^{(k)} \subseteq \bigcup_{(c_1',c_2')\text{ eligible}} B_{\{a_1,a_2\}}(c_1',c_2') + A'\ZZ^2
    \end{align}
    where the union is taken over all eligible pairs $c_1',c_2'$ as discussed above. To see this, assume now that $b\in\fund_{\{a_1,a_2\}}^{(k)}$ and pick $P_1',P_2',P_3'$ so that $a_1\defeq P'_2-P'_1,a_2\defeq P'_3-P'_1$ and that $P'_1 = c_1a_1 + c_2a_2$ with $\gcd(c_1,k) = \gcd(c_2,k) = 1$, and at least one of
    \[ c_1 + 1 \equiv 0 \pmod{k}, \qquad c_2 + 1 \equiv 0 \pmod{k}, \qquad c_1 + c_2 \equiv 0 \pmod{k} \]
    holds. Define $c_1' = c_1\mod k\in\{0,\dots,k-1\}$ and $c_2' = (c_2 \mod k) - k\in\{-k,\dots,-1\}$ and set $b' = b - \frac{1}{k}(c_1-c_1')a_1 + \frac{1}{k}(c_2-c_2')a_2 \in b + A'\ZZ^2$. Then,
    \[ c_1'a_1 + c_2'a_2 \in A'\ZZ^2 \cap (-kb' + k\tri) = (A'\ZZ^2 \cap (-kb + (c_1-c_1')a_1 + (c_2-c_2')a_2 + k\tri)  \]
    as well as $(c_1'+1)a_1 + c_2'a_2,c_1'a_1 + (c_2'+1)a_2 \in A'\ZZ^2 \cap (-kb' + k\tri)$. It follows that $b'\in B_{\{a_1,a_2\}}(c_1',c_2')$.

\medskip

    Fix a reduced basis $\{a_1,a_2\}$ and an eligible pair $c_1',c_2'$ as discussed above. By relabeling $P'_1,P'_2,P'_3$ and redefining $c_1'$ and $c_2'$ accordingly, we may assume that $c_2'=-1$. 
    Indeed, if $c_1'+1\equiv0\mod k$ we exchange the labels of $P'_2$ and $P'_3$, $a_1$ and $a_2$, as well as $c_1'$ and $c_2'$. If $c_1'+c_2'\equiv 0 \mod k$ we switch the labels of $P'_1$ and $P'_2$. This amounts to exchanging $\{a_1,a_2\}$ with the basis $\{a_1-a_2,-a_2\}$. Then $P'_2 = c_1'(a_1-a_2) + (-c_1'-c_2'-1)(-a_2)$ will be relabeled to $P'_1$ and $(-c_1'-c_2'-1) + 1\equiv0\mod k$. Hence, we have
    \begin{align}\label{eq:second-dec}
        \bigcup_{(c_1',c_2')\text{ eligible}} B_{\{a_1,a_2\}}(c_1',c_2') = \bigcup_{\gcd(c_1',k)=1}\left(\overline{B}_{a_1,a_2}(c_1') \cup \overline{B}_{a_2,a_1}(c_1') \cup \overline{B}_{a_1-a_2,-a_2}(c_1')\right),
    \end{align}
    where
    \[ \overline{B}_{x_1,x_2}(c_1') \defeq \left\{\begin{aligned}b\in \RR^2\st &P'_1,P'_2,P'_3\in A'\ZZ^2\cap(-kb + k\tri), \text{ where}\\~& P'_1 = c_1'x_1 -x_2;~ x_1=P'_2-P'_1,~x_2=P'_3-P'_1~\end{aligned}\right\}. \]
    for any reduced basis $\{a_1,a_2\}$ of $A'\ZZ^2$ and $c_1'\in\{0,\dots,k-1\}$. Notice that the set $\overline{B}_{a_1,a_2}(c_1')$ depends on the order of the elements $a_1$ and $a_2$, in contrast to $B_{\{a_1,a_2\}}(c_1',c_2')$.

    We will now proceed by bounding the Lebesgue measure of $\overline{B}_{a_1,a_2}(c_1')$ for any fixed $a_1,a_2\in A'\ZZ^2$ and $c_1'\in\{0,\dots,k-1\}$ with $\gcd(c_1',k)=1$.
    Define
    \[ P'_1\defeq c_1'a_1 + (k-1)a_2,\quad P'_2 \defeq P'_1 +a_1,\quad P'_3 \defeq P'_1 + a_2 \]
    as well as the three points
    $Q_1, Q_2, Q_3 \in \bigcup_{\ell=1}^{k-1} \frac{1}{\ell} A\mathbb{Z}^2$
    as follows
    \begin{equation}\label{eq:def-of-Q_i}
        \begin{alignedat}{2}
            Q_1 &\defeq \frac{h''}{k''} a_1,\quad  Q_2 &\defeq \frac{h'}{k'} a_1,\quad  Q_3 &\defeq \frac{c_1'}{k-1} a_1 - \frac{1}{k-1} a_2,
        \end{alignedat}
    \end{equation}
    where $\frac{h'}{k'} < \frac{c_1'}{k} < \frac{h''}{k''}$ are three successive Farey fractions (see \Cref{lem:Farey-fractions}).
    Finally, set
    \begin{align}\label{def:Ps-Qs}
        P_i \defeq \frac{1}{k} P'_i \quad (i=1,2,3),
        \qquad
        \mathcal{P} = \{P_1,P_2,P_3\},
        \qquad
        \mathcal{Q} = \{Q_1,Q_2,Q_3\}.
    \end{align}
    Using \Cref{lem:level-equivalence}, we notice that
    \begin{align}\label{eq:cond-on-b}
        \overline{B}_{a_1,a_2}(c_1') \subseteq \tri(\mathcal{P}) \setminus\left(\bigcup_{Q\in\mathcal{Q}}\tri(\mathcal{P}\cup\{ Q\})\right),
    \end{align}
    where $\tri(\cdot)$ is defined in \eqref{eq:def delta P}.
    Indeed, observe that $P_1',P_2',P_3'\in(-kb + k\tri)$ is equivalent to
    \[ b\in (-\tfrac{1}{k}P_1' + \tri) \cap  (-\tfrac{1}{k}P_2' + \tri) \cap (-\tfrac{1}{k}P_3' + \tri) = \tri\left(\mathcal{P}\right). \]
    Assume that $b\in\overline{B}_{a_1,a_2}(c_1') $. Then $\level(A',b)=k$ and so $b\not\in (-Q + \tri)$ for any $Q\in\bigcup_{\ell=1}^{k-1}\tfrac{1}{\ell}A\ZZ^2$ by \Cref{lem:level-equivalence}. In particular, we obtain
    \[ b\not\in(-Q_1+\tri)\cup(-Q_2+\tri)\cup(-Q_3+\tri). \]
    Thus, $b\in \tri(\mathcal{P}) \setminus\left(\bigcup_{Q\in\mathcal{Q}}\tri(\mathcal{P}\cup\{ Q\})\right)$ and so the inclusion $\subseteq$ follows. 

\subsection{First reduction}
    \begin{lemma}\label{lem:directions-of-reduced-elements}
        Recall the definitions of $\mathcal{P}$ and $\mathcal{Q}$ in \eqref{def:Ps-Qs} and assume that
        \[ \tri(\mathcal{P}) \setminus\left(\bigcup_{Q\in\mathcal{Q}}\tri(\mathcal{P}\cup\{ Q\})\right) \neq \emptyset. \]
        Then, there are pairwise different indices $j_1,j_2,j_3\in\{1,2,3\}$ such that for all $i\in\{1,2,3\}$
        \[ \nu_{j_i}(\mathcal{P}\cup\{Q_i\}) = \langle v_{j_i}, -Q_i\rangle > \nu_{j_i}(\mathcal{P})\quad\text{and}\quad \nu_{j}(\mathcal{P}\cup\{Q_i\}) = \nu_{j}(\mathcal{P}) \]
        for all $j\neq j_i\in\{1,2,3\}$.
        In fact, it holds that 
        \begin{align}\label{eq:bounds-on-alpha_i}
            0 < \nu_{j_1}(\mathcal{P}\cup\{Q_1\}) - \nu_{j_1}(\mathcal{P}) &\leq \frac{1}{kk'}\langle v_{j_1}, -a_1 \rangle,\\
            0 < \nu_{j_2}(\mathcal{P}\cup\{Q_2\}) - \nu_{j_2}(\mathcal{P}) &\leq \frac{1}{kk''}\langle v_{j_2}, a_1 \rangle,\\
            0 < \nu_{j_3}(\mathcal{P}\cup\{Q_3\}) - \nu_{j_3}(\mathcal{P}) &\leq \frac{1}{k(k-1)}\langle v_{j_3}, a_2 \rangle.
        \end{align}
    \end{lemma}
    \begin{proof}
        By assumption, we have $\tri(\mathcal{P}\cup\{Q_i\}) \subsetneq \tri(\mathcal{P})$ for any $i\in\{1,2,3\}$. Thus, by \Cref{lem:intersec-tri} there exists at least one index $j_i\in\{1,2,3\}$ for which $\nu_{j_i}(\mathcal{P}\cup\{Q_i\}) > \nu_{j_i}(\mathcal{P})$ and evidently for any such index $\nu_{j_i}(\mathcal{P}\cup\{Q_i\}) = \langle v_{j_i}, -Q_i\rangle$ holds. This shows existence of the indices $j_1,j_2,j_3\in\{1,2,3\}$ as described above. Next we will show that these indices are pairwise different, which also proves that $\nu_{j}(\mathcal{P}\cup\{Q_i\}) = \nu_{j}(\mathcal{P})$ for all $j\neq j_i\in\{1,2,3\}$ and all $i\in\{1,2,3\}$.
    
        First, we consider $i=1$. Notice that
        \begin{align}
            -Q_1 &= -\frac{h''}{k''}a_1\notag\\
            &= -\frac{1}{k}P_1 - \frac{1}{kk''}a_1 - \frac{1}{k}a_2\notag\\
            &= -\frac{1}{k}P_2 + \frac{k''-1}{kk''}a_1 - \frac{1}{k}a_2\notag\\
            &= -\frac{1}{k}P_3 - \frac{1}{kk''}a_1,\label{eq:Q_1-P_3}
        \end{align}
        where we used that $\frac{c_1'}{k} - \frac{h''}{k''} = -\frac{1}{kk''}$ (\Cref{lem:Farey-fractions}). Thus, $\langle v_{j_1}, -Q_1 \rangle > \nu_{j_1}(\mathcal{P})$ is equivalent to
        \begin{align}
            -\langle v_{j_1}, a_1\rangle &> k''\langle v_{j_1}, a_2 \rangle,\notag\\
            (k''-1)\langle v_{j_1}, a_1\rangle & > k''\langle v_{j_1}, a_2 \rangle,\label{eq:Q_1-P_2>0}\\
            0 &> \langle v_{j_1}, a_1\rangle.\label{eq:Q_1-P_3>0}
        \end{align}
        Since $1\leq k''\leq k-1$, combining \eqref{eq:Q_1-P_2>0} and \eqref{eq:Q_1-P_3>0}  we obtain $0>\langle v_{j_1}, a_2 \rangle$.

        Next, we consider $i=2$. Notice that 
        \begin{align}
            -Q_2 &= -\frac{h'}{k'}a_1\notag\\
            &= -\frac{1}{k}P_1 + \frac{1}{kk'}a_1 - \frac{1}{k}a_2\notag\\
            &= -\frac{1}{k}P_2 + \frac{k'+1}{kk'}a_1 - \frac{1}{k}a_2\notag\\
            &= -\frac{1}{k}P_3 + \frac{1}{kk'}a_1,\label{eq:Q_2-P_3}
        \end{align}
        where we used that $\frac{c_1'}{k} - \frac{h'}{k'} = \frac{1}{kk'}$ (\Cref{lem:Farey-fractions}). Thus, $\langle v_{j_2}, -Q_2 \rangle > \nu_{j_2}(\mathcal{P})$ is equivalent to
        \begin{align}
            \langle v_{j_2}, a_1\rangle &> k'\langle v_{j_2}, a_2 \rangle,\label{eq:Q_2-P_1>0} \\
            (k'+1)\langle v_{j_2}, a_1\rangle & > k'\langle v_{j_2}, a_2 \rangle,\notag\\
            \langle v_{j_2}, a_1\rangle &>0.\label{eq:Q_2-P_3>0}
        \end{align}
        From \eqref{eq:Q_1-P_3>0} and \eqref{eq:Q_2-P_3>0} it immediately follows that $j_2$ is different from $j_1$.
        
        Last, we consider $i=3$. Notice that
        \begin{align}
            -Q_3 &= -\frac{c_1'}{k-1}a_1  +\frac{1}{k-1}a_2\notag\\
            &= -\frac{1}{k}P_1 - \frac{c_1'}{k(k-1)}a_1 + \frac{1}{k(k-1)}a_2\label{eq:Q_3-P_1}\\
            &= -\frac{1}{k}P_2 + \frac{(k-1) - (c_1')}{k(k-1)} a_1 + \frac{1}{k(k-1)}a_2\label{eq:Q_3-P_2}\\
            &= -\frac{1}{k}P_3 - \frac{c_1'}{k(k-1)}a_1 + \frac{1}{k-1}a_2.\notag
        \end{align}
        Thus, $\langle v_{j_3}, -Q_3 \rangle > \nu_{j_3}(\mathcal{P})$  is equivalent to
        \begin{align}
            \langle v_{j_3}, a_2\rangle &> c_1'\langle v_{j_3}, a_1 \rangle,\label{eq:Q_3-P_1>0}\\
            \langle v_{j_3}, a_2\rangle & > (c_1' - (k-1))\langle v_{j_3}, a_1 \rangle,\label{eq:Q_3-P_2>0}\\
            k\langle v_{j_3}, a_2\rangle & >c_1'\langle v_{j_3}, a_1 \rangle.\notag
        \end{align}
        As $\gcd(c_1,k)=1$, we have $c_1' \in \{1,\dots k-1\}$ and so $\langle v_{j_3}, a_2\rangle > 0$ since the right hand side of $\eqref{eq:Q_3-P_1>0}$ or $\eqref{eq:Q_3-P_2>0}$ is non-negative. We deduce that $j_3\neq j_1$ since $0>\langle v_{j_1},a_2\rangle$ holds. It remains to prove that $j_2\neq j_3$. Assume otherwise, that is $j_2=j_3$. Observe that the inequalities \eqref{eq:Q_2-P_1>0} and \eqref{eq:Q_3-P_1>0} are then incompatible. Indeed,
        \[ \langle v_{j_3}, a_2\rangle \stackrel{\eqref{eq:Q_3-P_1>0}}{>} c_1'\langle v_{j_3}, a_1\rangle = c_1'\langle v_{j_2}, a_1\rangle \stackrel{\eqref{eq:Q_2-P_1>0}}{>} c_1'k'\langle v_{j_2}, a_2\rangle \]
        and since $\langle v_{j_2}, a_1 \rangle > 0$ this implies that $1 > c_1' k'$ which is a contradiction to $1\leq c_1', k'\leq k-1$. Hence, $j_2$ and $j_3$ are different.       

        Finally, we prove the bounds on $\nu_{j_i}(\mathcal{P}\cup\{Q_i\})-\nu_{j_i}(\mathcal{P})$. The lower bounds are obvious. It remains to verify the upper bounds. We can use \eqref{eq:Q_1-P_3} and \eqref{eq:Q_2-P_3} to deduce
        \begin{align*}
            \nu_{j_1}(\mathcal{P}\cup\{Q_1\}) - \nu_{j_1}(\mathcal{P}) &\leq \langle v_{j_1}, -Q_1\rangle -\langle v_{j_1}, -P_3\rangle \leq \frac{1}{kk'}\langle v_{j_1}, -a_1\rangle,\\
            \nu_{j_2}(\mathcal{P}\cup\{Q_2\}) - \nu_{j_2}(\mathcal{P}) &\leq \langle v_{j_2}, -Q_2\rangle -\langle v_{j_2}, -P_3\rangle \leq \frac{1}{kk''}\langle v_{j_2}, a_1\rangle.
        \end{align*}
        Moreover, we use either \eqref{eq:Q_3-P_1} or \eqref{eq:Q_3-P_2}, that is we choose $j=1$ or $2$ depending on whether $\langle v_{j_3}, a_1\rangle \geq 0$ or $\langle v_{j_3}, a_1\rangle<0$, to deduce
        \[ \langle v_{j_3}, -Q_3\rangle - \nu_{j_3}(\mathcal{P}) \leq \langle v_{j_3}, -Q_3\rangle -\langle v_{j_3}, -P_j\rangle \leq \frac{1}{k(k-1)}\langle v_{j_3}, a_2\rangle. \]
    \end{proof}

\subsection{Proof of \Cref{prop:upper-bound-level-measure}}
    \begin{proof}
    \noindent
    \textbf{Step 1:}
        Following the discussion in Section~\ref{subsec:Notation and Setup}, it suffices to give an upper bound on the Lebesgue measure of the right hand side in \eqref{eq:cond-on-b}. If the set is empty, then the bound is trivial. Hence, throughout the proof we assume that
        \begin{align}
            \label{eq:a 1}
            \tri(\mathcal{P}) \setminus\left(\bigcup_{Q\in\mathcal{Q}}\tri(\mathcal{P}\cup\{ Q\})\right) \neq\emptyset.
        \end{align}
        By the inclusion-exclusion principle, it follows that
        \begin{align} \label{eq:incl-excl}
        \Leb\left(\tri(\mathcal{P})\setminus\bigcup_{Q\in\mathcal{Q}}\tri(\mathcal{P}\cup\{Q\})\right)=&\,\Leb(\tri(\mathcal{P}))-\sum_{i=1}^{3}\Leb(\tri(\mathcal{P}\cup\{Q_i\})) \notag\\
            &+\sum_{1\le i<j\le 3}\Leb(\tri(\mathcal{P}\cup\{Q_i,Q_j\}))-\Leb(\tri(\mathcal{P}\cup\mathcal{Q})).
        \end{align}
        To compute the right hand side of \eqref{eq:incl-excl}, we use \Cref{lem:intersec-tri} and compute
        \[ s(\mathcal{P}),\quad s(\mathcal{P}\cup\{Q_i\}), \quad s(\mathcal{P}\cup\{Q_i,Q_j\}),\quad s(\mathcal{P}\cup\mathcal{Q}) \]
        for all $i\neq j\in\{1,2,3\}$.
        
        By \eqref{eq:a 1}, we have $\tri(\mathcal{P})\neq \emptyset$, or equivalently $s(\mathcal{P}) > 0$ (\Cref{lem:intersec-tri}). Since $s(\mathcal{P})$ is translation invariant we have
        \begin{align*} 
            s(\mathcal{P}) = s(\mathcal{P} - \tfrac{1}{k}P_1) &= 1 - \frac{\sqrt{2}}{3}\left(\sum_{i=1}^3\nu_i(\left\{0,\tfrac{1}{k}a_1,\tfrac{1}{k}a_2\right\})\right).
        \end{align*}
        By \Cref{lem:directions-of-reduced-elements}, there exist (unique) pairwise different indices $j_1,j_2,j_3\in\{1,2,3\}$ such that for all $i\in\{1,2,3\}$ and $j\neq j_i\in\{1,2,3\}$
        \[ \langle v_{j_i}, -Q_i \rangle > \nu_{j_i}(\mathcal{P}) \quad \text{and}\quad \langle v_{j}, -Q_i \rangle \leq \nu_j(\mathcal{P}). \]
        By the uniqueness of $j_i$, we get
        \begin{align*}
            s(\mathcal{P}\cup\{Q_i\}) &= \max\left\{s(\mathcal{P}) - \frac{\sqrt{2}}{3}\left(\langle v_{j_i},-Q_i\rangle - \nu_{j_i}(\mathcal{P})\right),0\right\},
        \end{align*}
        for all $i\in\{1,2,3\}$, and since $j_1,j_2$ and $j_3$ are pairwise different we get
        \begin{align*}
            s(\mathcal{P}\cup\{Q_i,Q_j\}) &= \max\left\{s(\mathcal{P}\cup\{Q_i\}) + s(\mathcal{P}\cup\{Q_j\}) - s(\mathcal{P}),0\right\}
        \end{align*}
        for all $i\neq j\in\{1,2,3\}$.
        
        Finally, if $s(\mathcal{P}\cup\mathcal{Q}) > 0$, or equivalently $\tri(\mathcal{P}\cup\mathcal{Q}) \neq \emptyset$, then $\bigcup_{i\in\{1,2,3\}}\tri(\mathcal{P}\cup\{Q_i\})$ covers $\tri(\mathcal{P})$ and so $\Leb(\overline{B}_{a_1,a_2}(c_1')) = 0$ follows.
        Indeed, if there is $P\in\tri(\mathcal{P})$ so that for all $i\in\{1,2,3\}$ we have $P\not\in\tri(\mathcal{P}\cup\{Q_i\})$, then $\langle v_{j_i}, P-Q_i\rangle > \frac{1}{\sqrt{2}}$, and so
        \begin{align*}
            \frac{3}{\sqrt{2}} < \sum_{i\in\{1,2,3\}}\langle v_{j_i}, P-Q_i\rangle&= \sum_{i\in\{1,2,3\}}\langle v_{j_i}, P\rangle + \sum_{i\in\{1,2,3\}}\langle v_{j_i}, -Q_i\rangle\\
            &= 0 + \sum_{i\in\{1,2,3\}}\langle v_{j_i}, -Q_i\rangle \leq \sum_{i\in\{1,2,3\}}\nu_{j_i}(\mathcal{P}\cup\mathcal{Q}),
        \end{align*} 
        which is a contradiction to 
        \[ s(\mathcal{P}\cup\mathcal{Q}) = 1 - \frac{\sqrt{2}}{3}\left(\sum_{i\in\{1,2,3\}}\nu_{j_i}\mathcal{P}\cup\mathcal{Q})\right) > 0. \]
        Thus, we will further assume that $s(\mathcal{P}\cup\mathcal{Q}) = 0$.

        \textbf{Step 2:}
        So far we only considered a fixed matrix $A'\in\SL_2(\RR)$. In order to upper bound
        \[ \int_{\RR_{>0}}\Leb(\fund_{r^{1/2}A'}^{(k)})\frac{\dd r}{r} \]
        we need to review the discussion above for varying parameters $r$. First, notice that if $\{a_1,a_2\}$ is a reduced basis for $A'\ZZ^2$, then $\{r^{1/2}a_1, r^{1/2}a_2\}$ is a reduced basis for $r^{1/2}A'\ZZ^2$. The definitions of $P'_1,P'_2$ and $P'_3$ and hence also $P_1,P_2$ and $P_3$, are linear in the basis $\{a_1,a_2\}$, thus if we define $r^{1/2}\mathcal{P} \defeq \{r^{1/2}P_1, r^{1/2}P_2, r^{1/2}P_3\}$ we get
        \begin{align*}
            s(r^{1/2}\mathcal{P}) &= 1 - \frac{\sqrt{2}}{3}\left(\sum_{i=1}^3\nu_i(\left\{0,r^{1/2}\tfrac{1}{k}a_1,r^{1/2}\tfrac{1}{k}a_2\right\})\right)\\
            &=1 - \frac{\sqrt{2}}{3}r^{1/2}\left(\sum_{i=1}^3\nu_i(\left\{0,\tfrac{1}{k}a_1,\tfrac{1}{k}a_2\right\})\right) = 1 - \frac{\sqrt{2}}{3}r^{1/2}\alpha_{\mathcal{P}},    
        \end{align*}
        where we define 
        \[ \alpha_{\mathcal{P}} \defeq \sum_{i=1}^3\nu_i(\left\{0,\tfrac{1}{k}a_1,\tfrac{1}{k}a_2\right\}). \]
        Similarly, the definitions of $Q_1,Q_2$ and $Q_3$ are linear in the basis $\{a_1,a_2\}$, thus, for all $i\in\{1,2,3\}$ we have
        \begin{align*}
            s(r^{1/2}\mathcal{P}\cup \{r^{1/2}Q_i\}) &= \max\left\{s(r^{1/2}\mathcal{P}) - \frac{\sqrt{2}}{3}r^{1/2}\left(\langle v_{j_1},-Q_1\rangle - \nu_{j_1}(\mathcal{P})\right),0\right\}\\
            &= \max\left\{1-\frac{\sqrt{2}}{3}r^{1/2}(\alpha_{\mathcal{P}} + \alpha_i),0\right\},
        \end{align*}
        where we define 
        \[ \alpha_i \defeq \langle v_{j_i},-Q_i\rangle - \nu_{j_i}(\mathcal{P}). \]
        Finally, we define 
        \[ \alpha_{\mathcal{Q}} \defeq \sum_{i=1}^3\nu_i(\mathcal{Q}), \]
        so that
        \[ s(r^{1/2}\mathcal{P}\cup r^{1/2}\mathcal{Q}) = \max\left\{1 - \frac{\sqrt{2}}{3}r^{1/2}\alpha_{\mathcal{Q}}, 0\right\}, \]
        where $r^{1/2}\mathcal{Q} \defeq \{r^{1/2}Q_1, r^{1/2}Q_2, r^{1/2}Q_3\}$.
        Observe that
        \begin{align}\label{eq:imp-identity}
            \alpha_1 + \alpha_2 + \alpha_3 = \alpha_{\mathcal{Q}} - \alpha_{\mathcal{P}}.
        \end{align}

        The above allows to give lower and upper bounds on the parameter $r$, determining whether the summands on the right hand side in \eqref{eq:incl-excl} appear or not. Using \Cref{lem:intersec-tri}, we have 
        \begin{align*}
            \begin{split}
                \eqref{eq:incl-excl}&= \frac{3\sqrt{3}}{2}\Bigg(\int_{(\frac{\sqrt{2}}{3}\alpha_{\mathcal{Q}})^{-2}}^{(\frac{\sqrt{2}}{3}\alpha_{\mathcal{P}})^{-2}}(1-\tfrac{\sqrt{2}}{3}\alpha_{\mathcal{P}}r^{1/2})^2~\frac{\dd r}{r}\\
                &-\sum_{i\in\{1,2,3\}}\int_{(\tfrac{\sqrt{2}}{3}\alpha_{\mathcal{Q}})^{-2}}^{(\tfrac{\sqrt{2}}{3}(\alpha_{\mathcal{P}}+\alpha_i))^{-2}}(1-\tfrac{\sqrt{2}}{3}(\alpha_{\mathcal{P}}+\alpha_i)r^{1/2})^2~\frac{\dd r}{r}\\
                &+ \sum_{i\neq j\in\{1,2,3\}}\int_{(\tfrac{\sqrt{2}}{3}\alpha_{\mathcal{Q}})^{-2}}^{(\tfrac{\sqrt{2}}{3}(\alpha_{\mathcal{P}}+\alpha_i+\alpha_j))^{-2}}(1-\tfrac{\sqrt{2}}{3}(\alpha_{\mathcal{P}}+\alpha_i+\alpha_j)r^{1/2})^2~\frac{\dd r}{r}\Bigg). 
            \end{split}
        \end{align*}
        The last summand corresponding to $\Leb(\tri(\mathcal{P}\cup\mathcal{Q}))$ is always equal to $0$, since we may assume that $s(\mathcal{P}\cup\mathcal{Q})=0$ as discussed above. Using the substitution $(1-\tfrac{\sqrt{2}}{3}\alpha r^{1/2}) \mapsto t$ for all $\alpha\in\{\alpha_{\mathcal{P}},\alpha_1,\alpha_2,\alpha_3\}$ in the integrals above, we obtain
        \begin{align*}
            \eqref{eq:incl-excl}=& 3\sqrt{3}\Bigg(\int_{0}^{1-\frac{\alpha_{\mathcal{P}}}{\alpha_{\mathcal{Q}}}}t^2 \frac{\dd t}{1-t}
            -\sum_{i\in\{1,2,3\}}\int_{0}^{1-\frac{\alpha_{\mathcal{P}}+\alpha_i}{\alpha_{\mathcal{Q}}}}t^2\frac{\dd t}{1-t} + \sum_{i\neq j\in\{1,2,3\}}\int_{0}^{1-\frac{\alpha_{\mathcal{P}}+\alpha_i+\alpha_j}{\alpha_{\mathcal{Q}}}}t^2\frac{\dd t}{1-t}\Bigg).
        \end{align*}
        For $\abs{t}<1$ the (indefinite) integrals above are given by $-\frac{t(t+2)}{2} - \ln(1-t)$. By \eqref{eq:imp-identity}, we have
        \[ 1- \frac{\alpha_{\mathcal{P}}}{\alpha_\mathcal{Q}} = \frac{\alpha_1}{\alpha_\mathcal{Q}} + \frac{\alpha_2}{\alpha_\mathcal{Q}} + \frac{\alpha_3}{\alpha_\mathcal{Q}} \]
        and using this identity a straightforward calculation yields that the summands of the (definite) integrals corresponding to the terms $-\frac{t(t+2)}{2}$ cancel. Hence, we obtain
        \begin{align*}
            \eqref{eq:incl-excl}=&-3\sqrt{3}\Bigg(\ln\left(\tfrac{\alpha_{\mathcal{P}}}{\alpha_{\mathcal{Q}}}\right) - \sum_{i\in\{1,2,3\}}\ln\left(\tfrac{\alpha_{\mathcal{P}}+\alpha_i}{\alpha_{\mathcal{Q}}}\right) + \sum_{i\neq j\in\{1,2,3\}}\ln\left(\tfrac{\alpha_{\mathcal{P}}+\alpha_i+\alpha_j}{\alpha_{\mathcal{Q}}}\right)\Bigg)\\
            &= -3\sqrt{3}\ln\left(\frac{\alpha_{\mathcal{P}}(\alpha_{\mathcal{P}}+\alpha_1+\alpha_2)(\alpha_{\mathcal{P}}+\alpha_2+\alpha_3)(\alpha_{\mathcal{P}}+\alpha_3+\alpha_1)}{\alpha_{\mathcal{Q}}(\alpha_{\mathcal{P}}+\alpha_1)(\alpha_{\mathcal{P}}+\alpha_2)(\alpha_{\mathcal{P}}+\alpha_3)}\right).
        \end{align*}
        Substituting $\alpha_{\mathcal{Q}} = \alpha_{\mathcal{P}} +\alpha_1+\alpha_2+\alpha_3$ and setting $\beta_{i} \defeq \tfrac{\alpha_i}{\alpha_\mathcal{P}}$ this translates to
        \begin{align}\label{eq:final-estimate}
            \eqref{eq:incl-excl}=&-3\sqrt{3}\ln\left(\frac{(1+\beta_1+\beta_2)(1 +\beta_2+\beta_3)(1+\beta_3+\beta_1)}{(1+\beta_1+\beta_2+\beta_3)(1+\beta_1)(1+\beta_2)(1+\beta_3)}\right).
        \end{align}
        Finally, using the following identity
        \begin{align}\label{eq:identity-betas}
            \frac{(1+\beta_1+\beta_2)(1 +\beta_2+\beta_3)(1+\beta_3+\beta_1)}{(1+\beta_1+\beta_2+\beta_3)(1+\beta_1)(1+\beta_2)(1+\beta_3)} = 1 - \frac{\beta_1\beta_2\beta_3(2+\beta_1+\beta_2+\beta_3)}{(1+\beta_1+\beta_2+\beta_3)(1+\beta_1)(1+\beta_2)(1+\beta_3)} 
        \end{align}
        as well as the standard inequality $-\ln(1-t) \leq \frac{t}{1-t}$, we obtain
        \begin{align*}
            \int_{\RR_{>0}}\Leb\left(\overline{B}_{r^{1/2}a_1,r^{1/2}a_2}(c_1')\right)\frac{\dd r}{r} &\leq 3\sqrt{3}\frac{\beta_1\beta_2\beta_3(2+\beta_1+\beta_2+\beta_3)}{(1+\beta_1+\beta_2)(1+\beta_2+\beta_3)(1+\beta_3+\beta_1)}\\
            &\leq 3\sqrt{3}\frac{8}{(k-1)k'(k-k')}\left(2 + \frac{2}{k-1} + \frac{2}{k'} + \frac{2}{k''}\right)\\
            &\leq 3\sqrt{3}\cdot64\frac{1}{(k-1)k'(k-k')},
        \end{align*}
        where the second to last inequality holds by trivially bounding the denominator by $1$ and using \Cref{lem:bounds-on-betas} and the fact that $k'' = k-k'$ (see \Cref{lem:Farey-fractions}) to bound the numerator.

        If we now consider all elements $c_1'\in\{0,\dots,k\}$ with $\gcd(c_1',k)=1$, the parameter $k'$ in the successive Farey fractions $\frac{h'}{k'} \leq \frac{c_1'}{k}\leq \frac{h''}{k''}$ will run through all co-prime residue classes modulo $k$ (see \Cref{rem:all-coprime-classes}). Thus, using \eqref{eq:first-dec} as well as \eqref{eq:second-dec} we obtain for every reduced basis $\{r^{1/2}a_1,r^{1/2}a_2\}$ of $r^{1/2}A'\ZZ^2$ 
        \begin{align*}
             \int_{\RR_{>0}}\Leb(\fund_{\{r^{1/2}a_1,r^{1/2}a_2\}^{(k)}})\frac{\dd r}{r} &= \sum_{(c_1,c_2)\text{ eligible}}\int_{\RR_{>0}}\Leb\left(B_{\{r^{1/2}a_1,r^{1/2}a_2\}}(c_1',c_2')\right)\frac{\dd r}{r}\\
             &\leq \sum_{(a_1',a_2')\in\mathfrak{a}}\left(\sum_{\gcd(c_1',k)=1}\int_{\RR_{>0}}\Leb\left(\overline{B}_{r^{1/2}a_1',r^{1/2}a_2'}(c_1')\right)\frac{\dd r}{r}\right)\\
             &\leq \sum_{\{a_1',a_2'\}\in\mathfrak{a}}\left(3\sqrt{3}\cdot64 \sum_{k'=1}^{k-1} \frac{1}{(k-1)k'(k-k')}\right) \leq 3\cdot 3\sqrt{3}\cdot64\frac{H_{k-1}}{k(k-1)} ,
         \end{align*}
        where $\mathfrak{a} = \{ (a_1,a_2), (a_2,a_1), (-a_1,a_2-a_1)\}$ is a set consisting of three pairs and $H_{k-1} = \sum_{\ell=1}^{k-1} \tfrac{1}{\ell}$ is the standard $(k-1)$-th harmonic number.
        In total, summing over all possible reduced bases, of which there are $12$ in total, we get
        \begin{align*}
            \int_{\RR_{>0}}\Leb\left(\fund_{r^{1/2}A'}^{(k)}\right)\frac{\dd r}{r} = \sum_{\mathfrak{a}}\left(\int_{\RR_{>0}}\Leb(B_{r^{1/2}\mathfrak{a}})\frac{\dd r}{r}\right)\leq 12\cdot9\sqrt{3}\cdot64\frac{H_{k-1}}{k(k-1)}
        \end{align*}
        Lastly, note that $H_{k-1} \leq \ln(k) + 1$, so \Cref{prop:upper-bound-level-measure} follows.
    \end{proof}

     \begin{lemma}\label{lem:bounds-on-betas}
        Using the notations as above, we have
        \[ 0 < \beta_1 \leq \frac{2}{k'},\quad 0 < \beta_2 \leq \frac{2}{k''},\quad 0 < \beta_3 \leq \frac{2}{k-1}. \]
    \end{lemma}
    \begin{proof}
        Recall that $\beta_i = \frac{\alpha_i}{\alpha_{\mathcal{P}}}$. By definition $\alpha_i>0$ and $\alpha_{\mathcal{P}}>0$, so the lower bounds are immediate. To obtain the upper bounds, we first give a lower bound on $\alpha_{\mathcal{P}}$. We have
        \[ \alpha_{\mathcal{P}} = \sum_{i=1}^3\nu_i(\left\{0,\tfrac{1}{k}a_1,\tfrac{1}{k}a_2\right\}) = \frac{1}{k}\left(\sum_{i=1}^3\max_{P\in\left\{0,a_1,a_2\right\}}\langle v_i, -P \rangle\right). \]
        Let $j\in\{1,2\}$. Since $\langle v_1 + v_2 + v_3, -a_j\rangle = 0$ and $a_j\neq 0$, there is at least one index $i_j\in\{1,2,3\}$ for which $\langle v_{i_j}, -a_j\rangle > 0$. Thus, it follows that
        \begin{align}\label{eq:bound-denominator}
             \alpha_{\mathcal{P}} \geq \frac{1}{k}\max_{i\in\{1,2,3\}}\left\{\max_{j\in\{1,2\}}\{\langle v_i,-a_j\rangle\}\right\} \geq \frac{1}{2k}\max_{i\in\{1,2,3\}}\left\{\max_{j\in\{1,2\}}\{\abs{\langle v_i,a_j\rangle}\}\right\},
        \end{align}
        where the second inequality follows from \Cref{eq:basics-on-vectors}.

        It remains to give an upper bound on $\alpha_i$. By \eqref{eq:bounds-on-alpha_i} we have
        \begin{align}\label{eq:bound-numerator}
            \begin{split}
                \alpha_1 & \leq  \frac{1}{kk'}\langle v_{j_1},-a_1\rangle \leq \frac{1}{kk'}\max_{i\in\{1,2,3\}}\left\{\max_{j\in\{1,2\}}\{\abs{\langle v_i,a_j\rangle}\}\right\},\\
                \alpha_2 & \leq  \frac{1}{kk''}\langle v_{j_2},a_1\rangle \leq \frac{1}{kk''}\max_{i\in\{1,2,3\}}\left\{\max_{j\in\{1,2\}}\{\abs{\langle v_i,a_j\rangle}\}\right\},\\
                \alpha_3 & \leq  \frac{1}{k(k-1)}\langle v_{j_3},a_2\rangle \leq \frac{1}{k(k-1)}\max_{i\in\{1,2,3\}}\left\{\max_{j\in\{1,2\}}\{\abs{\langle v_i,a_j\rangle}\}\right\}.
            \end{split}
        \end{align}
        Combining the bounds in \eqref{eq:bound-denominator} and \eqref{eq:bound-numerator} give the upper bounds as claimed.
    \end{proof}

\subsection{Proof of \Cref{prop:lower-bound-level-measure}}
    \begin{proof}
        To obtain a lower bound, we will compute a lower bound for $\Leb(\overline{B}_{a_1,a_2}(c_1'))$ for an explicit reduced basis $\{a_1,a_2\}$ and $c_1'\in\{0,\dots,k-1\}$. 
        
        For $A'\ZZ^2\in X_2$ we denote the reduced minima by $\rho_i \defeq \rho_i(A'\ZZ^2)$, for $i\in\{1,2,3\}$, as well as the first and second successive minima of $A'\ZZ^2$ with respect to $\tri$ by $\succ_j \defeq \succ_j(\tri,A'\ZZ^2)$ for $j\in\{1,2\}$.
        We define the set
        \[ Z_k \defeq \left\{A'\ZZ^2\in X_2~\Big{|}~ \exists i\neq e\in\{1,2,3\}:\begin{aligned}
        \langle v_i,-\rho_i\rangle =\succ_1, ~&~\langle v_i,\rho_e\rangle \geq \frac{1}{2}\succ_1\\
         \min\left\{\frac{1}{2}\succ_1,\frac{1}{2k}\succ_2\right\} &> \langle v_e,\rho_i\rangle,
        \end{aligned}\right \} \subset X_2. \]
        Notice that $Z_k$ depends on the level $k$. 
        For any $A'\ZZ^2\in Z_k$  we let $j\in\{1,2,3\}\setminus\{i,e\}$. The basis $a_1 = -\rho_i$ and $a_2 = \rho_j$ is reduced and we may assume that $\fund_{A'}$ is defined as the parallelogram spanned by $a_1$ and $a_2$. For $c_1'=1$ we will obtain a lower bound of the Lebesgue measure of $\overline{B}_{a_1,a_2}(c_1')$. We have
        \begin{align}\label{eq:concrete-set}
            \overline{B}_{a_1,a_2}(c_1') \subseteq \tri(\mathcal{P}) \setminus\left(\bigcup_{Q\in\mathcal{Q}}\tri(\mathcal{P}\cup\{ Q\})\right)
        \end{align}
        where $\mathcal{P} = \{\tfrac{1}{k}P'_1,\tfrac{1}{k}P'_2,\tfrac{1}{k}P'_3\}$ with
        \[ P'_1 = a_1 -a_2=-\rho_i-\rho_j=\rho_e,\quad P'_2 =2a_1-a_2=-\rho_i+\rho_e,\quad P'_3 = a_1=-\rho_i \]
        and $\mathcal{Q}=\{Q_1,Q_2,Q_3\}$ for $Q_1,Q_2,Q_3\in \bigcup_{\ell=1}^{k-1}\tfrac{1}{\ell}A'\ZZ^2$ defined as in \eqref{eq:def-of-Q_i}.
        More precisely, by the choice of $c_1'=1$ the successive Farey fractions $\tfrac{h'}{k'}< \tfrac{1}{k}<\tfrac{h''}{k''}$ are given by $h' = 0$, $k' = 1$ and $h'' = 1$, $k''=k-1$ and we get
        \begin{align*}
            \begin{alignedat}{1}
            Q_1 &=\frac{0}{1}a_1=0,\\
            Q_2&= \frac{1}{k-1}a_1 = -\frac{1}{k-1}\rho_i,\\
            Q_3 &= \frac{1}{k-1}a_1 - \frac{1}{k-1}a_2=\frac{1}{k-1}\rho_e.
            \end{alignedat}
        \end{align*}        
        We claim that in this case \eqref{eq:concrete-set} turns out to be an equality. Indeed, if $\overline{B}_{a_1,a_2}(c_1')\neq\emptyset$, then the right hand side in \eqref{eq:concrete-set} is non-empty as well and so by \Cref{lem:directions-of-reduced-elements} there are unique and pairwise different indices $j_1, j_2, j_3\in\{1,2,3\}$ so that for all $i\in\{1,2,3\}$ and $i'\neq j_i\in\{1,2,3\}$ we have
        \[ \nu_{j_i}(\mathcal{P}\cup\mathcal{Q}) =\langle v_{j_i},-Q_i\rangle > \nu_{j_i}(\mathcal{P}) \quad\text{and}\quad \nu_{i'}(\mathcal{P}\cup\mathcal{Q}) = \nu_{i'}(\mathcal{P}). \]
        In fact, it is easy to verify that in this case these indices are given by $j_1 = i$, $j_2 = j$, $j_3 = e$.
        Assume now that there is $b\in \tri(\mathcal{P}) \setminus\left(\bigcup_{Q\in\mathcal{Q}}\tri(\mathcal{P}\cup\{Q\})\right)$ such that $b\not\in \overline{B}_{a_1,a_2}(c_1')$.
        By \Cref{lem:level-equivalence}, there exists $Q\in \bigcup_{\ell=1}^{k-1}\tfrac{1}{\ell}A\ZZ^2$, such that $b\in -Q +\tri$. Let $\mathcal{Q}=\{Q_1.Q_2,Q_3\}$, and by assumption 
        \[ b\not\in (-Q_1+\tri)\cup(-Q_2+\tri)\cup(-Q_3+\tri), \]
        so that $Q\not\in\mathcal{Q}$. Moreover, by the description of $\tri$ given in \eqref{eq:desc-tri} and the uniqueness of $j_i$, we get
        \[ \langle v_{j_i}, b+Q \rangle > \frac{-1}{\sqrt{2}} \geq \langle v_{j_i}, b+Q_i \rangle  \]
        and so $\langle v_{j_i}, Q-Q_i\rangle > 0$ for all $i\in\{1,2,3\}$.
        We write
        \[ Q = \frac{d_1}{\ell}a_1 + \frac{d_2}{\ell}a_2 \]
        for some $d_1,d_2\in\NN$ and $1\leq \ell\leq k-1$. Using the definitions of $a_1 = -\rho_i$, $a_2=\rho_j$, $j_1$, $j_2$, $j_3$, $Q_1,Q_2$ and $Q_3$ we obtain
        \begin{align}
            \langle v_{j_1}, Q-Q_1\rangle > 0 &&\Leftrightarrow  &&\frac{d_1}{\ell}\langle v_i, -\rho_i \rangle &> -\frac{d_2}{\ell}\langle v_i, \rho_j \rangle,\label{eq:Q-Q_1}\\
            \langle v_{j_2}, Q-Q_2\rangle > 0 &&\Leftrightarrow && -\left(\frac{d_1}{\ell}-\frac{1}{k-1}\right)\langle v_j,\rho_i \rangle &> \frac{d_2}{\ell}\langle v_j, -\rho_j \rangle,\label{eq:Q-Q_2}\\
            \langle v_{j_3}, Q-Q_3\rangle > 0 &&\Leftrightarrow && -\left(\frac{d_1}{\ell}-\frac{1}{k-1}\right)\langle v_e, \rho_i \rangle &> -\left(\frac{d_2}{\ell}+\frac{1}{k-1}\right)\langle v_e, \rho_j \rangle,\label{eq:Q-Q_3}
        \end{align}
        where we arranged the terms in such a way that all inner products are non-negative. Notice that since $\ell \leq k-1$, the terms $\frac{d_1}{\ell}$,$(\frac{d_1}{\ell}-\tfrac{1}{k-1})$ and the terms $\frac{d_2}{\ell}$,$(\frac{d_2}{\ell}+\tfrac{1}{k-1})$ cannot have opposite signs. If $d_2\leq0$, then \eqref{eq:Q-Q_1} implies $d_1>0$. However, then \eqref{eq:Q-Q_3} yields a contradiction as the right hand side is non-negative, whereas the left hand side is negative. We conclude that $d_2>0$ and using \eqref{eq:Q-Q_2} this implies $d_1\leq0$. Since $\langle v_j,-\rho_j\rangle \geq \langle v_j,\rho_i\rangle$, \eqref{eq:Q-Q_2} gives
        \begin{align}\label{eq:cond-on-d_2}
            -\left(\frac{d_1}{\ell}-\frac{1}{k-1}\right) > \frac{d_2}{\ell}.
        \end{align}
        Using one of the inequalities defining the set $Z_k$, we have
        \[ \langle v_i,\rho_j\rangle = \langle v_i, -\rho_i\rangle - \langle v_i,\rho_e\rangle \leq \succ_1 - \frac{1}{2}\succ_1 = \frac{1}{2}\succ_1. \]
        Hence, \eqref{eq:Q-Q_1} gives $d_2 > -2d_1$. Combining this with \eqref{eq:cond-on-d_2} we obtain $\frac{d_1}{\ell} > -\frac{1}{k-1}$, which is a contradiction as $d_1<0$ and $\ell\leq k-1$. Thus, such a $Q$ cannot exist and  \eqref{eq:concrete-set} is indeed an equality. 
        
        In particular, for these choices of $a_1$, $a_2$, $c_1$ and $c_2$ following the same
        calculation leading to \eqref{eq:final-estimate}, we obtain, using \eqref{eq:identity-betas} as well as the standard inequality $-\ln(1-t)\geq t$, that
        \begin{align}\label{eq:concrete-lower-bound}
            \begin{split}
                \int_{\RR_{>0}}\Leb\left(B_{r^{1/2}a_1,r^{1/2}a_2}(c_1,c_2)\right)\frac{\dd r}{r} &\geq 3\sqrt{3}\frac{\beta_1\beta_2\beta_3(2+\beta_1+\beta_2+\beta_3)}{(1+\beta_1+\beta_2+\beta_3)(1+\beta_1)(1+\beta_2)(1+\beta_3)}\\
                &\geq 3\sqrt{3}\frac{\beta_1}{1+\beta_1}\frac{\beta_2}{1+\beta_2}\frac{\beta_3}{1+\beta_3}.
            \end{split}
        \end{align}
        Recall that $\beta_{i'} = \frac{\alpha_{i'}}{\alpha_{\mathcal{P}}}$, where $\alpha_{\mathcal{P}} = \sum_{i'=1}^3\nu_{i'}(\mathcal{P})$ and $\alpha_{i'} = \nu_{j_{i'}}(\mathcal{Q})- \nu_{j_{i'}}(\mathcal{P}) = \langle v_{j_{i'}}, -Q_{i'}\rangle - \nu_{j_{i'}}(\mathcal{P})$ for all $i'\in\{1,2,3\}$. On the one hand, we have the lower bounds     
        \begin{align}
            \begin{split}\label{eq:concrete-nu_i}
                \alpha_1 = \nu_{j_1}(\mathcal{Q})-\nu_{j_1}(\mathcal{P}) &=\frac{1}{k}\langle v_i, \rho_e \rangle\geq \frac{1}{2k}\succ_1,\\
                \alpha_2 =\nu_{j_2}(\mathcal{Q})-\nu_{j_2}(\mathcal{P}) & = \frac{1}{k(k-1)}\langle v_j, \rho_i\rangle = \frac{1}{2k(k-1)}\succ_1,\\
                \alpha_3 =\nu_{j_3}(\mathcal{Q})-\nu_{j_3}(\mathcal{P}) &= \frac{1}{k(k-1)}\langle v_e,-\rho_e\rangle  -\frac{1}{k-1}\langle v_e, \rho_i \rangle \geq \frac{1}{2k(k-1)}\succ_2,
            \end{split}
        \end{align}
        where the inequalities follow by definition of the set $Z_k$.
        On the other hand, using \eqref{eq:dir-identity}, we have the upper bound
        \[ \alpha_{\mathcal{P}} = \frac{1}{k}\left(\langle v_i,-\rho_e\rangle+\langle v_j,\rho_i\rangle+\langle v_e,\rho_i-\rho_e\rangle\right) = \frac{1}{k}(\langle v_j,\rho_e\rangle+\langle v_i,-\rho_i\rangle)\leq \frac{2}{k}\succ_2. \]
        This, together with the lower bounds in \eqref{eq:concrete-nu_i}, yields
        \begin{align*}
            \beta_{j_1} \geq \frac{\succ_1}{4\succ_2},\quad \beta_{j_2} \geq \frac{\succ_1}{2(k-1)\succ_2},\quad \beta_{j_3} \geq \frac{1}{2(k-1)}.
        \end{align*}
        To sum up, we get
        \begin{align*}
            \frac{\beta_{j_1}}{1+\beta_{j_1}} \geq \frac{\succ_1}{4\succ_2 + \succ_1} \geq \frac{\succ_1}{5\succ_2},\quad
            \frac{\beta_{j_2}}{1+\beta_{j_2}} \geq \frac{\succ_1}{2(k-1)\succ_2 + \succ_1} \geq \frac{\succ_1}{(2k-1)\succ_2},\quad
            \frac{\beta_{j_3}}{1+\beta_{j_3}} &\geq \frac{1}{2k-1}.
        \end{align*}
        Thus, by \eqref{eq:concrete-lower-bound} we have for every $A'\ZZ^2\in Z_k$ that
        \[ \int_{\RR_{>0}}\Leb\left(\overline{B}_{r^{1/2}a_1,r^{1/2}a_2}(c_1')\right)\frac{\dd r}{r} \geq \frac{3\sqrt{3}}{5(2k-1)^2}\left(\frac{\succ_1}{\succ_2}\right)^2. \]

        What is left to show is that
        \[ \int_{Z_k}\left(\frac{\succ_1}{\succ_2}\right)^2\dd\msr_{X_2} \gg \frac{1}{k}. \]
        Let $A'\ZZ^2\in X_2$ and assume that $\langle v_i,-\rho_i\rangle = \succ_1$. Let $\gamma$ denote the (absolute) angle between the half-line $\{x\in\cone_i\st \langle v_e,x\rangle = 0\}$ and the point $\rho_i\in\cone_i$ , where $e$ is chosen so that this angle is minimized, so we have $\gamma \leq 30^{\circ}$. Then
        \[ \langle v_e,\rho_i\rangle = \frac{2}{\sqrt{3}}\sin(\gamma)\langle v_i,-\rho_i\rangle = \frac{2}{\sqrt{3}}\sin(\gamma)\succ_1\leq \frac{2}{\sqrt{3}}\sin(\gamma)\succ_2. \]
        It follows that if $\sin(\gamma)\leq \frac{\sqrt{3}}{4k}$, then $\langle v_e, \rho_i\rangle < \min\{\frac{1}{2}\succ_1,\frac{1}{2k}\succ_2\}$. If additionally $\langle v_i,\rho_e\rangle \geq \frac{1}{2}\succ_1$, then $A'\ZZ^2\in Z_k$. However, since $\succ_1 = \langle v_i,-\rho_i\rangle = \langle v_i,\rho_e\rangle + \langle v_i,\rho_j\rangle$, this condition eliminates roughly half of the lattices satisfying $\sin(\gamma)\leq \frac{\sqrt{3}}{4k}$. It is clear that $\left(\frac{\succ_1}{\succ_2}\right)^2$ is integrable in the space of all unimodular lattices $X_2/\SO_2(\RR)$ up to rotation, and we set
        \[ 0 < c'\defeq \int_{X_2/\SO_2(\RR)}\left(\frac{\succ_1}{\succ_2}\right)^2\dd\msr_{X_2/\SO_2(\RR)} < \infty. \] 
        Since $\Leb(\{\gamma\in[0,2\pi)\st \sin(\gamma) \leq \frac{\sqrt{3}}{4k}\}) \geq \frac{\sqrt{3}}{4k}$ we get
        $\int_{Z_k}\left(\frac{\succ_1}{\succ_2}\right)^2\dd\msr_{X_2} \gg \frac{1}{k}$
        where the implied constant is independent of $k$. This finishes the proof of the lower bound.
    \end{proof}

	\printbibliography[title=References]
\end{document}